\documentclass[reqno]{amsart}

\usepackage{verbatim}
\usepackage[textsize=scriptsize]{todonotes}
\usepackage{tikz-cd}
\usetikzlibrary{arrows}
\usepackage{etoolbox}
\usepackage{etex}
\usepackage{hyperref}
\usepackage{cleveref}
\usepackage[T1]{fontenc}

\usepackage{chemarr}
\usepackage{amssymb}
\usepackage{amsmath}
\usepackage{comment}
\usepackage{spectralsequences}
\usepackage{mathtools}
\usepackage{rotating}
\usepackage{wrapfig}
\usepackage{outlines}
\usepackage{graphicx}

\usepackage{aliascnt}

\let\oldtheorem\newtheorem
\RenewDocumentCommand{\newtheorem}{s m o m O{}}{%
\IfBooleanTF{#1}%
{\oldtheorem{#2}{#4}}%
{\IfNoValueTF{#3}{\oldtheorem{#2}{#4}[#5]}%
{\newaliascnt{#2}{#3}%
\oldtheorem{#2}[#2]{#4}%
\aliascntresetthe{#2}}}}

\theoremstyle{definition}
\newtheorem{nul}{}[section]
\newtheorem{dfn}[nul]{Definition}

\newtheorem{rmk}[nul]{Remark}

\newtheorem{cnstr}[nul]{Construction}
\newtheorem{cnv}[nul]{Convention}
\newtheorem{ntn}[nul]{Notation}
\newtheorem{exm}[nul]{Example}

\newtheorem{rec}[nul]{Recollection}

\newtheorem*{dfn*}{Definition}
\newtheorem*{axm*}{Axiom}
\newtheorem*{ntn*}{Notation}
\newtheorem*{exm*}{Example}
\newtheorem*{exr*}{Exercise}
\newtheorem*{int*}{Intuition}
\newtheorem*{qst*}{Question}
\newtheorem*{rmk*}{Remark}

\newtheorem{definition}[nul]{Definition}

\newtheorem{remark}[nul]{Remark}

\theoremstyle{plain}

\newtheorem{thm}[nul]{Theorem}
\newtheorem{prop}[nul]{Proposition}
\newtheorem{lem}[nul]{Lemma}

\newtheorem{cor}[nul]{Corollary}

\newtheorem{theorem}[nul]{Theorem}

\newtheorem{lemma}[nul]{Lemma}

\newtheorem{proposition}[nul]{Proposition}
\newtheorem*{thm*}{Theorem}
\newtheorem*{prop*}{Proposition}
\newtheorem*{cor*}{Corollary}
\newtheorem*{lem*}{Lemma}
\newtheorem*{cnj*}{Conjecture}

\let\oldwidetilde\widetilde
\protected\def\widetilde{\oldwidetilde}

\DeclareMathOperator{\im}{\mathrm{im}}

\DeclareMathOperator{\fib}{\mathrm{fib}}

\DeclareMathOperator{\coker}{\mathrm{coker}}

\DeclareMathOperator{\F}{\mathbb{F}}

\DeclareMathOperator{\E}{\mathbb{E}}

\DeclareMathOperator{\MO}{\mathrm{MO}}

\DeclareMathOperator{\Sp}{\mathrm{Sp}}

\DeclareMathOperator{\Sq}{\mathrm{Sq}}

\newcommand{\wt}{\widetilde}
\newcommand{\BOn}{\mathrm{BO \langle n \rangle}}

\newcommand{\BO}{\mathrm{BO}}

\newcommand{\BSO}{\mathrm{BSO}}

\newcommand{\MStr}{\mathrm{MString}}

\newcommand{\MSpin}{\mathrm{MSpin}}
\newcommand{\MString}{\mathrm{MString}}

\newcommand{\tmf}{\mathrm{tmf}}

\newcommand{\Z}{\mathbb{Z}}

\newcommand{\mo}{\mathrm{MO}}
\newcommand{\littleo}{\mathrm{o}\langle n-1 \rangle}
\newcommand{\R}{\Sigma^{\infty}\mathrm{O}\langle n-1 \rangle}
\def\H{\mathrm{H}}

\def\bgl{\mathrm{bgl}}
\def\Einf{\mathbb{E}_\infty}
\def\bo{\mathrm{bo}}
\def\ko{\mathrm{ko}}
\def\bP{\mathrm{bP}}
\def\MOn{\mathrm{MO} \langle n \rangle}
\def\Th{\mathrm{Th}}
\def\Ss{\mathbb{S}}
\def\OP{\mathbb{OP}}
\def\HP{\mathbb{HP}}
\def\RR{\mathbb{R}}

\DeclarePairedDelimiter\abs{\lvert}{\rvert}%
\makeatletter
\let\oldabs\abs
\def\abs{\@ifstar{\oldabs}{\oldabs*}}

\NewDocumentCommand{\tens}{e{_^}}{%
  \mathbin{\mathop{\otimes}\displaylimits
    \IfValueT{#1}{_{#1}}
    \IfValueT{#2}{^{#2}}
  }%
}

\usepackage{tikz}
\usetikzlibrary{matrix,arrows,decorations}
\usepackage{tikz-cd}

\usepackage{adjustbox}

\let\oldtocsection=\tocsection
 
\let\oldtocsubsection=\tocsubsection
 
\let\oldtocsubsubsection=\tocsubsubsection
 
\renewcommand{\tocsection}[2]{\hspace{0em}\oldtocsection{#1}{#2}}
\renewcommand{\tocsubsection}[2]{\hspace{1em}\oldtocsubsection{#1}{#2}}
\renewcommand{\tocsubsubsection}[2]{\hspace{2em}\oldtocsubsubsection{#1}{#2}}

\newcommand\restr[2]{{
  \left.\kern-\nulldelimiterspace 
  #1 
  \vphantom{\big|} 
  \right|_{#2} 
  }}
\usepackage{todonotes}

\usepackage{etoolbox}
\newtoggle{draft}
\togglefalse{draft}

\iftoggle{draft} {
\usepackage[margin=1.25in]{geometry}
\newcommand{\NB}[1]{\todo[color=gray!40]{#1}}
\newcommand{\TODO}[1]{\todo[color=red]{#1}}
\newcommand{\style}[1]{\todo[color={rgb,256:red,141;green,211;blue,239}]{#1}}
}{ 
\usepackage[margin=1.25in]{geometry}
\newcommand{\NB}[1]{}
\newcommand{\TODO}[1]{}
\newcommand{\style}[1]{}
\renewcommand{\todo}[1]{}
\renewcommand{\todo}[1]{}
}
\title{Inertia groups of $(n-1)$-connected $2n$-manifolds}

\author{Andrew Senger}
\address{Department of Mathematics, University of Maryland, College Park, MD, USA}
\email{senger@umd.edu}

\author{Adela YiYu Zhang}
\address{Centre for Geometry and Topology, University of Copenhagen, Denmark}
\email{yz@math.ku.dk}

\begin{document}
\begin{abstract}
  In this paper, we compute the inertia groups of $(n-1)$-connected, smooth, closed, oriented $2n$-manifolds where $n \geq 3$.
  As a consequence, we complete the diffeomorphism classification of such manifolds, finishing a program initiated by Wall sixty years ago, with the exception of the $126$-dimensional case of the Kervaire invariant one problem.

  In particular, we find that the inertia group always vanishes for $n \neq 4,8,9$---for $n \gg 0$, this was known by the work of several previous authors, including Wall, Stolz, and Burklund and Hahn with the first named author. When $n = 4,8,9$, we apply Kreck's modified surgery and a special case of Crowley's $Q$-form conjecture, proven by Nagy, to compute the inertia groups of these manifolds. In the cases $n=4,8$, our results recover unpublished work of Crowley--Nagy and Crowley--Olbermann.

  In contrast, we show that the homotopy and concordance inertia groups of $(n-1)$-connected, smooth, closed, oriented $2n$-manifolds with $n \geq 3$ always vanish.
%
\end{abstract}
\maketitle

\setcounter{tocdepth}{1}
\tableofcontents
\vbadness 5000


\section{Introduction} \label{sec:intro}
\todo{Crowley's paper and Wall}
The goal of this paper is to complete the classification of $(n-1)$-connected, smooth, closed, oriented $2n$-manifolds when $n \geq 3$, following the program initiated by Wall \cite{Wall62}. We are able to obtain such a classification for $n \neq 63$, where the $126$-dimensional Kervaire invariant one problem remains unsolved.

\begin{cnv}
  In the remainder of this paper, we assume that all manifolds are smooth and oriented.
\end{cnv}
To classify such manifolds, Wall associated to each $(n-1)$-connected, closed $2n$-manifold a collection of algebraic invariants that he called an \emph{$n$-space} and proved that such manifolds are determined by their $n$-spaces up to connected sum with an exotic sphere.
Building on the work of many previous authors \cite{Wall62,KervaireMilnor,BP66, Wall67, Kosinski, MT67, Bro69, Sch72, Lam, BJM84,stolz,Sto87,HHR,bhs}, Burklund and the first named author gave a complete answer\footnote{Outside of the case $n=63$, where the answer depends on the final unknown case of the Kervaire invariant one problem.} to when an $n$-space is realized by an $(n-1)$-connected $2n$-manifold \cite{geography}.

To finish the classification, it therefore suffices to answer the following question: given an $n$-space which is realized by an $(n-1)$-connected, closed $2n$-manifold $M$, for which homotopy spheres $\Sigma$ is there an orientation-preserving diffeomorphism between $M \# \Sigma$ and $M$? In other words, what is the \emph{inertia group} $I(M) \subset \Theta_{2n}$ of $M$, in the sense of the following definition?

\begin{dfn}
  Let $M$ denote a smooth, closed, oriented $m$-manifold.
  The \emph{inertia group} of $M$ is the subgroup $I(M) \subset \Theta_m$ of exotic spheres $\Sigma$ such that $M \# \Sigma$ is diffeomorphic to $M$ via an orientation-preserving diffeomorphism.
\end{dfn}

The main result of this paper is the following computation of the inertia groups of $(n-1)$-connected, closed $2n$-manifolds for $n \geq 3$.

\begin{thm}\label{determination}
  Let $M$ denote an $(n-1)$-connected, smooth, closed, oriented $2n$-manifold with $n\geq 3$. 
\begin{enumerate}

  \item The inertia group $I(M)$ of $M$ is equal to zero when $n \neq 4,8,9$.

  \item In the remaining cases, the inertia group is determined as follows:
  \begin{itemize}
    \item (Crowley--Nagy \cite{crowleynagy}) A $3$-connected $8$-manifold $M_8$ has inertia group
      \[I(M_8) = \begin{cases} 0 &\text{ if } 8 \mid p_1 \in \H^4 (M_8);\\ \Theta_8 \cong \Z/2\Z &\text{ if } 8 \nmid p_1 \in \H^4 (M_8). \end{cases}\]
    \item (Crowley--Olbermann) A $7$-connected $16$-manifold $M_{16}$ has inertia group
      \[I(M_{16}) = \begin{cases} 0 &\text{ if } 24 \mid p_2 \in \H^8 (M_{16});\\ \Theta_{16} \cong \Z/2\Z &\text{ if } 24 \nmid p_2 \in \H^8 (M_{16}). \end{cases}\]
    \item An $8$-connected $18$-manifold $M_{18}$ has inertia group
      \[I(M_{18}) = \begin{cases} 0 &\text{ if } H(M_{18}) \textit{ is zero;}   \\  \Z/8\Z \cong \mathrm{bSpin}_{19} \subset \Theta_{18} &\text{ otherwise}, \end{cases}\]
        where $\mathrm{bSpin}_{19}$ is the subgroup of $\Theta_{18}$ consisting of those homotopy $18$-spheres which bound a spin $19$-manifold and $H(M_{18}) \subset \pi_{9} \mathrm{BSO}$ is the image of the map $\pi_{9}(M_{18})\rightarrow \pi_9(\mathrm{BSO})\cong \Z/2\Z$ induced by the map classifying the stable normal bundle of $M_{18}$.
  \end{itemize}
\end{enumerate}
\end{thm}
\begin{remark}
When $n=4$ and $8$, examples of $(n-1)$-connected $2n$-manifolds with nontrivial inertia groups are given by $\mathbb{H}\mathbb{P}^2$ and $\mathbb{O}\mathbb{P}^2$, as was first proven by Kramer and Stolz \cite{kramerstolz}.

  When $n=9$, an example is given by the linear $S^9$ bundle over $S^9$ classified by the nonzero element of $\pi_9 \mathrm{BSO}(10) \cong \pi_9 \mathrm{BSO} \cong \Z/2\Z$.
  Indeed, this follows from \Cref{determination} and the fact that the stable normal bundle of a linear sphere bundle $S^n \to M \to S^m$ classified by a map $S^m \to \mathrm{BSO} (n+1)$ is classified by the composite
  \[M \to S^m \to \mathrm{BSO} (n+1) \to \mathrm{BSO} \xrightarrow{-} \mathrm{BSO},\]
  where the final map takes a virtual vector bundle $V$ to its additive inverse $-V$ \cite[Fact 3.1]{CrowleyEscher}.
\end{remark}

\begin{rmk}
  The $n=4$ and $n=8$ cases of \Cref{determination} are the subject of unpublished work of Crowley--Nagy \cite{crowleynagy} and Crowley--Olbermann, respectively.
  In \Cref{sec:det}, we will present our own proofs of these cases, making crucial use of Nagy's proof \cite{nagy, nagyarxiv} of a special case of Crowley's $Q$-form conjecture \cite[Problem 11]{CrowleyProblems}. 
\end{rmk}

The problem of the computation of the inertia groups of these manifolds has a long history, and all but finitely many cases of \Cref{determination} were known before our work.
Indeed, let $M$ denote an $(n-1)$-connected, closed $2n$-manifold.
When $n \equiv 3,5,6,7 \mod 8$, Wall proved that $I(M)$ is trivial \cite{WallAction}.
When $M$ is stably frameable, Kosinski and Wall proved that $I(M) = 0$ \cite[\S 16]{Wall67} \cite{Kosinski}.\footnote{However, as we shall note below, \cite[Theorem 10]{Wall67} is incorrect and contradicts \Cref{determination}.}

In high enough dimensions, the inertia group of $(n-1)$-connected $2n$ manifolds is known to always vanish, whether or not the manifold is stably frameable. In the case $n\neq 1 \mod 8$, this is due to Stolz:
\begin{thm}[{\cite[Theorem D]{stolz}}]
Let $M$ be an $(n-1)$-connected $2n$-manifold. Then the inertia group of $M$ is trivial when $n \neq 1 \mod 8$ and $n \geq 106$ or $n\equiv 2$ mod 8 with $n>10$.
\end{thm}
Stolz's technique involves the analysis of vanishing lines in the $\F_2$-Adams spectral sequence and relies on unpublished work of Mahowald \cite{mahowald73} (see \cite[\S 15]{bhs} and \cite{chang}). In \cite{bhs}, Burklund, Hahn, and the first named author built on the work of Stolz to show that any $(n-1)$-connected $2n$-manifold  has trivial inertia group for $n\equiv 1$ mod $8$ if $n>232$.

The contribution of this paper is to resolve the remaining cases, i.e., when $n=10$, $n\equiv 0,4$ mod 8 with $n<106$, or $n\equiv 1$ mod 8 with $n\leq 232$.

  %

Using standard methods (cf.\ \cite[1.\ Schritt]{stolz}), the statement that for $n\geq 3$ the inertia groups of all $(n-1)$-connected closed $2n$-manifolds are trivial may be reduced to the following statement in homotopy theory: the kernel of the unit map
\[\pi_{2n} \Ss \to \pi_{2n} \MO\langle n \rangle\]
is generated by the image of $J$ and, when $n = 3,7$, the classes  $\nu^2$, $\sigma^2$ respectively.\footnote{The classes $\nu^2$ and $\sigma^2$ do not contribute to the inertia group because they have Kervaire invariant one, hence do not lift along the natural map $\Theta_{2n} \to \coker(J)_{2n}$.}
We prove the following theorem, including the cases $n =1,2$ for the sake of completeness:
\begin{thm}\label{thm:ker-main}
  Let $n \geq 1$.
  The kernel of the unit map
  \[ \pi_{2n} \Ss \to \pi_{2n} \MO \langle n \rangle\]
  is generated by the image of $J$ unless $n = 1,3,4,7,8,9$.
  \begin{itemize}
    \item When $n=1$, it is generated by the image of $J$ and $\eta^2 \in \pi_2 \Ss$.
    \item When $n=3$, it is generated by the image of $J$ and $\nu^2 \in \pi_6 \Ss$.
    \item When $n=4$, it is generated by the image of $J$ and $\varepsilon \in \pi_8 \Ss$.
    \item When $n=7$, it is generated by the image of $J$ and $\sigma^2 \in \pi_{14} \Ss$.
    \item When $n=8$, it is generated by the image of $J$ and $\eta_4 \in \pi_{16} \Ss$.
    \item When $n=9$, it is generated by the image of $J$ and $[h_2 h_4] \in \pi_{18} \Ss$.\footnote{The symbol $[h_2 h_4]$ does not unambiguously determine a class in $\pi_{18} \Ss$, even up to multiplication by a $2$-adic unit. To be precise, here we choose a representative of $[h_2 h_4]$ that lies in the kernel of the unit map $\pi_{18} \Ss \to \pi_{18} \ko \cong \F_2$. This kernel is isomorphic to $\Z/8\Z$, so this determines an element uniquely up to multiplication by a $2$-adic unit.}
  \end{itemize}
\end{thm}

The proof of \Cref{thm:ker-main} involves a combination of the general techniques of \cite{inertia}, which improve on those of \cite{bhs}, with ad hoc arguments similar to those employed in \cite{geography}.
In particular, we make heavy use of the category of $\F_2$-synthetic spectra introduced by Pstr\k{a}gowski \cite{Pstragowski} to provide lower bounds on the $\F_2$-Adams filtration of elements in the kernel of the unit map.
An additional technical ingredient is an improved $v_1$-banded vanishing line on the mod $2$ Moore spectrum due to Chang \cite{chang}, which helps us to conclude that a class in the stable homotopy groups of spheres must be in the image of $J$ if it is of high $\F_2$-Adams filtration.

When the kernel of the unit map is not generated by the image of $J$ and an element of Kervaire invariant one for some $n$, we require a method to compute what the inertia group of a given $(n-1)$-connected $2n$-manifold is.

For this purpose, we make use of  Kreck's modified surgery \cite{kreck} and a special case of Crowley's $Q$-form conjecture \cite[Problem 11]{CrowleyProblems} proven by Nagy \cite{nagy, nagyarxiv}, which again help us reduce the question to one about the stable homotopy groups of certain cobordism spectra that we are able to fully answer.

\begin{rmk}
  Theorem 10 of \cite{Wall67} implies that the inertia group of an $(n-1)$-connected $2n$-manifold only depends on the mod $2$ map $\H_n (M)/2 \to \pi_n \BSO/2$ induced by the stable normal bundle.
  This result contradicts our \Cref{determination}, as well as the earlier work of Crowley--Nagy and Crowley--Olbermann.
  Frank has pointed out an error in Wall's argument in (3) of \cite[\S 8]{frank74wall}.
\end{rmk}

Given a manifold $M$, one may further study certain subgroups of the inertia group: the \emph{homotopy inertia group} $I_h (M) \subset I(M)$ of homotopy spheres $\Sigma$ such that $M \# \Sigma \cong M$ via a diffeomorphism homotopic to the standard homeomorphism, as well as the \emph{concordance inertia group} $I_c (M) \subset I_h (M)$ of homotopy spheres $\Sigma$ such that $M \# \Sigma$ is concordant to $M$ \cite{munkres}. \todo{Check cites}

In contrast to the inertia groups, we show that the homotopy and concordance inertia groups of $(n-1)$-connected closed $2n$-manifolds always vanish when $n \geq 3$:

\begin{thm} \label{thm:h-c}
  Let $M$ denote an $(n-1)$-connected, smooth, closed, oriented $2n$-manifold with $n\geq 3$. Then the concordance and homotopy inertia groups of $M$ vanish: $I_c (M) = I_h (M) = 0$.
\end{thm}

This is a consequence of \Cref{determination} when $n \neq 4,8,9$. Some of the remaining cases were already proven in the literature: when $n=4$, Kasilingam showed that $I_c (M) = I_h (M) = 0$ for all $(n-1)$-connected $M$ \cite{ramesh}. When $n=8$, he further proved that $I_c (M) = 0$ under the additional hypothesis $H^8 (M) = \mathbb{Z}$.

Furthermore, \Cref{thm:ker-main} allows us to draw consequences for the classification of $(n-1)$-connected $(2n+1)$-manifolds.
In \cite{Wall67}, Wall determined the diffeomorphism classes of $(n-1)$-connected almost closed\footnote{That is, those with boundary a homotopy sphere.} $(2n+1)$-manifolds in terms of certain systems of invariants when $n \geq 3$ and $n \neq 3,7$.
When $n=3,7$, such a classification was given by Crowley \cite{CrowleyThesis}, who built on earlier work of Wilkens \cite{WilkinsThesis, WilkinsPaper}.
For the case $n=2$, see the work of Barden \cite{Barden}.

Using \Cref{thm:ker-main}, we are able to determine precisely which exotic $2n$-spheres are the boundaries of $(n-1)$-connected almost closed $(2n+1)$-manifolds when $n \geq 3$. As a consequence, we deduce that every $(n-1)$-connected almost closed $(2n+1)$-manifold may be filled in to obtain a closed manifold when $n \geq 3$ and $n \neq 4,8,9$.



\begin{thm} \label{thm:which-bdy}
  Suppose that $n \geq 3$ and $n \neq 4,8,9$. Then a $2n$-dimensional homotopy sphere is the boundary of an $(n-1)$-connected smooth $(2n+1)$-manifold if and only if it is diffeomorphic to the standard sphere. In particular, one may always glue a disk along the boundary to obtain an $(n-1)$-connected closed $(2n+1)$-manifold. In the remaining cases:
  \begin{itemize}
    \item Any homotopy $8$-sphere is the boundary of a $3$-connected $9$-manifold.
    \item Any homotopy $16$-sphere is the boundary of a $7$-connected $17$-manifold.
    \item A homotopy $18$-sphere $\Sigma$ is the boundary of an $8$-connected $19$-manifold if and only if it is the boundary of a spin $19$-manifold.
  \end{itemize}
\end{thm}

To deduce \Cref{thm:which-bdy} from \Cref{thm:ker-main}, one notes that a homotopy $2n$-sphere is the boundary of an $(n-1)$-connected $(2n+1)$-manifold precisely when it lies in the kernel of $\Theta_{2n} \to \coker(J)_{2n} \to \pi_{2n} \MO \langle n \rangle$ by the Pontryagin--Thom construction and surgery below the middle dimension.
It then suffices to recall that $\Theta_{2n} \to \coker(J)_{2n}$ is the inclusion of elements of Kervaire invariant $0$ and that $\varepsilon$, $\eta_4$, and $[h_2 h_4]$ are generators for $\coker(J)_8$, $\coker(J)_{16}$, and the kernel of $\pi_{18} \Ss \cong \coker(J)_{18} \to \pi_{18} \mathrm{MSpin}$, respectively \cite{Isaksen}.

\begin{rmk}
  The fact that any homotopy $8$-sphere is the boundary of a $3$-connected $9$-manifold and any homotopy $16$-sphere is the boundary of a $7$-connected $17$-manifold is not original to us.
  The $8$-dimensional case was originally proven by Frank \cite[Theorem 3]{Frank}, though the authors are not aware of a full account of Frank's proof that appears in print.
  Both cases are consequences of work of Bier and Ray \cite{BierRay} (cf.\ \cite[Satz 12.1(i)]{stolz}).
\end{rmk}

When $n =4,8,9$, we also determine the boundary map $\partial : A_{2n+1} ^{\langle n \rangle} \to \Theta_{2n}$ up to multiplication by a $2$-adic unit.
In particular, we determine precisely which $(n-1)$-connected almost closed $(2n+1)$-manifolds may be filled in to obtain closed manifolds.

\begin{thm} \label{thm:kernel-main-intro}
  \begin{enumerate}
    \item[(i)] (Frank \cite[Theorem 3]{Frank}) The boundary of a $3$-connected $9$-manifold $M$ is the standard sphere if and only if $\Psi_{-L_{\mathbb{H}}} (M) \in \{1, [\nu_4 \circ \eta_{7}]\}$. Otherwise, the boundary is $[\epsilon] \in \coker(J)_8 \cong \Theta_8$.
    \item[(ii)] The boundary of a $7$-connected $17$-manifold $M$ is the standard sphere if and only if $\Psi_{L_{\mathbb{O}}} (M) \in \{1, [\sigma_8 \circ \eta_{15}]\}$. Otherwise, the boundary is $[\eta_4] \in \coker(J)_{16} \cong \Theta_{16}$.
    \item[(iii)] The boundary of an $8$-connected $19$-manifold is diffemorphic to $\omega(f) [h_2 h_4]$. Here, $\omega(f) \in \Z/8\Z$ is an invariant defined by Wall in \cite{Wall67} and $[h_2 h_4] \in \pi_{18} (\Ss) \cong \coker(J)_{18} \cong \Theta_{18}$ is a generator for the kernel of the unit map $\pi_{18} \Ss \to \pi_{18} \ko$, which is isomorphic to $\Z/8\Z$.\footnote{This statement should be interpreted as holding up to multiplication by a $2$-adic unit. Indeed, we have only specified $[h_2 h_4]$ up to multiplication by a $2$-adic unit, and Wall only defined $\omega (f)$ up to a $2$-adic unit.}
  \end{enumerate}
\end{thm}
This theorem will be proved in \Cref{sec:kernel} as \Cref{thm:kernel-main}. We refer the reader to \Cref{sec:frank} for the definition of the invariant $\Psi_g$, as well as $L_{\mathbb{H}}$ and $L_{\mathbb{O}}$.
%
%

This may be regarded as a first step towards a classification of $(n-1)$-connected closed $(2n+1)$-manifolds.
The second step is to compute the inertia groups of $(n-1)$-connected closed $(2n+1)$-manifolds.
This problem would require additional techniques to resolve: the inertia group of such a manifold might have nontrivial intersection with $\bP_{2n+2} \subseteq \Theta_{2n+1}$, and our techniques only deal with the image of the inertia group in $\coker(J)_{2n+1}$.
To see how complicated the problem of determining the $\bP$-component of the inertia group can be, we refer the reader to \cite{CrowleyNordstrom}, in which the inertia groups of $2$-connected $7$-manifolds are computed.

\subsection{Outline}
In Section \ref{sec:background}, we recall from \cite{stolz} how one may reduce the $n \neq 4,8,9$ cases of Theorem \ref{determination} to \Cref{thm:ker-main}.
Moreover, we recall from \cite{bhs} how the kernel of the unit map $\pi_{2n} \Ss \to \pi_{2n} \MOn$ may be studied using the bar spectral sequence.

In Section \ref{sec:D2}, we compute the $\mathrm{E}^1$-page of this bar spectral sequence.
Section \ref{sec:lower} is devoted to obtaining lower bounds on the $\F_2$-Adams filtration of elements in the kernel of the unit map $\pi_{2n} \Ss \to \pi_{2n} \MOn$.
Our methods make crucial use of Pstr\k{a}gowski's category of $\F_2$-synthetic spectra \cite{Pstragowski} and refine the techniques in \cite{inertia}.
In Section \ref{sec:vanishing}, we combine Chang's improved $v_1$-banded vanishing line for the mod $2$ Moore spectrum \cite{chang} with the bounds of \Cref{sec:lower} to prove all but a small number of exceptional cases of \Cref{determination}.

%
We proceed to analyze the exceptional cases in Section \ref{sec:exceptional} using various ad hoc arguments, including explicit knowledge of the stable homotopy groups of the sphere in low dimensions \cite{Isaksen}, the Ando--Hopkins--Rezk string orientation \cite{AHR}, the results of \cite{geography}, and a slight optimization of the bounds obtained in \Cref{sec:lower}.
This finishes the proof the triviality of the inertia group of $(n-1)$-connected $2n$-manifolds $M$ outside of $n=4,8,9$.
In Section \ref{sec:det}, we apply  Kreck's modified surgery theory \cite{kreck} and a special case of Crowley's $Q$-form conjecture, proven by Nagy \cite{nagy, nagyarxiv}, to determine the size of the inertia group of $M$ when $n=4,8,9$, thereby completing the proof of Theorem \ref{determination}. Additionally, we prove in \Cref{sec:h-c} that the homotopy and concordance inertia groups of $(n-1)$-connected $2n$-manifolds always vanish when $n \geq 3$. In Section \ref{sec:kernel}, we build upon Frank's work \cite{Frank,frank74wall} and  determine the kernel of the boundary map $\partial : A_{2n+1} ^{\langle n \rangle} \to \Theta_{2n}$ when $n=4,8$.

\subsection{Acknowledgements}
We thank Robert Burklund and Sanath Devalapurkar for helpful discussions regarding the contents of this paper. We are particularly grateful to Diarmuid Crowley for assistance with the application of modified surgery and the $Q$-form conjecture in \Cref{sec:det}. We would also like to thank the referee for their careful reading of the paper and their many helpful comments. During the course of this work, the first named author was partially supported by NSF grant DMS-2103236, and the second named author was partially supported by NSF grant DMS-1906072, the DNRF through the Copenhagen Centre for Geometry and Topology (DRNF151), and the European Union via the Marie Skłodowska-Curie postdoctoral fellowship (project 101150469).

\section{Background} \label{sec:background}
\todo{Cites throughout}
\NB{Mention $n \equiv 2 \mod 8$ case is done by Stolz when $n > 10$. When $n = 10$, use that $\pi_{20} \Ss \to \pi_{20} \tmf$ is injective.}
In this section, we recall from \cite{stolz,bhs, inertia} how to reduce the geometric problem of determining the inertia groups of highly connected manifolds to a computation in homotopy theory.

To begin, we relate inertia groups of $(n-1)$-connected closed $2n$-manifolds to boundaries of $(n-1)$-connected almost closed $(2n+1)$-manifolds.

\begin{proposition}[{{\cite{Wall62}\cite[\S 15]{stolz}}}]
  Let $M$ be an $(n-1)$-connected closed $2n$-manifold, and let $\Sigma \in I(M)$. Then $\Sigma$ is the boundary of an $(n-1)$-connected almost closed $(2n+1)$-manifold.
\end{proposition}

We will use the above proposition to prove the vanishing of $I(M)$ when $n \neq 4,8,9$. We defer the discussion of the cases $n = 4,8,9$ to \Cref{sec:det}.

Our goal is therefore to understand which exotic $2n$-spheres are the boundaries of $(n-1)$-connected almost closed $(2n+1)$-manifolds.
To this end, it is useful to consider the cobordism group $A_{2n+1} ^{\langle n \rangle}$ of $(n-1)$-connected $(2n+1)$ almost closed manifolds modulo oriented cobordisms that restrict to $h$-cobordisms on the boundary.
The operation of taking the boundary determines a group homomorphism
\[\partial : A_{2n+1} ^{\langle n \rangle} \to \Theta_{2n},\]
which we would like to prove is zero.

By work of Wall, Frank, and Stolz \cite{Wall67,frank74wall,stolz}, the groups $A_{2n+1}^{\langle n \rangle}$ have been computed as below.

\begin{thm}\label{thm:An}\cite{Wall67,frank74wall,stolz}
  When $n \geq 3$, there are isomorphisms
  \[A_{2n+1} ^{\langle n \rangle} \cong \begin{cases} \Z /2 \Z \oplus \Z/ 2 \Z &\mathrm{ when } \quad n \equiv 0 \mod 8 \quad \mathrm{ or } \quad n = 4 \\ \Z /8 \Z &\mathrm{ when } \quad n \equiv 1 \mod 8 \\ \Z /2 \Z &\mathrm{ when } \quad n \equiv 2 \mod 8 \\ \Z /2 \Z &\mathrm{ when } \quad n \equiv 4 \mod 8 \quad \mathrm{ and } \quad n \neq 4 \\ 0 &\mathrm{ otherwise}. \end{cases}\]
\end{thm}

An immediate consequence is the theorem of Wall that $I(M) = 0$ whenever $n \equiv 3,5,6,7 \mod 8$ \cite{WallAction}.
Furthermore, we learn that the problem is entirely $2$-primary.

To compute $A_{2n+1} ^{\langle n \rangle}$, Stolz introduced the spectrum $A[n]$ \cite{stolz}, which lies in a diagram
\begin{center}
    \begin{tikzcd}
    & & A[n] \ar[d, "b"]&\\
      \Ss\ar[r] & \mo\langle n \rangle \ar[r] & \mo\langle n \rangle/\Ss \ar[r,"\delta"] & \Ss^1,
    \end{tikzcd}
\end{center}
where $\MOn$ is the Thom spectrum of the canonical map $\tau_{\geq n} \BO \to \BO$, and wrote down a map $A_{2n+1}^{\langle n \rangle} \to \pi_{2n+1} A[n]$ which is an isomorphism for $n \geq 9$.
Moreover, he identified the composite
\[A_{2n+1} ^{\langle n \rangle} \xrightarrow{\partial} \Theta_{2n} \to \coker(J)_{2n}\]
with the composite
\[A_{2n+1} ^{\langle n \rangle} \to \pi_{2n+1} A[n] \xrightarrow{b_*} \pi_{2n+1} \MOn / \Ss \xrightarrow{\delta_*} \pi_{2n} (\Ss) \to \coker(J)_{2n}.\]

Since the image of the map $\Theta_{2n} \to \coker(J)_{2n}$ consists of the elements of Kervaire invariant zero and $bP_{2n+1}=0$ so that $\Theta_{2n} \to \coker(J)_{2n}$ is an injection, this implies the following criterion for detecting that the image of $\partial$ is the trivial element of $\Theta_{2n}$.

\begin{prop} \label{kernelunit}
Let $n\geq 3$. Suppose the kernel of the $2$-completed unit map 
  \[\pi_{2n}(\Ss_2)\rightarrow\pi_{2n}(\mo\langle n\rangle_2)\]
is generated by the image of $J$ and an element of Kervaire invariant one.
  Then the inertia group $I(M)$ of any $(n-1)$-connected closed $2n$-manifold $M$ is trivial. 
\end{prop}

In particular, the cases $n \neq 4,8,9$ of \Cref{determination} follow from \Cref{thm:ker-main}.
We will therefore focus on the proof of \Cref{thm:ker-main} until we take up the cases $n=4,8,9$ for \Cref{determination} in \Cref{sec:det}.

We begin by noting that the work of Stolz allows us to immediately deduce several cases of \Cref{determination}.

\begin{thm}[\cite{stolz}] \label{thm:reduction}
  The kernel of the unit map
  \[\pi_{2n} \Ss \to \pi_{2n} \MOn\]
  is generated by the image of $J$ when $n \equiv 2,3,5,6,7 \mod 8$ and $n > 10$.
\end{thm}

\begin{proof}
  It follows from \cite[Satz 3.1(i)]{stolz} that the image of
  \[\pi_{2n+1} A[n] \xrightarrow{b_*} \pi_{2n+1} \MOn / \Ss \xrightarrow{\delta_*} \pi_{2n} \Ss\]
  and the image of $J$ generate the kernel of the unit map
  \[\pi_{2n} \Ss \to \pi_{2n} \MOn.\]
  The $n \equiv 3,5,6,7 \mod 8$ cases of the theorem therefore follow from \Cref{thm:An} and the isomorphism $A^{\langle n \rangle} _{2n+1} \cong \pi_{2n+1} A[n]$ for $n \geq 9$ from \cite[Theorem A]{stolz}.

  The $n \equiv 2 \mod 8$ case of the theorem follows from (the proof of) \cite[Theorem B(i)]{stolz}. More precisely, see \cite[Beweis von Satz 12.2]{stolz} on page 103 of \cite{stolz}.
\end{proof}

Therefore, the remaining cases of \Cref{thm:ker-main} are when $n\equiv 0,1,4 \mod 8$ and $n = 2,3,5,6,7,10$. The latter will be dealt with in \Cref{sec:exceptional}, so we will focus on the case $n \equiv 0,1,4 \mod 8$ in the next few sections.

Finally, we reduce \Cref{thm:ker-main} to a $2$-primary question.

\begin{prop} \label{prop:2-only}
  Let $n \geq 1$. The kernel of the unit map
  \[\pi_{2n} \Ss \to \pi_{2n} \MOn\]
  is generated by the image of $J$ and $2$-power torsion elements.
\end{prop}

\begin{proof}
  The groups $\pi_2 \Ss \cong \Z/2\Z$ and $\pi_4 \Ss \cong 0$ are $2$-torsion, so we may assume $n \geq 3$.
  By \cite[Satz 3.1(i)]{stolz}, it suffices to prove that the image of $\pi_{2n+1} A[n] \to \pi_{2n} \Ss$ is $2$-power torsion.
  Since the cokernel of $\Theta_{2n} \to \coker(J)_{2n}$ is at most $2$-torsion, this follows from \cite[Satz 1.7]{stolz} and \Cref{thm:An}.
\end{proof}

\subsection{$\MOn$ as a bar construction} \label{sec:MOnBar}
To study the kernel of the unit map, we recall from \cite{bhs} how $\mo\langle n\rangle$ may be expressed as the geometric realization of a two-sided bar construction and draw consequences for the kernel of the unit map.

Looping the defining map $\tau_{\geq n}\mathrm{BO}\rightarrow \mathrm{BO}$ of $\mo\langle n\rangle$ once yields a map $$\mathrm{O}\langle n-1\rangle:=\Omega^{\infty}\tau_{\geq_{n-1}}\Sigma^{-1}\mathrm{bo}\rightarrow \mathrm{O}.$$ Composing with the classical $J$-homomorphism $O\rightarrow\Omega^{\infty}\Ss$ and applying the $(\Sigma^{\infty},\Omega^{\infty})$ adjunction, we obtain a map  of non-unital $\mathbb{E}_{\infty}$-ring-spectra $$J:\R\rightarrow\Ss,$$ and hence a map of unital $\mathbb{E}_\infty$-ring spectra $J_+:\Sigma^\infty _+ \mathrm{O}\langle n-1 \rangle\rightarrow\Ss$. By \cite[Definition 4.1]{ABGHR}, we can express $\mo\langle n\rangle$ as the relative tensor product $$\mo\langle n\rangle\simeq\Ss \tens_{\Sigma^\infty _+ \mathrm{O}\langle n-1 \rangle} \Ss\simeq|\mathrm{Bar}_\bullet(\Ss, \Sigma^\infty _+ \mathrm{O} \langle n-1 \rangle, \Ss)|,$$ where the action is given by the augmentation on the left and $J_+$ on the right.
There is a bar spectral sequence computing this relative tensor product, which is obtained by the skeletal filtration of the geometric realization of the two-sided simplicial bar construction. This bar spectral sequence takes the form
\[\mathrm{E}^1 _{t,s} = \pi_{t} (\R^{\otimes s}) \Rightarrow \pi_{t+s} \MOn. \]
For connectivity reasons, to understand the homotopy group in degree $2n\leq 3n-2$, it is equivalent to computing the bar spectral sequence for $|\mathrm{Bar}_\bullet(\Ss, \Sigma^\infty _+ \mathrm{O}\langle n-1 \rangle, \Ss)|_{\leq 2}$.

\begin{rmk}
  The image of the map $\pi_{k} (J) : \pi_k (\R) \to \pi_k (\Ss)$ is not in general equal to the image of the classical $J$-homomorphism. Whenever we refer to the ``image of $J$'' in this work, we mean the image of the classical $J$-homomorphism and not the image of $\pi_{k} (J) : \pi_k (\R) \to \pi_k (\Ss)$.
\end{rmk}

\begin{proposition}[{{\cite[Theorem 5.2]{bhs}}}]
After applying $\tau_{\leq 3n-2}$, there is an equivalence of $\mathbb{E}_0$-algebras between $\mo\langle n\rangle$ and the total cofiber of the three-term complex
\begin{center}
    \begin{tikzcd}
    \R^{\tens 2}\ar[r, "m-1\tens J"]\ar[rr, bend left=40,"0"{name=C}]&\R\ar[r,"J"]\ar[Leftrightarrow, from=C, "\mathrm{can}"]&\Ss
    \end{tikzcd},
\end{center}
where $m$ is the multiplication and $\mathrm{can}$ the null homotopy coming from the canonical filling of the square
\begin{center}
    \begin{tikzcd}
    \R^{\tens 2}\ar[d,"m"]\ar[r,"1\tens J"]&\R\ar[d,"J"]\\
    \R\ar[r,"J"]\ar[Leftrightarrow, ur,"\mathrm{can}"]&\Ss
    \end{tikzcd}
\end{center}
witnessing $J$ as a map of non-unital $\mathbb{E}_{\infty}$-algebras.
\end{proposition} 
 Hence to compute the kernel of the unit map $\pi_{2n}(\Ss)\rightarrow\pi_{2n}(\mo\langle n \rangle)$, it suffices to consider two types of elements:

(1) Elements in the image of the map $\pi_{2n}(J)$, or the $d_1$-differential of the bar spectral sequence;

(2) Elements in the image of the  Toda bracket
\begin{center}
    \begin{tikzcd}
    \Ss^{2n-1}\ar[rr, bend left=40,"0"{name=D}]\ar[r]&\R^{\tens 2}\ar[r, "m-1\tens J"]\ar[Leftrightarrow, from=D,"\mathrm{arbitrary}"]\ar[rr, bend right=40,"0"{name=C}]&\R\ar[r,"J"]\ar[Leftrightarrow, from=C, "\mathrm{can}"]&\Ss
    \end{tikzcd},
\end{center}
i.e., the $d_2$-differential in the bar spectral sequence.

Since elements in the image of the classical $J$-homomorphism have high $\F_2$-Adams filtrations, our plan is to first show that both types of elements have high $\F_2$-Adams filtrations, then use this to show that all elements of sufficiently high $\F_2$-Adams filtration lie in the image of $J$, following the strategy employed in \cite{stolz,bhs,inertia,geography}.

\section{The $\mathrm{E}^1$-page of the bar spectral sequence} \label{sec:D2}
\todo{Ideally, remake all spectral sequence pictues using Hood's package}
\todo{Explain how to compute homotopy of $A[n]$ using this spectral sequence}
To analyze the kernel of the unit map
\[\pi_{2n} \Ss \to \pi_{2n} \MOn\]
using the bar spectral sequence of \Cref{sec:MOnBar}, the first step is to compute the relevant part of the $\mathrm{E}^1$-page, which consists of the groups $\pi_{2n-1} \R$, $\pi_{2n} \R$ and $\pi_{2n-1} \R^{\otimes 2}$.
By \Cref{thm:reduction}, we only need to consider the cases $n \equiv 0, 1, 4 \mod 8$.
In this section, we will compute these groups using the Adams spectral sequence.

\begin{cnv}\label{cnv:2-compl}
  In the remainder of the paper, we will implicitly work in the $2$-complete category. This is justified by \Cref{prop:2-only}.
\end{cnv}

Below we include a picture of the part of the bar spectral sequence that is relevant to the computation of the kernel of the unit map in degree $2n$.
In particular, this is the part of the spectral sequence which interacts with the differentials entering $\pi_{2n} \Ss$.
We display the spectral sequence in homological Adams grading. A $d_r$-differential has bidegree $(-r,-1)$.

\begin{table}[h!]
    \centering
        \begin{tabular}{c|cc}
        $s=2$ && $\pi_{2n-1} (\R^{\otimes 2})$\\ 
        $s=1$ &$\pi_{2n-1} (\R)$&$\pi_{2n} (\R)$\\
        $s=0$ &$\pi_{2n}(\Ss)$&\\\hline
        $t+s$  &$2n$&$2n+1$\\
    \end{tabular}
    \label{barsseq}
\end{table}

\begin{thm}\label{thm:R-homotopy}
  The homotopy groups $\pi_{2n-1} (\R)$, $\pi_{2n} (\R)$, and $\pi_{2n-1} (\R^{\otimes 2})$ are given by:
  \begin{center}
  \begin{tabular}{c|c|c|c}
    $n \mod 8$ &  $0$ & $1$ & $4$ \\\hline
    $\pi_{2n-1} (\R)$ & $\pi_{2n} \mathrm{bo}$ & $\pi_{2n} \mathrm{bo} \oplus \Z/2\Z$ & $\pi_{2n} \mathrm{bo}$ \\\hline
    $\pi_{2n} (\R)$ & $\pi_{2n+1} \mathrm{bo} \oplus \Z/2\Z$ & $\pi_{2n+1} \mathrm{bo} \oplus \Z/2\Z$ & $\pi_{2n+1} \mathrm{bo} \oplus \Z/2\Z \oplus \Z/2\Z$ \\\hline
    $\pi_{2n-1} (\R^{\otimes 2})$ & $\Z/2\Z$ & $\Z/4\Z$ & $0$
  \end{tabular}
  \end{center}
\end{thm}

Plugging \Cref{thm:R-homotopy} into the bar spectral sequence, we obtain the following diagram, where the groups in parentheses come from the summands of the form $\pi_* \mathrm{bo}$:

\begin{table}[h!]
    \centering
        \begin{tabular}{c|cc|cc|cc}
        $n \mod 8$ &$n \equiv 0$&&$n \equiv 1$ &&$n \equiv 4$  \\\hline
        $s=2$ &&$\Z/2\Z$&&$\Z/4\Z$&&0\\ 
        $s=1$ &$(\Z)$&$\Z/2\Z\oplus(\Z/2\Z)$&$\Z/2\Z\oplus(\Z/2\Z)$&$\Z/2\Z$&$(\Z)$ & $\Z/2\Z\oplus\Z/2\Z\oplus(\Z/2\Z)$\\
        $s=0$ &$\pi_{2n}(\Ss)$&&$\pi_{2n}(\Ss)$&&$\pi_{2n}(\Ss)$\\\hline
        $t+s$  &$2n$&$2n+1$&$2n$&$2n+1$&$2n$&$2n+1$\\
    \end{tabular}
\end{table}

\subsection{Computation of $\pi_* (\R)$}
In order to compute the homotopy groups of $\R$, we will follow \cite[Section 4]{bhs} and use the Goodwillie tower of the forgetful functor $U : \mathrm{Alg}_{\mathbb{E}^{\mathrm{nu}}_\infty}(\mathrm{Sp})\rightarrow\mathrm{Sp}$. Given a non-unital $\Einf$-ring $A$, Kuhn has identified the layers of the Goodwillie tower \cite[Theorem 3.10]{kuhn} as
\begin{center}
    \begin{tikzcd}[column sep=0.5cm]
    & &D_n(\mathrm{TAQ}(\Ss\oplus A))\ar[d]&&D_2(\mathrm{TAQ}(\Ss\oplus A))\ar[d]&\mathrm{TAQ}(\Ss\oplus A)\ar[d]&\\
    U(A)\ar[r]&\cdots\ar[r]& P_n(U)(A)\ar[r]&\cdots\ar[r]& P_2(U)(A)\ar[r]&P_1(U)(A)\ar[r]&0,
    \end{tikzcd}
\end{center}
where $\Ss \oplus A$ is the augmented $\Einf$-ring associated to $A$.
Specializing to $A=\R$, so that $A=\Sigma^\infty \Omega^\infty\littleo$ where $\littleo=\tau_{\geq n-1}\Sigma^{-1}\mathrm{bo}$, we have $\mathrm{TAQ}(\Ss \oplus \R)\simeq \littleo$ \cite[Example 3.9]{kuhn}. 

\begin{proposition} [{{\cite[Lemma 4.10]{inertia}}}] \label{prop:D2-split}
The cofiber sequence $$D_2(\littleo)\rightarrow P_2(U)(\R) \rightarrow \littleo$$ induces in degree $*\leq 3n-4$ split short exact sequences
$$\pi_*D_2(\littleo)\rightarrow \pi_*\R \rightarrow \pi_*\littleo,$$ 
where the splitting is given by the composition
  \[\pi_*\littleo \cong \pi_* \mathrm{O} \langle n-1 \rangle \xrightarrow{\Sigma^\infty} \pi_* \R.\]
\end{proposition}

We therefore need to compute $\pi_* D_2 (\littleo)$ for $*\leq 2n$, which we will accomplish using the $\F_2$-Adams spectral sequence.

As input, we need to compute $\H_* (D_2 (\littleo); \F_2)$ with its right action of the Steenrod algebra in the range $* \leq 2n+1$.
First, we recall the structure of $\H_* (\littleo; \F_2)$ for $* \leq n+2$ from \cite[Example 2.31]{tmfAdamsBook}\cite[Theorem A]{Stong}. We have pictured the result in cell diagram notation below:
\begin{equation}\label{littleocelldiagram}
    \begin{tikzcd}[every arrow/.append style={dash}, column sep=10mm]
  n  \mod 8  & 0 & 1 &  4\\[-18pt]
      n+2           & &y_{n+2} &y_{n+2}\\[-15pt]
      n+1           & &        &y_{n+1}\ar[u]\\[-15pt]
n             &&y_n \ar[uu, bend right=45] &\\[-15pt]
      n-1 & y_{n-1} & y_{n-1} \ar[u] &  y_{n-1} \ar[uu,bend left=45]
    \end{tikzcd}
\end{equation}
In the figure, each $y_i$ indicates a copy of $\F_2$ with generator $y_i$, lines of length $1$ indicate $\Sq_1$ and lines of length two indicate $\Sq_2$.
The pictured groups are cyclic, and we write $y_k \in \H_k (\littleo; \F_2)$ for a generator.

Using the known mod $2$ homology of second extended powers \cite[Lemma 1.3]{may} and the Nishida relations \cite[Theorem 9.4]{may}, we obtain the following diagram of $\H_* (D_2 (\littleo); \F_2)$:
%
%
%
\begin{center}
    \begin{tikzcd}[column sep=0.2cm, font=\small]
        n \mod 8  && 0  && 1  &&  4\\
        2n+1 && Q_3 (y_{n-1})\ar[dash, d] &Q_3 (y_{n-1})& Q_1 (y_n)& y_{n-1} \cdot y_{n+2}& Q_3 (y_{n-1})\ar[d,dash] &y_{n-1} \cdot y_{n+2}\ar[d,dash]\\
        2n && Q_2 (y_{n-1})\ar[dash, dd, bend right=60]& Q_2 (y_{n-1})& Q_0(y_n)\ar[u,dash]&& Q_2 (y_{n-1})&y_{n-1} \cdot y_{n+1}\\
        2n-1 && Q_1(y_{n-1})\ar[dash, d]&Q_1 (y_{n-1})\ar[uur, dash]\ar[u,dash]&y_{n-1} \cdot y_n\ar[dl,dash]\ar[uur,bend right, dash]&& Q_1(y_{n-1})\\
        2n-2 && Q_0 (y_{n-1})& Q_0 (y_{n-1})\ar[uur, dash]&&&Q_0 (y_{n-1})\ar[u,dash]\ar[uu, bend left=50,dash]\ar[uur,bend right, dash]\\
    \end{tikzcd}
\end{center}
Recall that the Dyer Lashof operation $Q_i$ sends a class in  $H_j(X;\F)$ to $H_{2j+i}(D_2(X);\F)$ for $i\geq 0$.

Plugging this into the $\F_2$-Adams spectral sequence, we obtain the following proposition by standard computations:
\begin{prop}\label{prop:D2o-E2}
  The $\mathrm{E}_2$-page of the $\F_2$-Adams spectral sequence for $\pi_* D_2 (\littleo)$ is given as below:
 \begin{center}
    \begin{tikzcd}[every arrow/.append style={dash}, column sep=0.2cm, row sep=0.2cm]
   &   n  \equiv 0 & && && & n  \equiv 1 &\\[-5pt]
            &    & h_1 Q_1 (y_{n-1})   & &  &&& & h_1 Q_1 (y_{n-1}) \\
      Q_0 (y_{n-1}) &    & & &&& Q_0 (y_{n-1})&Q_1 (y_{n-1})\ar[ru] & \\[-5pt]
    2n-2  & 2n-1   & 2n  &&&&  2n-2  & 2n-1   & 2n.
    \end{tikzcd}
\end{center} 
 \begin{center}
    \begin{tikzcd}[every arrow/.append style={dash}, column sep=0.2cm, row sep=0.2cm]
      &n \equiv 4&\\[-5pt]
      &&h_1 Q_1 (y_{n-1}) \\
      Q_0 (y_{n-1})&&Q_2(y_{n-1}) + y_{n-1} \cdot y_{n+1}\\[-5pt]
      2n-2&2n-1&2n.
    \end{tikzcd}
\end{center} 
\end{prop}

\begin{proof}
  The is a standard computation, which may be carried out by hand or using Bruner's \texttt{ext} program \cite{BrunerExt}.
\end{proof}

\subsection{Computation of $\pi_* (\R ^{\otimes 2})$}
Our computation of $\pi_* (\R ^{\otimes 2})$ proceeds similarly as our computation of $\pi_* (D_2 (\littleo))$ above.
To begin with, since $\pi_* (\R ^{\otimes 2}) \cong \pi_* (\littleo ^{\otimes 2})$ for $* \leq 3n-4$, we may as well work with the latter. Using the K\"unneth theorem and the Cartan formula, we obtain the following computation of $\H_* (\littleo^{\otimes 2}; \F_2)$ for $* \leq 2n$:

%
\begin{center} 
    \begin{tikzcd} [column sep=0.2pt]
        & n \equiv 0 & n \equiv 1&& n \equiv 4\\
        2n && y_n\tens y_n\ar[d,dash]\ar[dr, dash]\ar[dd, dash, bend right=60]&&y_{n-1}\tens y_{n+1}&y_{n+1}\tens y_{n-1}\\
        2n-1 &&y_{n-1}\tens y_n\ar[d,dash] & y_n \tens y_{n-1}\ar[dl,dash]&\\
        2n-2 & y_{n-1}\tens y_{n-1}& y_{n-1}\tens y_{n-1} && y_{n-1}\tens y_{n-1}\ar[uu,dash]\ar[uur, dash, bend right]\\
        
    \end{tikzcd}
\end{center} 
Plugging these computations into the $\F_2$-Adams spectral sequence, we obtain the following proposition.
\begin{prop} \label{prop:Otens2} 
  The $\mathrm{E}_2$-page of the $\F_2$-Adams spectral sequence for $\R^{\otimes 2}$ is given by:
\begin{center}
    \begin{tikzcd}[every arrow/.append style={dash}, column sep=0.4cm, row sep=0.4cm]
      n  \equiv 0 & && &&  n  \equiv 1 &\\[-5pt]
      \vdots&&&&& & \\
        h_0 (y_{n-1} \otimes y_{n-1}) \ar[u]    & h_1 (y_{n-1} \otimes y_{n-1})   &  &  &&& h_1 (y_{n-1} \otimes y_{n-1}) \\
      y_{n-1} \otimes y_{n-1} \ar[u] \ar[ur] &    & & && y_{n-1} \otimes y_{n-1} \ar[ur]&y_{n-1} \otimes y_n + y_n \otimes y_{n-1} \ar[u] \\[-5pt]
    2n-2  & 2n-1   & &&&  2n-2  & 2n-1.
    \end{tikzcd}
\end{center} 
\begin{center}
    \begin{tikzcd}[every arrow/.append style={dash}, column sep=0.4cm, row sep=0.4cm]
      n \equiv 4&\\[-5pt]
      \vdots & \\
      h_0 (y_{n-1} \otimes y_{n-1}) \ar[u]& \\
      y_{n-1} \otimes y_{n-1} \ar{u}&\\[-5pt]
      2n-2&2n-1.
    \end{tikzcd}
\end{center} 
\end{prop}

Once again, there is no space for differentials or hidden extensions, so we may read off the homotopy groups of $\R^{\otimes 2}$ claimed in \Cref{prop:D2o-E2}.

%

\section{Lower bounds} \label{sec:lower}
In this section, we provide lower bounds for the $\F_2$-Adams filtration of the images of the $d_1$ and $d_2$-differentials with target $\pi_{2n} (\Ss)$ in the bar spectral sequence for $\pi_* \MOn$.

Our method is based on \cite[Section 5]{inertia}, where Pstr\k{a}gowski's symmetric monoidal category of $\F_2$-synthetic spectra was used to obtain the desired statement from the fact that $J : \R \to \Ss$ is of high $\F_2$-Adams filtration when restricted to a skeleton.
Here we make the following refinement to this method: from the $\mathrm{E}^1$-page of the bar spectral sequence, we know that the images of the $d_1$ and $d_2$-differentials must be of certain torsion orders, so that they are represented by maps from Moore spectra.
We prove that these maps from Moore spectra must be of high $\F_2$-Adams filtration, which allows us to make use of better vanishing lines available for Moore spectra.

We assume familiarity with the language and basic properties of the category of $\F_2$-synthetic spectra. A general introduction can be found in \cite{Pstragowski} and a more computational perspective is detailed in \cite[Section 9 and Appendix A]{bhs}.

\begin{cnv}
  Let $\wt{2} \in \pi_{0,1} \Ss$ denote the synthetic lift of $2$ in the sense of \cite[Theorem 9.19]{bhs} (see also \cite[Proposition A.20]{bhs}), and let $\wt{2^n} = \wt{2}^n$. Note that $\tau \wt{2} = 2$ by definition.
  We implicitly complete all $\F_2$-synthetic spectra at the class $\wt{2} \in \pi_{0,1} \Ss$. (Recall from \Cref{cnv:2-compl} that we are implicitly $2$-completing all spectra.)
\end{cnv}

Our results are stated in terms of the following constants:
\begin{ntn}
  Given an integer $n$, we let
  \[M_1 = h(n-1) - \lfloor \log_2 (n+3) \rfloor + 1\]
  and
  \[M_2 = h(n-1) - \lfloor \log_2 (2n+2) \rfloor + 1,\]
  where $h(k)$ denotes the number of integers $0 < s \leq k$ which are congruent to $0$, $1$, $2$ or $4$ mod $8$.
  Note that we have suppressed the dependence of $M_1$ and $M_2$ on $n$ in the notation.
\end{ntn}

Our main results are as follows:
\begin{thm} \label{thm:d1}
   If $n \equiv 0,1,4 \mod 8$, then every element in the image of the $d_1$-differential of the bar sequence, i.e., the map $J : \pi_{2n} (\R) \to \pi_{2n} (\Ss)$, is represented modulo the image of $J$ by a map $\Ss^{2n} / 2 \to \Ss$ of $\F_2$-Adams filtration at least $2M_1 - 3$.
\end{thm}

\begin{thm} \label{thm:d2}
  \hspace{1cm}
  \begin{enumerate}
    \item If $n \equiv 0 \mod 8$ and $n \geq 16$, then any class in $\pi_{2n} \Ss$ which is killed by a $d_2$-differential in the bar spectral sequence for $\MOn$ is represented by a map $\Ss^{2n} /4 \to \Ss$ that admits a lift to an $\F_2$-synthetic map of the form $\Ss^{2n,2n+2M_2-4} / \wt{4} \to \Ss^{0,0}$. \footnote{In the language of \cite[Section 3]{BHHM}, this means that $\Ss^{2n} / 4 \to \Ss$ has \emph{modified $\F_2$-Adams filtration} at least $2M_2 -4$.}
    \item If $n \equiv 1 \mod 8$ and $n \geq 17$, then any class in $\pi_{2n} \Ss$ which is killed by a $d_2$-differential in the bar spectral sequence for $\MOn$ is represented by a map $\Ss^{2n} /8 \to \Ss$ that admits a lift to an $\F_2$-synthetic map of the form $\Ss^{2n,2n+2M_2-5} / \wt{8} \to \Ss^{0,0}$. \footnote{In the language of \cite[Section 3]{BHHM}, this means that $\Ss^{2n} / 8 \to \Ss$ has \emph{modified $\F_2$-Adams filtration} at least $2M_2-5$.}
    \item If $n \equiv 4 \mod 8$, then $\pi_{2n-1} (\R^{\otimes 2}) = 0$ and thus the $d_2$-differential in the bar spectral sequence landing in $\pi_{2n} \Ss$ is zero.
  \end{enumerate}
\end{thm}
\NB{This is poorly phrased. Rewrite it.}

Before proceeding to the proofs, we recall a special case of \cite[Lemma 10.18]{bhs}, which will be a key input to the arguments in this section.

\begin{lem} \label{lem:R-AF}
  Let $N_1 \to \R$ and $N_2 \to \R$ denote the inclusion of an $(n+2)$-skeleton and a $(2n+1)$-skeleton, respectively. Then the composite maps $N_i \to \R \xrightarrow{J} \Ss$ have $\F_2$-Adams filtration at least $M_i$, i.e. they admit $\F_2$-synthetic lifts
  \[\wt{J} : \Sigma^{0,M_i} \nu N_i \to \Ss^{0,0}.\]
\end{lem}

\subsection{Lower bounds for the $d_1$-differentials}
Our proof of \Cref{thm:d1} will make use of the following lemma, which compares the homotopy groups of the second extended power in the category of $\F_2$-synthetic spectra with the homotopy groups of the second extended power in the category of spectra.

\begin{lemma}{\cite[Lemma 4.13]{inertia}}\label{D2}
If $X$ is an $n$-connective spectrum, then the map 
$$\pi_{t-s,t}(D_2(\Sigma^{a,b}\nu X))\rightarrow \pi_{t-s}(D_2(\Sigma^a X))$$ obtained by inverting $\tau$ is an isomorphism for $t\leq 2n+2b$.
\end{lemma}

\begin{proof}[Proof of \Cref{thm:d1}]
  We will compute the image of the map $\pi_{2n} (J) : \pi_{2n} \R \to \pi_{2n} \Ss$.
  Recall from \Cref{prop:D2-split} that $\pi_{2n} \R \cong \pi_{2n} \littleo \oplus \pi_{2n} D_2 (\littleo)$.
  By definition, the restriction of $\pi_{2n} (J)$ to the $\pi_{2n} \littleo$ factor is the classical $J$-homomorphism.
 Hence it suffices to treat the $\pi_{2n} D_2 (\littleo)$ factor.

  By \cite[Example 4.9]{inertia}, there is a commutative diagram of the form
  \begin{center}\begin{tikzcd}
    D_2(\R) \ar[r, "\hat{m}"] \ar[d, "D_2(\pi)"] &
    \R \ar[d] \ar[dr, "\pi"] \\
    D_2(\littleo) \ar[r] &
    P_2(U)(\R) \ar[r] &
    \littleo.
  \end{tikzcd}\end{center}
  Here, $\hat{m}$ is induced by the non-unital $\mathbb{E}_\infty$-ring structure on $\R$, and $\pi$ is the natural map from $\R$ to its stabilization as a non-unital $\mathbb{E}_\infty$-ring \cite[Notation 4.1 and 4.6]{inertia}.

  Since the map $D_2 (\pi) : D_2(\R) \to D_2 (\littleo)$ is an equivalence in degrees at most $3n-4$, it suffices to show that everything in the image of the composite
  \[\pi_{2n} D_2 (\R) \xrightarrow{\pi_{2n} D_2 (J)} \pi_{2n} D_2(\Ss) \to \pi_{2n} \Ss\]
  is represented by a map $\Ss^{2n} / 2 \to \Ss$ of $\F_2$-Adams filtration at least $2M_1 -3$.

  Note that we may as well replace $\R$ by an $(n+2)$-skeleton $N_1 \to \R$, since the induced map $\pi_{2n} D_2 (N_1) \to \pi_{2n} D_2 (\R)$ will be an isomorphism.
  By \Cref{lem:R-AF}, the composite $N_1 \to \R \xrightarrow{J} \Ss$ admits a synthetic lift
  \[\wt{J} : \Sigma^{0,M_1} \nu N_1 \to \Ss^{0,0}.\]

  Then the composite
  \[D_2 (\Sigma^{0,M_1} \nu N_1) \xrightarrow{D_2 (\wt{J})} D_2 (\Ss^{0,0}) \to \Ss^{0,0}\]
  is an $\F_2$-synthetic lift of
  \[D_2 (N_1) \to D_2 (\R) \xrightarrow{D_2 (J)}  D_2(\Ss) \to \Ss.\]

  Now, \Cref{thm:R-homotopy} implies that $\pi_{2n} D_2 (\littleo)$ is simple $2$-torsion when $n \equiv 0,1,4 \mod 8$.
  It therefore follows from \Cref{D2} that the unique nonzero class in $\pi_{2n} D_2 (N_1)$ lifts to a $2$-torsion class in $\pi_{2n, 2n+2M_1 - 2} D_2 (\Sigma^{0,M_1} \nu N_1)$. Multiplying by $\tau$, we obtain a $\wt{2}$-torsion lift to $\pi_{2n,2n+2M_1 -3} D_2 (\Sigma^{0,M_1} \nu N_1)$.
  Mapping down to $\Ss^{0,0}$ via $D_2 (\wt{J})$, we obtain the desired theorem.
\end{proof}
\subsection{Lower bounds for $d_2$-differentials}

We now move on to the proof of \Cref{thm:d2}.
Our techniques are careful refinements of those in \cite[Section 5]{inertia}.
We start by recalling from \cite[Section 5]{inertia} how the Toda bracket defining the $d_2$-differential entering $\pi_{2n} \Ss$ in the bar spectral sequence may be decomposed into a more convenient form.
This differential is defined on those $x \in \pi_* \R^{\otimes 2}$ on which the $d_1$-differential vanishes, i.e. for which $(m-1 \otimes J) (x) = 0$, where $m : \R^{\otimes 2} \to \R$ is the multiplication.
Since $\R$ is a non-unital $\E_\infty$-ring, the map $m$ factors as
\[\R^{\otimes 2} \xrightarrow{c} D_2 (\R) \xrightarrow{\hat{m}} \R.\]
Then the following proposition follows from the proof of \Cref{thm:d2} below.

\begin{prop}\label{prop:todabracket}
  Suppose that $n \geq 16$ and $n \equiv 0,1,4 \mod 8$.
  Given any $x \in \pi_{2n-1} \R^{\otimes 2}$ which satisfies $(m - 1 \otimes J) (x) = 0$, we have $c (x) = (1 \otimes J) (x) = 0$.
  We may therefore form the following diagram of maps and homotopies:
  \begin{center}
  \begin{tikzcd}
    \Ss^{2n-1} \ar[dr, "x"] \ar[rr, "0"{name=D}] \ar[ddr, bend right, "0"{name=C} left] & & \R \otimes \Ss \ar[dr, "J \otimes 1"] \\
    & \R^{\otimes 2}  \ar[Leftrightarrow, from=C, ""] \ar[Leftrightarrow, from=D, ""] \ar[d, "c"] \ar[ur, "1 \otimes J"] \ar[rr, "J \otimes J"{name=E}] \ar[phantom, from=1-3, to=E, "(\otimes)"] & & \Ss^{\otimes 2} \ar[dl, "c"] \ar[dd, "\simeq"] \\
    & D_2(\R) \ar[urr, phantom, "(c)"] \ar[r, "D_2(J)"] & D_2(\Ss) \ar[dr, "\hat{m}"] \\
    & & & \Ss,    
  \end{tikzcd}
  \end{center}
  where the homotopies $(\otimes)$ and $(c)$ are part of the symmetric monoidal structure on the category of spectra.
  The resulting map $\Ss^{2n} \to \Ss$ is a representative for $d_2 (x)$.
\end{prop}

The proof that the map $\Ss^{2n} \to \Ss$ determined by the above diagram computes $d_2 (x)$ under the assumption that $c(x) = (1 \otimes J) (x) = 0$ is given in \cite[Proof of Proposition 5.4 and Proposition 5.5]{inertia}.

To produce a lower bound on the $\F_2$-Adams filtration of $d_2 (x)$, it suffices to construct a suitable $\F_2$-synthetic lift of the above diagram.
Replacing $\R$ with a $(2n+1)$-skeleton $N_2$, we may form the following synthetic lift of the bottom right hand part of the diagram:
\begin{center}
  \begin{tikzcd}
    & & \Sigma^{0,M_2} \nu N_2 \otimes \Ss^{0,0} \ar[dr, "\wt{J} \otimes 1"] \\
    & (\Sigma^{0,M_2} \nu N_2) ^{\otimes 2} \ar[d, "c"] \ar[ur, "1 \otimes \wt{J}"] \ar[rr, "\wt{J} \otimes \wt{J}"{name=E}] \ar[phantom, from=1-3, to=E, "(\otimes)"] & & (\Ss^{0,0})^{\otimes 2} \ar[dl, "c"] \ar[dd, "\simeq"] \\
    & D_2(\Sigma^{0,M_2} \nu N_2) \ar[urr, phantom, "(c)"] \ar[r, "D_2(\wt{J})"] & D_2(\Ss^{0,0}) \ar[dr, "\hat{m}"] \\
    & & & \Ss^{0,0}.    
  \end{tikzcd}
\end{center}

To prove \Cref{thm:d2}, it then suffices to find maps from suitable synthetic Moore spectra to $(\Sigma^{0,M_2} \nu N_2)^{\otimes 2}$ such that:
\begin{enumerate}
  \item the composites with $1 \otimes \wt{J}$ and $c$ are null;
  \item upon inverting $\tau$ and restricting to the bottom cell, one obtains a generator of $\pi_{2n-1} (N_2)^{\otimes 2} \cong \pi_{2n-1} (\R)^{\otimes 2}$.
\end{enumerate}

The following lemma will be useful in showing that the composite with $1 \otimes \wt{J}$ is null.

\begin{lem} \label{lem:D2-lift}
  Suppose that $x \in \pi_{t-s,t} \nu\R$ satisfies $t-s \leq 3n-4$ and the image of $\tau^{-1} x \in \pi_{t-s} \R$ under $\pi: \R\rightarrow \littleo$ is zero.
  Then $\tau x$ is the image of a class in $\pi_{(t-1)-(s-1),t-1} \nu D_2 (\R)$ under the map $\nu \hat{m} : \nu D_2 (\R) \to \nu \R$.
\end{lem}

\begin{proof}
  It follows from \Cref{prop:D2-split} and \cite[Example 4.9]{inertia} that the sequence
  \[D_2 (\R) \xrightarrow{\hat{m}} \R \xrightarrow{\pi} \littleo\]
  yields the following cofiber sequence upon applying $\tau_{\leq 3n-4}$:
  \[\tau_{\leq 3n-4} D_2 (\R) \xrightarrow{\hat{m}} \tau_{\leq 3n-4} \R \xrightarrow{\pi} \tau_{\leq 3n-4} \littleo.\]
  Applying $\nu$, we obtain a diagram
  \begin{center}
    \begin{tikzcd}
      \nu \tau_{\leq 3n-4} D_2(\R) \ar[r] & F \ar[r]\ar[d] & \nu \tau_{\leq 3n-4} \R \ar[r,"\nu\pi"] & \nu \tau_{\leq 3n-4} \littleo,\\&  E
    \end{tikzcd}
  \end{center}
  where $F$ is the fiber of the rightmost map $\nu \pi$ and $E$ is the cofiber of the leftmost map.
  Applying \cite[Lemma 11.15]{bhs}, we learn that $E$ is a $C\tau$-module.
 
 It follows from \cite[Proposition 4.16]{inertia} that the map $\nu X \to \nu \tau_{\leq 3n-4} X$ induces an isomorphism on $\pi_{t-s,t}$ for $t-s \leq 3n-4$.
  Note that the $\F_2$-Adams spectral sequence for $\littleo$ degenerates at the $\mathrm{E}_2$-page,\footnote{By Bott periodicity, this follows from degeneration for $\mathrm{ko}$, $\tau_{\geq 1} \mathrm{ko} = \mathrm{bo}$, $\tau_{\geq 2} \mathrm{ko} = \mathrm{bso}$, and $\tau_{\geq 4} \mathrm{ko} = \mathrm{bspin}$. For $\mathrm{ko}$, this is proven in \cite[pp. 64-66]{GreenBook}. The other cases may be deduced similarly from the form of the $\mathrm{E}_2$-page and the $\mathrm{ko}$-module structure. The $\mathrm{E}_2$-page may be computed using the long exact sequences on $\mathrm{Ext}$ associated to the short exact sequences in cohomology described in \cite[Example 2.31]{tmfAdamsBook}. Note that \cite{GreenBook} uses $\mathrm{bo}$ for what we call $\mathrm{ko}$--when we use $\mathrm{bo}$, it refers to the spectrum associated to the infinite loop space $\mathrm{BO}$.} which implies that $\pi_{*,*} \nu \littleo$ is $\tau$-torsion-free.  Hence the map $x: S^{t-s,t}\rightarrow \nu \tau_{\leq 3n-4} \R$ lifts to a map to the cofiber $F$. Since $E$ is a $C\tau$-module, the $\tau$-multiple of the further composition of this lift with $F\rightarrow E$ is null, so it lifts to the fiber of $F\rightarrow E$ as desired.
%
\end{proof}

We are now ready to prove \Cref{thm:d2}. \Cref{thm:d2}(3) is trivial.

\begin{proof}[Proof of \Cref{thm:d2}(1)]
  Suppose that $n \equiv 0 \mod 8$. Let $y_{n-1} \in \pi_{n-1} \R \cong \Z$ denote a generator. It follows from \Cref{prop:Otens2} that a generator of $\pi_{2n-1} \R^{\otimes 2} \cong \Z/2\Z$ is given by $x = \eta (y_{n-1} \otimes y_{n-1})$.
  Assume that this is a cycle for the $d_1$-differential in the bar spectral sequence, i.e. $(m - 1 \otimes J) (x) = 0$.

  The class $x$ may be represented as the restriction to the bottom cell of the $\F_2$-synthetic map
  \[\wt{x} : \Ss^{2n-1, 2n+2M_2} / \wt{2} \xrightarrow{\wt{\eta}} \Ss^{2n-2,2n-2+2M_2} \xrightarrow{\nu y_{n-1} \otimes \nu y_{n-1}} (\Sigma^{0,M_2} \nu N_2)^{\otimes 2},\]
  where we have implicitly made use of the $\F_2$-synthetic relation $\wt{2} \wt{\eta} = 0$ (this element lives in $\pi_{1,3} (\Ss^{0,0}) = 0$)\cite[Proposition A.20]{bhs}.

  We claim that the following statements hold:
  \begin{enumerate}
    \item[(i)] the map $\tau \cdot  (1 \otimes \wt{J})(\wt{x}) : \Ss^{2n-1, 2n-1+2M_2} / \wt{2} \to \Sigma^{0,M_2} \nu N_2$ is nullhomotopic.
    \item[(ii)] the composite map
        \[\Ss^{2n-1, 2n+2M_2-4} / \wt{4} \to \Ss^{2n-1, 2n+2M_2-4}/ \wt{2} \xrightarrow{\tau^4 \cdot \wt{x}} (\Sigma^{0,M_2} \nu N_2)^{\otimes 2} \xrightarrow{c} D_2 (\Sigma^{0,M_2} \nu N_2)\]
      is nullhomotopic.
  \end{enumerate}
  Here the synthetic lift $\wt{4}$ of 4 is chosen so that $\wt{2}\cdot\wt{2}$.
  Given (i) and (ii), we may form the diagram
  \begin{center}
  \begin{tikzcd}
    \Ss^{2n-1,2n+2M_2-4} / \wt{4} \ar[dr, "\tau^4 \cdot \wt{x}"] \ar[rr, "0"{name=D}] \ar[ddr, bend right, "0"{name=C} left] & & (\Sigma^{0,M_2} \nu N_2) \otimes \Ss \ar[dr, "\wt{J} \otimes 1"] \\
    & (\Sigma^{0,M_2} \nu N_2)^{\otimes 2}  \ar[Leftrightarrow, from=C, ""] \ar[Leftrightarrow, from=D, ""] \ar[d, "c"] \ar[ur, "1 \otimes \wt{J}"] \ar[rr, "\wt{J} \otimes \wt{J}"{name=E}] \ar[phantom, from=1-3, to=E, "(\otimes)"] & & (\Ss^{0,0})^{\otimes 2} \ar[dl, "c"] \ar[dd, "\simeq"] \\
    & D_2(\Sigma^{0,M_2} \nu N_2) \ar[urr, phantom, "(c)"] \ar[r, "D_2(\wt{J})"] & D_2(\Ss^{0,0}) \ar[dr, "\hat{m}"] \\
    & & & \Ss^{0,0}.   
  \end{tikzcd}
  \end{center}
 The associated class $\Ss^{2n, 2n+2M_2-4} / \wt{4} \to \Ss^{0,0}$ determines a representative for $d_2 (x)$ after inverting $\tau$ and composing with the inclusion of the bottom cell, as desired.

  We begin by proving statement (i). First, we note that the composite
  \[\Ss^{2n-2} \xrightarrow{y_{n-1} \otimes y_{n-1}} (N_2) ^{\otimes 2} \xrightarrow{1 \otimes J} N_2 \to \R \xrightarrow{\pi} \littleo\]
  is null because $\pi_{2n-2} \littleo = 0$.
  It follows from \Cref{lem:D2-lift} that $$\tau \cdot (1 \otimes \wt{J}) (\nu y_{n-1} \otimes \nu y_{n-1}) \in \pi_{2n-2, 2n-2+2M_2-1} (\Sigma^{0,M_2} \nu N_2) \cong \pi_{2n-2, 2n-2+M_2-1} \nu \R$$ lifts to $$\wt{z} \in \pi_{2n-2, 2n-2+M_2-1} \nu D_2 (\R).$$
  In the above, note that the isomorphism $$\pi_{2n-2, 2n-2+2M_2-1} (\Sigma^{0,M_2} \nu N_2) \cong \pi_{2n-2, 2n-2+M_2-1} \nu \R$$ is a consequence of \cite[Proposition 4.16]{inertia} and the fact that $N_2 \to \R$ is a $(2n+1)$-skeleton.

Recall that the natural map $D_2(\R) \to D_2(\littleo)$ induces an isomorphism on $\pi_*$ for $*\leq 3n-4$.   It therefore follows from \Cref{prop:D2o-E2} and the $\tau$-Bockstein spectral sequence that $\pi_{2n-2,2n-2 + k} \nu D_2 (\R) = 0$ for all $k > 0$. It follows that $\wt{z} = 0$ as long as $M_2 = h(n-1) -\lfloor \log_2 (2n+2)\rfloor + 1 > 1$, which holds by the assumption $n \geq 16$.
  As a consequence, we have $\tau \cdot (1 \otimes \wt{J}) (\nu y_{n-1} \otimes \nu y_{n-1}) = 0$. Composing with $\wt{\eta} : \Ss^{2n-1, 2n-1+2M_2} / \wt{2} \to \Ss^{2n-2, 2n-2+2M_2-1}$, we obtain (i).

  We now move on to the proof of (ii). By (i) and the assumption that $(m - 1 \otimes J) (x) = 0$, we have $m (x) = 0$.
  Factoring $m = \hat{m} \circ c$ and applying \Cref{prop:D2-split}, we find that
  $$c (x) = 0 \in \pi_{2n-1} D_2 (N_2).$$

  Using \Cref{D2}, we find that
  \[\pi_{2n-1,2n-2+2M_2} D_2 (\Sigma^{0,M_2} \nu N_2) \to \pi_{2n-1} D_2 (N_2)\]
  is an isomorphism.
  We conclude that the composition
  \[\Ss^{2n-1, 2n+2M_2-2} \to \Ss^{2n-1, 2n+2M_2-2}/ \wt{2} \xrightarrow{\tau^2 \cdot \wt{x}} (\Sigma^{0,M_2} \nu N_2)^{\otimes 2} \xrightarrow{c} D_2 (\Sigma^{0,M_2} \nu N_2)\]
  is nullhomotopic. As a result, we obtain a factorization through the top cell
  \begin{center}
    \begin{tikzcd}[column sep = large]
      \Ss^{2n-1, 2n+2M_2-2}/ \wt{2} \ar[d] \ar[r, "\tau^2 \cdot c(\wt{x})"] &  D_2 (\Sigma^{0,M_2} \nu N_2) \\
      \Ss^{2n, 2n+2M_2-1}. \ar[ur, "\wt{w}"]
    \end{tikzcd}
  \end{center}
  Since
  \[\pi_{2n,2n-2+2M_2} D_2 (\Sigma^{0,M_2} \nu N_2) \to \pi_{2n} D_2 ( N_2) \cong \Z/2\Z\]
  is an isomorphism by \Cref{D2}, we conclude that $\wt{w}$ is $2 \tau = \wt{2} \tau^2$-torsion.

  This implies that the composition
  \[\Ss^{2n-1, 2n+2M_2-4} / \wt{4} \to \Ss^{2n-1, 2n+2M_2-4}/ \wt{2} \xrightarrow{\tau^4 \cdot \wt{x}} (\Sigma^{0,M_2} \nu N_2)^{\otimes 2} \xrightarrow{c} D_2 (\Sigma^{0,M_2} \nu N_2)\]
  is nullhomotopic, as desired.
\end{proof}

\begin{proof}[Proof of \Cref{thm:d2}(2)]
  We assume that $n \equiv 1 \mod 8$.
  Suppose that we are given $x \in \pi_{2n-1} (\R^{\otimes 2}) \cong \Z / 4 \Z$ which satisfies $(m - 1 \otimes J) (x) = 0$.
  It follows from the degeneration of the $\F_2$-Adams spectral sequence for $\R^{\otimes 2}$ in the relevant range (see \Cref{prop:Otens2}) that $x$ lifts uniquely to a $\wt{4}$-torsion class $\wt{x} \in \pi_{2n-1, 2n-1} \nu \R^{\otimes 2}$.
  Since $N_2$ is a $(2n+1)$-skeleton of $R$, it therefore follows from \cite[Proposition 4.16]{inertia} that $x$ lifts to a $\wt{4}$-torsion class in $\pi_{2n-1, 2n-1+2M_2} (\Sigma^{0,M_2} \nu N_2)^{\otimes 2}$, and hence to a map $\wt{x} : \Ss^{2n-1,2n-1+2M_2} / \wt{4} \to (\Sigma^{0,M_2} \nu N_2)^{\otimes 2}$.
  Then we claim that the following statements hold:
  \begin{enumerate}
    \item[(i)] The composite
      \[\Ss^{2n-1, 2n-2 + 2M_2} / \wt{4} \xrightarrow{\tau^2 \cdot \wt{x}} (\Sigma^{0,M_2} \nu N_2)^{\otimes 2} \xrightarrow{1 \otimes \wt{J}} \Sigma^{0,M_2} \nu N_2\]
      is nullhomotopic.
    \item[(ii)] The composite
      \[\Ss^{2n-1, 2n-5+2M_2} / \wt{8} \to \Ss^{2n-1, 2n-5+2M_2} / \wt{4} \xrightarrow{\tau^4 \cdot \wt{x}} (\Sigma^{0,M_2} \nu N_2)^{\otimes 2} \xrightarrow{c} D_2 (\Sigma^{0,M_2} \nu N_2)\]
      is nullhomotopic.
  \end{enumerate}
  As in the proof of \Cref{thm:d2}(1) above, combining these two claims yields the desired result. Indeed, we may then form the diagram
  \begin{center}
  \begin{tikzcd}
    \Ss^{2n-1,2n+2M_2-5} / \wt{8} \ar[dr, "\tau^4 \cdot \wt{x}"] \ar[rr, "0"{name=D}] \ar[ddr, bend right, "0"{name=C} left] & & \R \otimes \Ss \ar[dr, "\wt{J} \otimes 1"] \\
    & (\Sigma^{0,M_2} \nu N_2)^{\otimes 2}  \ar[Leftrightarrow, from=C, ""] \ar[Leftrightarrow, from=D, ""] \ar[d, "c"] \ar[ur, "1 \otimes \wt{J}"] \ar[rr, "\wt{J} \otimes \wt{J}"{name=E}] \ar[phantom, from=1-3, to=E, "(\otimes)"] & & (\Ss^{0,0})^{\otimes 2} \ar[dl, "c"] \ar[dd, "\simeq"] \\
    & D_2(\Sigma^{0,M_2} \nu N_2) \ar[urr, phantom, "(c)"] \ar[r, "D_2(\wt{J})"] & D_2(\Ss^{0,0}) \ar[dr, "\hat{m}"] \\
    & & & \Ss^{0,0},    
  \end{tikzcd}
  \end{center}
  which gives rise to a map $\Ss^{2n,2n+2M_2-5} / \wt{8} \to \Ss^{0,0}$ that yields a representative for $d_2 (x)$ upon inverting $\tau$ and restricting to the bottom cell.

  First, we prove statement (ii).
We start by showing that $\pi_{2n-1} (c) : \pi_{2n-1} \R^{\otimes 2} \to \pi_{2n-1} D_2 (\R)$ is trivial.
  Since the $\F_2$-Adams spectral sequence for $\pi_{2n-1} D_2 (\R)$ is concentrated in filtration zero by \Cref{prop:D2o-E2}, it suffices to show that the composite
  $\pi_{2n-1} \R^{\otimes 2} \to \pi_{2n-1} D_2 (\R) \to \H_{2n-1} (D_2 (\R); \F_2)$
  is zero.
  But it follows from \Cref{prop:Otens2} that the image of the Hurewicz map $\pi_{2n-1} (\R^{\otimes 2}) \to \H_{2n-1} (\R^{\otimes 2} ; \F_2)$ is spanned by $y_{n-1} \otimes y_n + y_n \otimes y_{n-1}$.
  The image of this class in $\H_{2n-1} (D_2 (\R) ; \F_2)$ is $2 y_{n-1} \cdot y_n = 0$.
  Hence $\pi_{2n-1} (c) = 0$ as desired.

 Using the isomorphism   
  \[\pi_{t-s,t} D_2 (\Sigma^{0,M_2} \nu N_2) \to \pi_{t-s} D_2 (N_2)\]
for $t \leq 2n-2+2M_2$ (\Cref{D2}),
  we deduce from the above that the composite
  \[\Ss^{2n-1, 2n-2+2M_2} \to \Ss^{2n-1, 2n-2+2M_2} / \wt{4} \xrightarrow{\tau \cdot \wt{x}} (\Sigma^{0,M_2} \nu N_2)^{\otimes 2} \xrightarrow{c} D_2 (\Sigma^{0,M_2} \nu N_2) \]
  is nullhomotopic. Hence there is a factorization through the top cell
  \begin{center}
    \begin{tikzcd}[column sep = large]
      \Ss^{2n-1, 2n+2M_2-2}/ \wt{4} \ar[d] \ar[r, "\tau \cdot c(\wt{x})"] &  D_2 (\Sigma^{0,M_2} \nu N_2) \\
      \Ss^{2n, 2n+2M_2}. \ar[ur, "\wt{w}"]
    \end{tikzcd}
  \end{center}
  Since $\pi_{2n} D_2 (\R) = \Z/2\Z$ by \Cref{prop:D2o-E2}, it follows from \Cref{D2} that
  \[2 \tau^2 \cdot \wt{w} = \wt{2} \tau^{3} \cdot \wt{w} = 0 \in \pi_{2n, 2n+2M_2-3} D_2 (\Sigma^{0,M_2} \nu N_2).\]
  As a consequence, we conclude that
    \[\Ss^{2n-1, 2n-5+2M_2} / \wt{8} \to \Ss^{2n-1, 2n-5+2M_2} / \wt{4} \xrightarrow{\tau^4 \cdot \wt{x}} (\Sigma^{0,M_2} \nu N_2)^{\otimes 2} \xrightarrow{c} D_2 (\Sigma^{0,M_2} \nu N_2)\]
  is nullhomotopic, as desired.

  Let us now prove statement (i).
  We begin by noting that the composition
  \[\Ss^{2n-1, 2n-1+2M_2} \to \Ss^{2n-1, 2n-1+2M_2} / \wt{4} \xrightarrow{\wt{x}} (\Sigma^{0,M_2} \nu N_2)^{\otimes 2} \xrightarrow{1 \otimes \wt{J}} \Sigma^{0,M_2} \nu N_2\]
  is nullhomotopic after inverting $\tau$.
  Indeed, after inverting $\tau$ this is just $(1 \otimes J) (x)$, which is zero by (ii) and the fact that $(m - 1 \otimes J) (x) = 0$.
  It therefore follows from \Cref{lem:D2-lift} that $\tau$ times the above composition lifts to an element of $\pi_{2n-1, 2n-2+2M_2} \Sigma^{0,M_2} \nu D_2 (N_2) \cong \pi_{2n-1, 2n-2+M_2} \nu D_2 (N_2)$.
  
Note that \Cref{prop:D2o-E2} implies that $\pi_{2n-1, 2n-1+k} \nu D_2 (N_2) = 0$ for all $k > 0$. Since $M_2 > 1$ by the assumption that $n \geq 17$, we conclude that 
  \[\Ss^{2n-1, 2n-2+2M_2} \to \Ss^{2n-1, 2n-2+2M_2} / \wt{4} \xrightarrow{\tau\cdot \wt{x}} (\Sigma^{0,M_2} \nu N_2)^{\otimes 2} \xrightarrow{1 \otimes \wt{J}} \Sigma^{0,M_2} \nu N_2\]
  is nullhomotopic. Therefore we obtain a factorization through the top cell
  \begin{center}
    \begin{tikzcd}[column sep = large]
      \Ss^{2n-1, 2n+2M_2-2}/ \wt{4} \ar[d] \ar[r, "\tau \cdot (1 \otimes \wt{J}) (\wt{x})"] &  \Sigma^{0,M_2} \nu N_2 \\
      \Ss^{2n, 2n+2M_2}. \ar[ur, "\wt{w}"]
    \end{tikzcd}
  \end{center}
  Since the composition
  \[\Ss^{2n, 2n+2M_2} \xrightarrow{\wt{w}} \Sigma^{0,M_2} \nu N_2 \xrightarrow{\nu \pi} \Sigma^{0,M_2} \nu \littleo\]
  is null after inverting $\tau$ because $\pi_{2n} \littleo = 0$,
  it follows from \Cref{lem:D2-lift}  that $\tau\cdot \wt{w}$ lifts to an element of $\pi_{2n,2n+2M_2-1} \Sigma^{0,M_2} \nu D_2 (N_2) \cong \pi_{2n,2n+M_2-1} \nu D_2 (N_2)$.
  From \Cref{prop:D2o-E2} we know that $\pi_{2n, 2n +k} \nu D_2 (N_2) = 0$ for all $k > 1$, so $\tau \cdot \wt{w} = 0${} as soon as $M_2 > 2$, which holds by the assumption that $n \geq 17$.

  In conclusion, we have found that
  \[\tau^{2} \cdot (1 \otimes \wt{J}) (\wt{x}) : \Ss^{2n-1, 2n+2M_2-3} / \wt{4} \to \Sigma^{0,M_2} \nu N_2\]
  is nullhomotopic, as desired.
\end{proof}

\section{Applications of vanishing lines} \label{sec:vanishing}
In this section, we combine the results of \Cref{sec:lower} with Chang's improved $v_1$-banded vanishing line for the mod 2 Moore spectrum to prove \Cref{thm:ker-main} when $n \equiv 0 ,1 \mod 8$ and $n \geq 25$. We will deal with the remaining cases in \Cref{sec:exceptional}.
%
%

We begin by recalling how $v_1$-banded vanishing lines help us prove that classes lie in the image of $J$.
In the first place, we recall the definition of a $v_1$-banded vanishing line from \cite[Section 13]{bhs}.

%

 \begin{definition}
   Let $X$ denote a synthetic spectrum.
   Given a real number $a$, we write $F^a \pi_k (\tau^{-1} X)$ for the image of $\pi_{k, k+\lceil a \rceil} X \to \pi_k (\tau^{-1} X)$.
   We say that $X$ has a \textit{$v_1$-banded vanishing line with parameters} $(b\leq d, v,m,c,r)$ with $m<\frac{1}{2}$ if the following conditions hold:
   \begin{enumerate}
     \item every $\tau$-power torsion class in $\pi_{k, k+s} (X)$ is $\tau^r$-torsion for $s \geq mk+c$ nd $k\geq v$,
     \item the natural map
       \[F^{\frac{1}{2}k+b} \pi_k (\tau^{-1} X) \to F^{mk+c} \pi_k (\tau^{-1} X)\]
       is an isomorphism for $k \geq v$,
     \item the composite
       \[F^{\frac{1}{2} k + b} \pi_k (\tau^{-1} X) \to \pi_k (\tau^{-1} X) \to \pi_k (L_{K(1)} \tau^{-1} X)\]
       is an equivalence for $k \geq v$,
     \item $\pi_{k,k+s} (X) = 0$ for $s > \frac{1}{2} k + d$.
   \end{enumerate}
%
%
%
\end{definition}

We refer the reader to \cite[Section 13]{bhs} for more details on this complicated definition.
However, we note that all we will use of this definition in this work is part (2).

\begin{prop} \label{prop:AF-J}
  Suppose that $\Ss^{0,0} / \wt{2^l}$ admits a vanishing line with parameters $(b \leq d,v,m,c,r)$.
  Then, given a map $\Ss^{k,k+s} / \wt{2^l} \to \Ss^{0,0}$, the element of $\pi_k \Ss$ obtained by inverting $\tau$ and restricting to the bottom cell lies in the subgroup of $\pi_{k} \Ss$ generated by the image of $J$ and the $\mu$-family\footnote{See \cite[Theorem 1.2]{AdamsJIV} for the definition of the $\mu$-family.} if we assume:
  \begin{enumerate}
    \item $s + l - 1 \geq m(k+1) + c$.
    \item $k + 1 \geq v$.
    \item $\frac{1}{2}(k+1)+b-l+1 \geq \frac{3}{10} k + 4 + v_2(k+2) + v_2 (k+1)$.
  \end{enumerate}
  Here $v_2 (n)$ is the $2$-adic valuation of $n$.
\end{prop}

\begin{proof}
  Under Spanier--Whitehead duality, maps $\Ss^{k,k+s} / \wt{2^l} \to \Ss^{0,0}$ correspond to elements of $\pi_{k+1,k+s+l} (\Ss^{0,0} / \wt{2^l})$.
  Using this correspondence, part (2) of the definition of a $v_1$-banded vanishing line and assumptions (1) and (2) above imply that the underlying map of $\Ss^{k,k+s} / \wt{2^l} \to \Ss^{0,0}$ is represented by a map $\Ss^{k, k +\lceil \frac{1}{2}(k+1)+b-l+1\rceil} / \wt{2^l} \to \Ss^{0,0}$.
  Restricting to the bottom cell, we obtain an element of $\pi_{k, k + \lceil \frac{1}{2}(k+1) +b-l+1 \rceil} \Ss^{0,0}$.
  In particular, the underlying element in $\pi_k \Ss$ is of $\F_2$-Adams filtration at least $\frac{1}{2}(k+1) +b-l+1$.
  Davis and Mahowald have proven in \cite{DM3} that any element of $\pi_k \Ss$ which has $\F_2$-Adams filtration at least $\frac{3}{10} k + 4 + v_2(k+2) + v_2 (k+1)$, where $v_2$ is the $2$-adic valuation, lies in the subgroup of $\pi_k (\Ss)$ generated by the image of $J$ and the $\mu$-family, so we are done by assumption (3).
\end{proof}
%
%
%

\begin{rmk}\label{rmk:ignore-mu}
  When applying the above proposition to elements in the kernel of $\pi_{2n} \Ss \to \pi_{2n} \MOn$ with $n \geq 4$, we may ignore the $\mu$-family.
  This is because the Atiyah--Bott--Shapiro orientation \cite{ABS}
  \[\mathrm{MO}\langle4\rangle = \mathrm{MSpin} \to \ko\]
  kills the image of $J$ but is injective on the $\mu$-family, so elements in the kernel of $\pi_{2n} \Ss \to \pi_{2n} \MOn$ which lie in the subgroup of $\pi_{2n} (\Ss)$ generated by the image of $J$ and the $\mu$-family must in fact lie in the image of $J$.
\end{rmk}

To apply \Cref{prop:AF-J}, we require $v_1$-banded vanishing lines on synthetic Moore spectra with explicit parameters.
Certain parameters for such vanishing lines are given in \cite[Proposition 15.8]{bhs}.
However, we require more optimized parameters than the ones proven there.
We will therefore make use of the following theorem of Chang, which gives improved parameters for a $v_1$-banded vanishing line on $\Ss^{0,0} / \wt{2}$.
%

\begin{theorem}[{{\cite[Theorem 4]{chang}\label{chang}}}]
  The synthetic spectrum $\Ss^{0,0}/ \wt{2}$ admits a $v_1$-banded vanishing line with parameters
  $$(b\leq d, v,m,c,r)=(-1.5\leq 1, 25,\frac{1}{5},5,3).$$
\end{theorem}

Applying \cite[Lemmas 13.10 and Proposition 13.11]{bhs} to the cofiber sequences 
$$\Sigma^{0,1}\Ss^{0,0}/ \wt{2}\rightarrow \Ss^{0,0}/ \wt{4}\rightarrow \Ss^{0,0}/ \wt{2},\,\,\,\Sigma^{0,1}\Ss^{0,0}/ \wt{4}\rightarrow \Ss^{0,0}/ \wt{8}\rightarrow \Ss^{0,0}/ \wt{2},$$
we obtain the following proposition, which gives improved parameters for $v_1$-banded vanishing lines on $\Ss^{0,0}/ \wt{4}$ and $\Ss^{0,0} / \wt{8}$.
 
\begin{proposition}\label{parameter}
  The synthetic spectrum $\Ss^{0,0} / \wt{4}$ admits a $v_1$-banded vanishing line with parameters
  $$(b\leq d, v,m,c,r)=(-4.5\leq 2, 45,\frac{1}{5},9,6),$$ and
  the synthetic spectrum $\Ss^{0,0}/\wt{8}$ admits a $v_1$-banded vanishing line with parameters
  $$(b\leq d, v,m,c,r)=(-7.5\leq 3, 68+\frac{1}{3}, \frac{1}{5},13,10).$$ 
\end{proposition}

Let's now determine conditions under which hypothesis (3) of \Cref{prop:AF-J} holds for the above vanishing lines.

\begin{prop}\label{prop:bounds-explicit}
  For the vanishing lines of \Cref{chang} and \Cref{parameter}, hypothesis (3) of \Cref{prop:AF-J} holds under the following conditions:
  \begin{itemize}
    \item When $l = 1$ and $k \geq 52$.
    \item When $l = 2$ and $k \geq 78$.
    \item When $l = 3$ and $k \geq 98$. \NB{Double check these three bounds}
  \end{itemize}
\end{prop}

\begin{proof}
  We leave this elementary but tedious verification to the reader.
\end{proof}

Finally, we are able to prove that elements of $\pi_{2n} \Ss$ which lie in the image of the $d_1$ and $d_2$-differentials in the bar spectral sequence for $\MOn$ must lie in the image of $J$, outside of a small number of exceptional cases which will be addressed in \Cref{sec:exceptional}.
For the $d_2$-differential, this was proven in \Cref{thm:d2}(3) for $n \equiv 4 \mod 8$.
The remaining non-exceptional cases are addressed by the four propositions below.

\begin{proposition}\label{exceptionald1}
  Suppose that $n \equiv 0,1,4 \mod 8$. Then the image of the map $J : \pi_{2n} \R \to \pi_{2n} \Ss$ is contained in the classical image of $J$ whenever $n \geq 28$. \todo{Remark somewhere earlier about $J$ and image of $J$.}
\end{proposition}
\begin{proof}
\NB{check all numbers}
  By \Cref{thm:d1}, any element in the image of $J : \pi_{2n} \R \to \pi_{2n} \Ss$ is represented modulo the classical image of $J$ by the a synthetic map $\Ss^{2n, 2n+2M_1-3} / \wt{2} \to \Ss^{0,0}$.

  Applying \Cref{prop:AF-J}, \Cref{chang} and \Cref{prop:bounds-explicit}, we find that the desired result holds whenever $2n \geq 52$ and
  \[2M_1-3=2h(n-1)-2\lfloor \log_2(n+3)\rfloor-1 \geq \frac{1}{5}(2n+1) +5.\]
  Using elementary arguments, one may show that this holds whenever $n \geq 26$. \NB{Check this} In \Cref{mod2d1}, we include some examples of values of either side of this inequality.
\end{proof}

\begin{table}[h!]
\begin{center}
    \begin{tabular}{|c|c|c|c|c|c|c|c|c|}
    \hline
        $n$ & 25 & 26 & 27&28 &29  &30&31&32\\
        \hline
        $2M_1-3$ & 15 &17 &19 & 19 &19&19&19&21\\
        \hline
        $0.4 n+5.2$ & 15.2 &15.6&16 & 16.4 &16.8 & 17.2  &17.6&18\\
        \hline
    \end{tabular}
  \end{center}
  \caption{}
  \label{mod2d1}
\end{table}

\begin{proposition} \label{prop:d1-small}
  Suppose that $n = 16,17,20,24,25$. Then the image of the map $J : \pi_{2n} \R \to \pi_{2n} \Ss$ in contained in the classical image of $J$.
\end{proposition}
\begin{proof}
  When $n = 16,17,20,24,25$, we have $2M_1 - 3 = 7,9,13,13,15$, respectively.
  We may then read off directly from \cite[p. 8]{Isaksen} that any element of $\pi_{2n} \Ss$ which is of $\F_2$-Adams filtration at least $2M_1-3$ in these degrees must lie in the subgroup generated by the image of $J$ and the $\mu$-family. \todo{The relevant sequence is $6,8,12,12,6$}
\end{proof}

\begin{proposition}
  \
  \begin{enumerate}
    \item If $n\equiv 0 \mod 8$, then any element of $\pi_{2n}(\Ss)$ which is killed by a $d_2$-differential in the bar spectral sequence for $\MOn$ lies in the image of $J$ when $n \geq 48$.
    \item If $n\equiv 1 \mod 8$, then any element of $\pi_{2n}(\Ss)$ which is killed by a $d_2$-differential in the bar spectral sequence for $\MOn$ lies in the image of $J$ when $n \geq 49$.
  \end{enumerate}
\end{proposition}
\begin{proof}
  \NB{check all numbers}
  We begin with the proof of (1).
  By \Cref{thm:d2}(1), we know that any element of $\pi_{2n} (\Ss)$ killed by a $d_2$-differential is represented by a synthetic map of the form
  \[\Ss^{2n, 2n+2M_2-4} / \wt{4} \to \Ss^{0,0}.\]

  Applying \Cref{prop:AF-J}, \Cref{chang} and \Cref{prop:bounds-explicit}, we find that the result holds whenever $2n \geq 78$ and the inequality
  \[2M_2-4=2h(n-1)-2\lfloor \log_2(2n+2)\rfloor-2 \geq \frac{1}{5}(2n+1) +9.\]
  holds.
  An elementary argument shows that this is true for $n\equiv 0 \mod 8$ whenever $n \geq 48$, so we are done.

  We now move on to a proof of (2).
  By \Cref{thm:d2}(2), we know that any element of $\pi_{2n} (\Ss)$ killed by a $d_2$-differential is represented by a synthetic map of the form
  \[\Ss^{2n, 2n+2M_2-5} / \wt{8} \to \Ss^{0,0}.\]

  Applying \Cref{prop:AF-J}, \Cref{chang} and \Cref{prop:bounds-explicit}, we find that the result holds whenever $n \geq 98$ and the inequality
  \[2M_2-5=2h(n-1)-2\lfloor \log_2(2n+2)\rfloor-3 \geq \frac{1}{5}(2n+1) +13.\]
  holds.
  An elementary argument shows that this is true for $n \equiv 1\mod 8$ whenever $n \geq 49$, so we are done.
%
%
%
%
\end{proof}

\begin{prop}
  Suppose that $n = 25, 32,33,40,41$.
  Then any element of $\pi_{2n}(\Ss)$ which is killed by a $d_2$-differential in the bar spectral sequence for $\MOn$ lies in the image of $J$.
\end{prop}
\begin{proof}
  \NB{Check numbers}
  It follows from \Cref{thm:d2} that any element of $\pi_{2n} (\Ss)$ which is killed by such a $d_2$-differential must have $\F_2$-Adams filtration at least $11, 16,17, 24,25$ when $n = 25,32,33,40,41$, respectively.

  On the other hand, we may read off from \cite[p. 8]{Isaksen} that any element of $\pi_{2n} (\Ss)$ which has at least this $\F_2$-Adams filtration must lie in the image of $J$ or the $\mu$-family, as desired.
\end{proof}

\section{Exceptional cases} \label{sec:exceptional}
In this section, we analyze the remaining cases  $n=1,2,3,4,5,6,7,8,9,16,17,24$.

\begin{thm} \label{thm:exceptional}
  Suppose that $n = 1,2,3,4,5,6,7,8,9,10,16,17,24$. Then the kernel of
  \[\pi_{2n} (\Ss) \to \pi_{2n} \MOn\]
  is generated by the image of $J$ unless $n=3,4,7,8,9$, in which cases we have:
  \begin{itemize}
    \item When $n=1$, it is generated by the image of $J$ and $\eta^2 \in \pi_2 \Ss$.
    \item When $n=3$, it is generated by the image of $J$ and $\nu^2 \in \pi_6 \Ss$.
    \item When $n=4$, it is generated by the image of $J$ and $\varepsilon \in \pi_8 \Ss$.
    \item When $n=7$, it is generated by the image of $J$ and $\sigma^2 \in \pi_{14} \Ss$.
    \item When $n=8$, it is generated by the image of $J$ and $\eta_4 \in \pi_{16} \Ss$.
    \item When $n=9$, it is generated by the image of $J$ and $[h_2 h_4] \in \pi_{18} \Ss$.\footnote{The symbol $[h_2 h_4]$ does not unambiguously determine a class in $\pi_{18} \Ss$, even up to multiplication by a $2$-adic unit. To be precise, here we choose a representative of $[h_2 h_4]$ that lies in the kernel of the unit map $\pi_{18} \Ss \to \pi_{18} \ko \cong \F_2$. This kernel is isomorphic to $\Z/8\Z$, so this determines an element uniquely up to multiplication by a $2$-adic unit.}
  \end{itemize}
\end{thm}

Combining \Cref{thm:exceptional} with \Cref{thm:reduction} and the results of \Cref{sec:lower}, we obtain a proof of \Cref{thm:ker-main}.
In the remainder of this section, we will prove \Cref{thm:exceptional} case-by-case.

\subsection{The case $n=1$}
In this case, the map $\pi_2 \Ss = \Z/2\Z\{\eta^2\} \to \pi_2 \mo \langle 1 \rangle = \pi_2 \MO$ is zero due to the fact that $\MO$ is an associative $\mathbb{F}_2$-algebra and hence the unit map $\Ss \to \MO$ factors as $\Ss \to \F_2 \to \MO$.

\subsection{The case $n=2$}
This case is trivial because $\pi_4 \Ss = 0$.
 
\subsection{The case $n=3$}
It follows from \cite[Corollary 2.7]{spin} that $\pi_6 (\Ss) = \Z/2\Z \{\nu^2\} \to \pi_6 \mo \langle 3 \rangle = 0$ is zero, since the elements $[M^8]^r \times \alpha^i$ appearing in that corollary are concentrated in degrees $\equiv 1,2 \mod 8$.

\subsection{The case $n=4$}
Similarly, it follows from \cite[Corollary 2.7]{spin} that $\pi_8 (\Ss) = \Z/2\Z \{\eta\sigma \} \oplus \Z/2\Z \{\varepsilon\} \to \pi_{8} \mo \langle 4 \rangle$ is zero, which is the desired result.

\subsection{The case $n=5$}
The group $\pi_{10} \Ss$ is generated by an element of the $\mu$-family. In particular, the unit map $\pi_{10} \Ss \to \pi_{10} \ko$ is injective.
It therefore follows from the Atiyah--Bott--Shapiro orientation \cite{ABS} $\MSpin \to \ko$ that the map
\[\pi_{10} \Ss \to \pi_{10} \mo \langle 5 \rangle = \pi_{10} \MString\]
is injective.

\subsection{The case $n=6$}
This case is trivial because $\pi_{12} \Ss = 0$.

\subsection{The case $n=7$}
First, we note that it is easy to see that $\sigma \in \pi_7 \Ss$ is in the kernel of $\pi_7 \Ss \to \pi_7 \mo \langle 7 \rangle$ using the bar spectral sequence and the fact that $\sigma$ is in the image of $J$.
It follows that $\sigma^2$ lies in the kernel of $\pi_{14} \Ss \to \pi_{14} \mo \langle 7 \rangle$.
It therefore suffices to prove that the map
\[\pi_{14} (\Ss) =  \Z/2\Z\{\sigma^2\} \oplus \Z/2\Z\{\kappa\} \to \pi_{14} \mo \langle 7 \rangle\]
is nonzero on $\kappa$.
Using the Ando--Hopkins--Rezk orientation \cite{AHR}
\[\mo \langle 7 \rangle \simeq \MString \to \tmf,\]
it suffices to note that $\kappa$ is mapped to a nonzero class under the map $\pi_{14} \Ss \to \pi_{14} \tmf$ by \cite[Theorem 11.80]{tmfAdamsBook}.

\subsection{The case $n=8$}
It follows from \cite{Giambalvo08,Giambalvo08corrected} that $\pi_{16} \mo \langle 8 \rangle = \pi_{16} \MString = \Z \oplus \Z$.
Since this is torsionfree, the map $\pi_{16} (\Ss) = \Z/2\Z \{[P c_0]\} \oplus \Z/2\Z \{\eta_4\} \to \pi_{16} \mo \langle 8 \rangle$ is zero, which is the desired result.

\subsection{The case $n=9$}\label{mo9}
Using the Atiyah--Bott-Shapiro orientation
\[\mo \langle 9 \rangle \to \MSpin \to \ko,\]
we conclude that the kernel of $\pi_{18} \Ss \to \pi_{18} \mo \langle 9 \rangle$ must be contained in the kernel of $\pi_{18} \Ss \to \pi_{18} \ko \cong \Z/2\Z$.
This kernel is a $\Z/8\Z$ generated by a class $[h_2 h_4]$, so it suffices to show that the kernel of $\pi_{18} \Ss \to \pi_{18} \mo \langle 9 \rangle$ contains a class detected on the $\mathrm{E}_2$-page of the $\F_2$-Adams spectral sequence by $h_2 h_4$.

Recall from \Cref{sec:background} that we have a diagram
\begin{center}
    \begin{tikzcd}
    & & A[9] \ar[d, "b"]&\\
      \Ss\ar[r] & \mo\langle 9 \rangle \ar[r] & \mo\langle 9 \rangle/\Ss \ar[r,"\delta"] & \Ss^1
    \end{tikzcd},
\end{center}
and that the image of $\pi_n A[9] \to \pi_{n-1} \Ss$ generates the kernel of $\pi_{n-1} \Ss \to \pi_{n-1} \mo \langle 9 \rangle$ modulo the image of $J$.
Moreover, we have $\pi_{18} A[9] \cong \Z/2\Z$ and $\pi_{19} A[9] \cong \Z/8\Z$ by \cite[Theorem A]{stolz}.
Before we proceed, we need the following lemma.

\begin{lem} \label{lem:eta-mult}
  The multiplication map $\eta : \Z/2\Z \cong \pi_{18} A[9] \to \pi_{19} A[9] \cong \Z/8\Z$ is nontrivial.
\end{lem}

\begin{proof}
  Since these are the first two nontrivial homotopy groups of $A[9]$, it suffices by consideration of the Atyah--Hirzubruch spectral sequence for stable homotopy to show that the integral Hurewicz map $\pi_{19} A[9] \to \H_{19} (A[9];\mathbb{Z})$ is not injective.
  But this follows from \cite[Theorem 5]{maymilgram}, since the the $\mathrm{E}_2$-page of the $\F_2$-Adams spectral sequence for $A[9]$ consists in total degree $19$ of an $h_0$-tower, which is killed (other than the bottom three dots) by an entering $d_2$-differential by \cite[p. 96]{stolz}:
\begin{center}
    \begin{tikzcd}[column sep=0.4cm, row sep=0.4cm]
  &&  \ &\ &\\
   s=4& &\bullet\ar[u]&\bullet\ar[u]&\\
    s=3&&\bullet\ar[u,dash]&\bullet\ar[u,dash]&&\\
    s=2 &&\bullet \ar[u,dash]&\bullet\ar[u,dash]\ar[uul,dashed]&\\
    s=1 &&\bullet \ar[u,dash]&\bullet\ar[u,dash]\ar[uul,dashed]&\\
   s=0 &\bullet&\bullet \ar[u,dash]&\bullet\\
     t-s& 18&19 &20
\end{tikzcd}.
\end{center}
  Indeed, \cite[Theorem 5]{maymilgram} then implies that $\H_{19} (A[9];\mathbb{Z}_{(2)})$ is $4$-torsion, so that the Hurewicz map from $\pi_{19} A[9] \cong \Z/8\Z$ cannot be injective.
\end{proof}

Note that  the image of a generator under $\pi_{18} A[9] \to \pi_{17} \Ss$ is detected by $h_1 ^2 h_4$ by \cite[Theorem 6.1]{geography}.
It therefore follows from \Cref{lem:eta-mult} that the image of $4$ times a generator under $\pi_{19} A[9] \to \pi_{18} \Ss$ is detected by $h_1^3 h_4 = h_0^2 h_2 h_4$.
Hence the image of a generator must be detected by $h_2 h_4$, as desired.

\subsection{The case $n=10$}
Using the Ando--Hopkins--Rezk orientation \cite{AHR}
\[\mo \langle 8 \rangle \to \tmf,\]
it suffices to note that the unit map
\[\pi_{20} \Ss = \Z/8\Z \{\overline{\kappa}\} \to \pi_{20} \tmf\]
is injective by \cite[Theorem 11.80]{tmfAdamsBook}.

\subsection{The cases $n=16,17,24$}
We will prove that the $d_1$ and $d_2$-differentials in the bar spectral sequence for $\MOn$ only kill the image of $J$ in $\pi_{2n} \Ss$.
For $d_1$, this was already proven in \Cref{prop:d1-small}.
We begin by proving the following proposition, which tells us exactly what $\F_2$-Adams filtration bounds we need to prove.

\begin{prop} \label{prop:upper-excep}
  Let $x \in \pi_{2n} \Ss$ be in the kernel of the unit map $\pi_{2n} \Ss \to \pi_{2n} \MOn$.
  \begin{itemize}
    \item If $n = 16$ and $x$ is of $\F_2$-Adams filtration at least $5$, then $x$ lies in the image of $J$.
    \item If $n = 17$ and $x$ is of $\F_2$-Adams filtration at least $7$, then $x$ lies in the image of $J$.
    \item If $n = 24$ and $x$ is of $\F_2$-Adams filtration at least $13$, then $x$ lies in the image of $J$.
  \end{itemize}
\end{prop}

\begin{proof}
  As in \Cref{rmk:ignore-mu}, we are free to ignore the difference between the image of $J$ and the subgroup of $\pi_{2n} \Ss$ generated by the image of $J$ and the $\mu$ family.
  Using the Ando--Hopkins--Rezk orientation \cite{AHR}
  \[\mo \langle 8 \rangle = \MString \to \tmf,\]
  it suffices to prove the result for elements in the kernel of $\pi_{2n} \Ss \to \pi_{2n} \tmf$.

  When $n=24$, this is an immediate consequence of \cite[p. 8]{Isaksen}.
  When $n = 17$, the unique element on the $\mathrm{E}_\infty$-page of the $\F_2$-Adams spectral sequence for $\pi_{34} \Ss$ which is of filtration at least $7$ and does not detect an element of the image of $J$ and or the $\mu$-family is $d_0 g$.
  The image of this class in the $\mathrm{E}_\infty$-page of the $\F_2$-Adams spectral sequence for $\tmf$ is nonzero \cite[Figure 5.2]{tmfAdamsBook}, so we are done.

  Now, let's suppose that $n=16$.
  The $n=16$ case is similar to the $n=17$ case, though slightly complicated by the presence of a filtration jump in the map $\pi_{32} \Ss \to \pi_{32} \tmf$.
  There is a unique element on the $\mathrm{E}_\infty$-page of the $\F_2$-Adams spectral sequence for $\pi_{32} \Ss$ which is of filtration at least $5$ and does not detect an element of the image of $J$.
  Following \cite[Table 1.1]{tmfAdamsBook}, we denote this class by $q$.
  (It is denoted $\Delta h_1 h_3$ in \cite[p. 8]{Isaksen}.)
  Since $\pi_{32} \Ss \to \pi_{32} \tmf$ kills the image of $J$, it suffices to show that some class in $\pi_{32} \Ss$ representing $q$ maps to a nonzero class in $\pi_{32} \tmf$.
  But this follows from \cite[Theorem 11.80]{tmfAdamsBook}.
\end{proof}

\begin{lemma}\label{lem:lower-excep}
  Suppose that $n \equiv 0,1 \mod 8$ and $n \geq 16$.
  Then any class in $\pi_{2n} \Ss$ which is killed by a $d_2$-differential in the bar spectral sequence for $\MOn$ is represented by a class of $\F_2$-Adams filtration at least $2M_2-1$.
\end{lemma}

Since $2M_2-1 = 5,7,13$ when $n = 16,17,24$, the combination of \Cref{prop:upper-excep} and \Cref{lem:lower-excep} completes the proof of \Cref{thm:exceptional} in these cases.
%

\begin{proof}[Proof of \Cref{lem:lower-excep}]
  Let $N \xrightarrow{\iota}N_2\rightarrow \R$ denote the inclusion of an $n$-skeleton into a $(2n+1)$-skeleton
into $\R$.
  Examining Figure (\ref{littleocelldiagram}) on page 9, we see that $N$ is also an $(n+1)$-skeleton of $\R$.
Hence $\pi_{2n-1} N^{\otimes 2} \to \pi_{2n-1} \R^{\otimes 2}$ and $\pi_{2n-1} D_2 (N) \to \pi_{2n-1} D_2 (\R)$ are isomorphisms.

  Suppose that we are given a class $x \in \pi_{2n-1} \R^{\otimes 2}$ which lies in the kernel of $d_1$, i.e. it satisfies $(m - 1 \otimes J) (x) = 0$.
  Then we may view $x$ as a class in $\pi_{2n-1} N^{\otimes 2}$, and it follows from the proof of \Cref{thm:d2} that $c(x) = 0 \in \pi_{2n-1} D_2 (N)$ and $(\iota \otimes J)(x) = 0 \in \pi_{2n-1} N_2$.
  We may therefore refine the diagram of \Cref{prop:todabracket} to the diagram below.
  \begin{center}
    \begin{tikzcd}[column sep = large]
    \Ss^{2n-1} \ar[dr, "x"] \ar[rr, "0"{name=D}] \ar[ddr, bend right, "0"{name=C} left] & & N_2 \otimes \Ss \ar[dr, "J \otimes 1"] \\
    & N^{\otimes 2}  \ar[Leftrightarrow, from=C, ""] \ar[Leftrightarrow, from=D, ""] \ar[d, "c"] \ar[ur, "\iota \otimes J"] \ar[rr, "J \otimes J"{name=E}] \ar[phantom, from=1-3, to=E, "(\otimes)"] & & \Ss^{\otimes 2} \ar[dl, "c"] \ar[dd, "\simeq"] \\
    & D_2(N) \ar[urr, phantom, "(c)"] \ar[r, "D_2(J)"] & D_2(\Ss) \ar[dr, "\hat{m}"] \\
    & & & \Ss,    
  \end{tikzcd}
  \end{center}

  By \cite[Lemma 10.18]{bhs}, the map  $J:N_2 \rightarrow\Ss^{0,0}$ lifts to a map
  $\wt{J}:\Sigma ^{0,M_2}\nu N_2 \rightarrow \Ss^{0,0}$, where $M_2=h(n-1)-\lfloor\log_2(2n+2)\rfloor+1$, and the composite $N \xrightarrow{\iota} N_2 \xrightarrow{J} \Ss^{0,0}$ lifts to a map $\wt{J} /\tau : \Sigma^{0,M_2+1} N \to \Ss^{0,0}$, since $M_2+1 = h(n-1) -\lfloor \log_2 (n+1) \rfloor +1$.
  In fact, examining the proof of \cite[Lemma 10.18]{bhs}, we find that the synthetic lifts produced by this lemma satisfy $\tau \cdot \wt{J}/ \tau = \wt{J} \circ \nu \iota$.

  Letting $\wt{x} \in \pi_{2n-1, 2n+2M_2} (\Sigma^{0,M_2} \nu N \otimes \Sigma^{0,M_2+1} \nu N)$ denote a shift of $\nu x$, we can therefore produce the following synthetic lift of the above diagram:
  \begin{center}
    \begin{tikzcd}
      \Ss^{2n-1,2n-1+2M_2} \ar[dr, "\tau\cdot \wt{x}"] \ar[rr, "0"{name=D}] \ar[ddr, bend right, "0"{name=C} left] & & \Sigma^{0,M_2} \nu N_2 \otimes \Ss^{0,0} \ar[dr, "\wt{J} \otimes 1"] \\
      & \Sigma^{0,M_2} \nu N \otimes \Sigma^{0,M_2+1} \nu N  \ar[Leftrightarrow, from=C, "?"] \ar[Leftrightarrow, from=D, "?"] \ar[d, "c\circ (1 \otimes \tau)"] \ar[ur, "\nu \iota \otimes \wt{J}/\tau"] \ar[rr, "\wt{J} \otimes \wt{J}/\tau"{name=E}] \ar[phantom, from=1-3, to=E, "(\otimes)"] & & (\Ss^{0,0})^{\otimes 2} \ar[dl, "c"] \ar[dd, "\simeq"] \\
      & D_2(\Sigma^{0,M_2+1} \nu N) \ar[urr, phantom, "(c)"] \ar[r, "D_2(\wt{J}/\tau)"] & D_2(\Ss^{0,0}) \ar[dr, "\hat{m}"] \\
      & & & \Ss^{0,0}.    
  \end{tikzcd}
  \end{center}

  To complete the proof, it suffices to show that the nullhomotopies labeled with a question mark exist.
%
  On the one hand, the composite $(c \circ (1 \otimes \tau))(\tau \cdot \wt{x}) \in \pi_{2n-1, 2n-1+2M_2} D_2 (\Sigma^{0,M_2+1} \nu N)$ is zero because it is zero after inverting $\tau$ and the map $\pi_{2n-1, 2n-1+2M_2} D_2 (\Sigma^{0,M_2+1} \nu N) \to \pi_{2n-1} D_2 (N)$ is an isomorphism by \Cref{D2}.

  On the other hand, since the composite $(\nu \iota \otimes \wt{J} /\tau) (\wt{x}) \in \pi_{2n-1, 2n+2M_2} \Sigma^{0,M_2} \nu N_2 \otimes \Ss^{0,0}$ is zero after inverting $\tau$, it follows from \Cref{lem:D2-lift} that $(\nu \iota \otimes \wt{J} /\tau) (\tau \cdot \wt{x})$ lifts to an element of $\pi_{2n-1, 2n-1+2M_2} \Sigma^{0,M_2} \nu D_2 (N_2) \cong \pi_{2n-1,2n-1+M_2} \nu D_2 (N_2)$.
  It follows from \Cref{prop:D2o-E2} (combined with \cite[Corollary 9.20]{bhs}) that this group is equal to zero, since $M_2 \geq 1$ in the range $n \geq 16$.
\end{proof}

\section{Determination of inertia groups} \label{sec:det}
In this section, we determine the conditions under which an $(n-1)$-connected $2n$-manifold can have non-trivial inertia group when $n=4,8,9$.
The authors would like to express their thanks to Diarmuid Crowley for clarifications and corrections regarding the application of modified surgery and his $Q$-form conjecture in this section.

\begin{thm}[Theorem \ref{determination}] \label{thm:det-main}
  Let $M$ denote an $(n-1)$-connected, smooth, closed, oriented $2n$-manifold. Then the inertia group $I(M)$ of $M$ is equal to zero unless $n = 4,8,9$, in which case it is determined as follows:
  \begin{itemize}
    \item (Crowley--Nagy \cite{crowleynagy}) A $3$-connected $8$-manifold $M_8$ has inertia group
      \[I(M_8) = \begin{cases} 0 &\text{ if } 8 \mid p_1 \in \H^4 (M_8)\\ \Theta_8 \cong \Z/2\Z &\text{ if } 8 \nmid p_1 \in \H^4 (M_8). \end{cases}\]
    \item (Crowley--Olbermann) A $7$-connected $16$-manifold $M_{16}$ has inertia group
      \[I(M_{16}) = \begin{cases} 0 &\text{ if } 24 \mid p_2 \in \H^8 (M_{16})\\ \Theta_{16} \cong \Z/2\Z &\text{ if } 24 \nmid p_2 \in \H^8 (M_{16}). \end{cases}\]
    \item An $8$-connected $18$-manifold $M_{18}$ has inertia group
      \[I(M_{18}) = \begin{cases} 0 &\text{ if } H(M_{18}) \textit{ is zero;}   \\  \Z/8\Z \cong \mathrm{bSpin}_{19} \subset \Theta_{18} &\text{ otherwise}, \end{cases}\]
        where $\mathrm{bSpin}_{19}$ is the subgroup of $\Theta_{18}$ consisting of those homotopy $18$-spheres which bound a spin $19$-manifold and $H(M_{18}) \subset \pi_{9} \mathrm{BSO}$ is the image of the map $\pi_{9}(M_{18})\rightarrow \pi_9(\mathrm{BSO})\cong \Z/2\Z$ induced by the map classifying the stable normal bundle of $M_{18}$.
%
  \end{itemize}
\end{thm}

To prove Theorem \ref{determination}, we use Kreck's modified surgery \cite{kreck} to reduce it to a computation in a bordism group, which we then carry out.
When $n = 4, 8$, we will make crucial use of a special case of Crowley's $Q$-form conjecture, established by Nagy in his PhD thesis \cite{nagy, nagyarxiv}.
The main content of this conjecture is to give a condition under which a modified surgery obstruction vanishes.

Crowley and Nagy have announced a complete determination of the inertia groups of $3$-connected $8$-manifolds \cite{crowleynagy} using the $Q$-form conjecture, while Crowley and Olbermann have independent work  determining the inertia groups of all $7$-connected $16$-manifolds. Since neither of the results are published to the knowledge of the authors, we record complete proofs for both cases.

We begin by stating the theorem which allows us to reduce the proof of \Cref{thm:det-main} to a bordism calculation. First, we require some definitions.

\begin{dfn}
  Let $H \subseteq \pi_n (\BO)$ denote a subgroup, and $\pi_H : \BOn \to K(\pi_n (\BO) / H, n)$ the unique map which is equal to the quotient map $\pi_n (\BO) \to \pi_n (\BO) / H$ on $\pi_n$.
  Then we define
  \[\BO \langle n \rangle_H \coloneqq \fib(\BO \langle n \rangle \xrightarrow{\pi_H} K(\pi_n (\BO )/ H, n))\]
  and we let $\MO \langle n \rangle_H$ denote the associated Thom spectrum.
\end{dfn}

\begin{exm}
  In the extremal cases $H = \pi_n (\BO)$ and $ 0$, we have $\MO\langle n \rangle_{\pi_n (\BO)} \simeq \MO\langle n \rangle$ and $\MO\langle n \rangle_{0} \simeq \MO\langle n+1 \rangle$.
\end{exm}

\begin{ntn}
  Let $M$ denote a closed $(n-1)$-connected $2n$-manifold.
  We let $H(M) \subset \pi_n \BO$ denote the image of $\pi_n (M)$ under the map $M \to \BO$ classifying the stable normal bundle.
\end{ntn}

\begin{thm} \label{thm:im=ker}
  Suppose that $n \geq 3$. Let $M$ denote a closed $(n-1)$-connected $2n$-manifold.
  Then the inertia group $I(M)$ of $M$ is equal to the kernel of the natural map
  \[\Theta_{2n} \to \pi_{2n} (\MO \langle n \rangle_{H(M)}).\]
\end{thm}

We note that when $n = 4$, this is a special case of \cite[Theorem 5.2.40]{nagy}, which applies more generally to simply-connected $8$-manifolds with torsionfree homology concentrated in even dimensions.
When $n \equiv 0,4,6 \mod 8$, it has also appeared as \cite[Theorem 1.6]{nagyarxiv}.

The case $n \not\equiv 2 \mod 8$ of this theorem is readily deduced from Kreck's work on modified surgery \cite{kreck} and Nagy's proof of a special case of Crowley's $Q$-form conjecture \cite[Theorem H]{nagy} \cite[Theorem 1.2]{nagyarxiv}. The $n \equiv 2 \mod 8$ case does not come up because we only need to consider the cases $n = 4,8,9$.
First, we must recall some definitions.

%
%

\begin{definition}
  Let $B\rightarrow \mathrm{BSO}$ be a fibration. Given a smooth manifold $M$, a map $M\rightarrow B$ is said to be a \textit{normal $k$-smoothing} if it is $(k+1)$-connected and the composition $M\rightarrow B\rightarrow\mathrm{BSO}$ is homotopic to the classifying map of the stable normal bundle of $M$.
  The \emph{normal $k$-type} of $M$ is the unique fibration $B \to \BSO$ for which there is a normal $k$-smoothing $M \to B$ and the fiber of $B \to \mathrm{BSO}$ is connected and has vanishing homotopy groups in degrees $\geq k+1$. Equivalently this is the $(k+1)$th stage of the Moore--Postnikov decomposition of the map $M\to \mathrm{BSO}$ classifying the stable normal bundle.
\end{definition}

\begin{dfn}
  For $M$ a simply-connected $2n$-manifold with $n$ even, the \textit{Q-form} of a normal $(n-1)$-smoothing $f:M\rightarrow B$ is the triple $$Q_n(f)=\Big(H_n(M),\lambda_M,f_*:H_n(M)\rightarrow H_n(B)\Big),$$ where $\lambda_M : H_n (M) \times H_n (M) \to \Z$ is the intersection form of $M$.
\end{dfn}

We can now state the theorems of Kreck and Nagy that we will make use of.

\begin{thm}[{\cite[Theorem D]{kreck}}] \label{thm:kreck-odd}
  Let $n \geq 2$ be odd. Then two closed simply-connected $2n$-manifolds $M_0$ and $M_1$ with the same Euler characteristic and the same normal $(n-1)$-type $B \to \BSO$ are diffeomorphic if and only if they admit $B$-bordant normal $(n-1)$-smoothings $M_i \to B$.
\end{thm}

\begin{thm}[{\cite[Theorem H]{nagy} \cite[Theorem 1.2]{nagyarxiv}}] \label{thm:Q-form}
  Let $n \geq 3$ be even and $B\rightarrow \mathrm{BSO}$ a fibration with $\pi_1(B)=0$ and $H_n(B;\Z)$ torsion-free. Let $M_0$ and $M_1$ denote closed oriented $2n$-manifolds.
  Then $M_0$ and $M_1$ are diffeomorphic if and only if they admit $B$-bordant normal $(n-1)$-smooothings $M_i \to B$ such that the associated $Q$-forms are isomorphic.
\end{thm}

As mentioned, \Cref{thm:Q-form} above is a special case of Crowley's $Q$-form conjecture. The full conjecture, which is stated in \cite[Problem 11]{CrowleyProblems}, does not assume that $B$ is simply-connected and $H_n (B;\Z)$ is torsionfree.

We now deduce \Cref{thm:im=ker} from \Cref{thm:kreck-odd} and \Cref{thm:Q-form}.

\begin{proof}[Proof of \Cref{thm:im=ker}]
  Since we have already proven in \Cref{sec:exceptional} that the kernel of
  \[\Theta_{2n} \to \pi_{2n} (\MO \langle n \rangle)\] is generated by the image of $J$ (and Kervaire invariant one elements when $n=3,7$)  
  and hence $I(M)$ is trivial for all $n\neq 4,8,9$ by \Cref{kernelunit}, we may assume that $n \equiv 0,1,4 \mod 8$.

  Given $\Sigma \in \Theta_{2n}$, we must determine when $M$ is diffeomorphic to $M \# \Sigma$.
  The manifolds $M$ and $M \# \Sigma$ have the same normal $(n-1)$-type, which is given by $\BO\langle n \rangle_{H(M)} \to \BSO$. Moreover, they admit unique normal $(n-1)$-smoothings $M \to \BO\langle n \rangle_{H(M)}$ and $M \# \Sigma \to \BO\langle n \rangle_{H(M)}$. Indeed, to see this uniqueness we let $F$ denote the fiber of $\BO \langle n \rangle_{H(M)} \to \BO \langle n \rangle$ and note that by the infinite loop space structure of $\BO \langle n \rangle_{H(M)} \to \BO \langle n \rangle$ any two normal $(n-1)$-smoothings of $M$ differ by a map $M \to F$, which must be null for connectivity reasons. The same argument applies to $M \# \Sigma$.

  Since $M$ and $M \# \Sigma$ have the same $Q$-forms and Euler characteristics, it follows from \Cref{thm:kreck-odd} (when $n \equiv 1 \mod 8$) and \Cref{thm:Q-form} (when $n \equiv 0,4 \mod 8$) that $M$ and $M \# \Sigma$ are diffeomorphic if and only if they are $\BO \langle n \rangle_{H(M)}$-bordant.
  Here we have used the fact that
  \[H_n (\BO\langle n \rangle_{H(M)};\Z) \cong \pi_n (\BO\langle n \rangle_{H(M)}) \cong H(M) \subseteq \pi_{n} \BSO\]
  is torsionfree when $n \equiv 0,4 \mod 8$.
  Since connected sum corresponds to addition in the cobordism ring, this holds if and only if $S^{2n}$ and $\Sigma$ are $\BO\langle n \rangle_{H(M)}$-bordant.
  The theorem then follows from the Pontryagin--Thom construction, which identifies the $\BO \langle n \rangle_{H(M)}$-bordism ring with $\pi_* (\MO \langle n \rangle_{H(M)})$.
\end{proof}

We now move on to the proof of \Cref{thm:det-main} using \Cref{thm:im=ker}. We start with the simplest case $n=9$. First, we prove a useful lemma (cf. \cite[Theorem 4.11]{bhs}).

\begin{lem} \label{lem:coker-inj}
  The unit map $\pi_{2n} \Ss \to \pi_{2n} \MO \langle n+1 \rangle$ factors through an injective map
  \[\coker(J)_{2n} \hookrightarrow \pi_{2n} \MO \langle n+1 \rangle\]
  for all $n \neq 1,3,7$.
  In these exceptional cases, the map is not injective and the cokernel is generated by $\eta^2, \nu^2$ and $\sigma^2$, respectively.
\end{lem}

\begin{proof}
  From the bar spectral sequence for $\MO\langle n+1 \rangle$ (\Cref{prop:D2-split} and \Cref{prop:D2o-E2}), we see that there is an exact sequence
  \[\Z\{\iota_{n} ^2\} \oplus \pi_{2n} \Sigma^{-1} \mathrm{bo} \cong \pi_{2n} \Sigma^\infty \mathrm{O} \langle n \rangle \to \pi_{2n} \Ss \to \pi_{2n} \MO \langle n+1 \rangle.\]

  Let $j_{k} \in \pi_k \Ss$ denote a generator of the image of $J$ in degree $k$.
  Then the cyclic group $\pi_{2n} \Sigma ^{-1} \mathrm{bo}$ maps onto $j_{2n}$, while $\iota_{n} ^2$ maps to $j_{n}^2$.
  The result now follows from the fact that $j_{n}^2$ lies in the image of $J$ whenever $n \neq 1,3,7$ by \cite[Lemma 3]{Novikov}. 
  When $n=1,3,7$, we note that $j_1 = \eta$, $j_3 = \nu$ and $j_7 = \sigma$.
\end{proof}

\begin{cor} \label{cor:theta-inj}
  The natural map $\Theta_{2n} \to \pi_{2n} \MO \langle n + 1 \rangle$ is injective.
\end{cor}

\begin{proof}
  The natural map factors as
  \[\Theta_{2n} \hookrightarrow \coker(J)_{2n} \to \pi_{2n} \MO \langle n+1 \rangle. \]
  Now, the image of $\Theta_{2n}$ in $\coker(J)$ consists of those elements which are of Kervaire invariant zero.
  In particular, it does not contain the classes $\eta^2, \nu^2$ or $\sigma^2$.
  The result therefore follows from \Cref{lem:coker-inj}.
\end{proof}

\begin{thm} \label{thm:9-det}
  The kernel of the map $\Theta_{18}\rightarrow \pi_{18}(\mo\langle 9\rangle_{H(M)})$ is zero if $H(M)=0$ and $\Z/8\Z \cong \mathrm{bSpin}_{19} \subset \Theta_{18}\cong\Z/8\Z\oplus \Z/2\Z$ otherwise.
\end{thm}

\begin{proof}
  When $H(M) = 0$, $\MO\langle 9 \rangle_{H(M)} = \MO \langle 10 \rangle$, so the result follows from \Cref{cor:theta-inj}.
  When $H(M) \cong \pi_9 \BO \cong \Z/2\Z$, we have $\MO \langle 9 \rangle_{H(M)} = \MO \langle 9 \rangle$, and \Cref{thm:ker-main} identifies the kernel of $\pi_{18} \Ss \cong \Theta_{18} \to \pi_{18} \mo \langle 9 \rangle$ with the kernel of the map $\pi_{18} \Ss \to \pi_{18} \ko$.
  To conclude, it suffices to note that the results of \cite{spin} imply that the kernel of $\pi_{18} \Ss \to \pi_{18} \ko$ is equal to the kernel of $\pi_{18} \Ss \to \pi_{18} \MSpin$.
\end{proof}

Next we tackle the cases $n=4,8$. Note that $\pi_n{\BO}=\Z$ in both cases.

\begin{ntn}
  Given a nonnegative integer $d$, we set 
  \[\MO\langle 4n \rangle_d \coloneqq \MO \langle 4n \rangle_{d\Z}.\]
  When $n = 4$ and $8$, respectively, then we write $\MSpin_d$ and $\MStr_d$ instead
\end{ntn}

\begin{thm} \label{thm:spin-det}
  The kernel of the map $\Z/2\Z \cong \coker(J)_8 \cong \Theta_{8} \to \pi_{8} \MSpin_{d}$ is equal to $\Theta_8$ if $4 \nmid d$ and equal to $0$ otherwise.
\end{thm}

\begin{thm} \label{thm:str-det}
  The kernel of the map $\Z/2\Z \cong \coker(J)_{16} \cong \Theta_{16} \to \pi_{16} \MStr_{d}$ is equal to $\Theta_{16}$ if $4 \nmid d$ and equal to $0$ otherwise.
\end{thm}

\begin{rmk}
  If $k$ is odd, then the canonical maps $\MSpin_{kd} \to \MSpin_{d}$ and $\MStr_{kd} \to \MStr_{d}$ are $2$-local equivalences.
  As a consequence, it suffices to prove Theorems \ref{thm:spin-det} and \ref{thm:str-det} when $d$ is a power of $2$.
\end{rmk}

\begin{proof}[Proof of \Cref{thm:det-main} assuming Theorems \ref{thm:9-det}, \ref{thm:spin-det} and \ref{thm:str-det}]
  What remains are the cases $n = 4,8,9$. When $n=9$, this is precisely the statement of \Cref{thm:9-det}.
  To translate the statements of \Cref{thm:spin-det} and \Cref{thm:str-det} into the cases $n=4,8$  of \Cref{thm:det-main}, it suffices to note that if we let $x_{4i} \in \pi_{4i} \BSO \cong \Z$ denote a generator, then $p_1 (x_4) = \pm 2$ and $p_2 (x_8) = \pm 6$ by \cite[Theorem 3.8]{Levine}.
\end{proof}

To prove Theorems \ref{thm:spin-det} and \ref{thm:str-det}, we view $\MO \langle 4n \rangle_d$ as a relative Thom spectrum (cf. \cite[Definition 4.5]{ABGHR}) over $\MO \langle 4n+1 \rangle$ and analyze the kernel of the unit map using the resulting bar spectral sequence.
The following lemma is an immediate consequence of the main theorem of \cite{beard}.

\begin{lem} \label{lem:thom}
  The $\Einf$-ring spectrum $\MO\langle 4n \rangle$ may be realized as the Thom spectrum over $\MO\langle 4n+1 \rangle$ of a map
  \[\Sigma^{4n} \Z \to \bgl_1 (\MO \langle 4n+1 \rangle). \]
  Moreover, when $d > 0$, $\MO\langle 4n \rangle_{d}$ may be realized as the Thom spectrum of the composition
  \[\Sigma^{4n} \Z \xrightarrow{d} \Sigma^{4n} \Z \to \bgl_1 (\MO \langle 4n+1 \rangle).\]
\end{lem}
%

In the above lemma, the map $\Sigma^{4n} \Z \to \bgl_1 (\MO \langle 4n+1 \rangle)$ arises in the following way: the composite
\[\tau_{\geq 4n+1} \bo \to  \bo \xrightarrow{J} \bgl_1 (\Ss) \to \bgl_1 (\MO \langle 4n+1 \rangle)\]
is equipped with a canonical nullhomotopy, so that we obtain a map
\[\bo / \tau_{\geq 4n+1} \bo \to \bgl_1 (\MO \langle 4n+1 \rangle)\]
whose composite with $\Sigma ^{4n} \Z \simeq \Sigma^{4n} \pi_{4n} \bo \to \bo / \tau_{\geq 4n+1} \bo$
is the map appearing in the lemma.
In particular, on $\pi_{4n}$, this map sends $1$ to the image in $\pi_{4n-1} \MO \langle 4n+1 \rangle$ of a generator $j_{4n-1} \in \pi_{4n-1} \Ss$ of the image of $J$.

As a consequence, we may write $\MO \langle 4n \rangle_d$ as the geometric realization of an $\Einf$-ring $R\langle 4n \rangle ^d _\bullet$ with
\[R \langle 4n \rangle ^d _{k} \simeq \Sigma^\infty _+ K(\Z, 4n-1)^{\otimes k} \otimes \MO \langle 4n+1 \rangle.\]

\begin{ntn}
  Let $[d] : K(\Z, m) \to K(\Z, m)$ denote the map obtained by applying $\Omega^\infty$ to $d : \Sigma^{m} \Z \to \Sigma^{m} \Z$.
\end{ntn}

Given a positive integer $\ell$, there are comparison maps of simplicial $\Einf$-rings $R \langle 4n \rangle ^{\ell d} _\bullet \to R \langle 4n \rangle ^{d} _\bullet$ which, in degree $k$, is given by the map
\[(\Sigma^\infty _+ [\ell]^{\otimes k} \otimes 1) : \Sigma^\infty _+ K(\Z, 4n-1)^{\otimes k} \otimes \MO \langle 4n+1 \rangle \to \Sigma^\infty _+ K(\Z, 4n-1)^{\otimes k} \otimes \MO \langle 4n+1 \rangle.\]
%
%
%
%

Associated to the simplicial object $R \langle 4n \rangle ^d _\bullet$, we have a spectral sequence:

\[\pi_{n} \Sigma^\infty K(\Z, 4n-1)^{\otimes k} \otimes \MO \langle 4n+1 \rangle \Rightarrow \pi_{n+k} \MO \langle 4n \rangle_d.\]

In light of \Cref{cor:theta-inj}, we know that the map $\Theta_{8n} \to \pi_{8n} \MO \langle 4n+1 \rangle$ is injective.
As a consequence, any element in the kernel of $\Theta_{8n} \to \pi_{8n} \MO \langle 4n \rangle_d$ must correspond to a differential in the above spectral sequence.
The following lemma implies that the only possible nonzero differential which can enter $\pi_{8n} \MO \langle 4n +1 \rangle$ is $d_1$.

\begin{lem}
  For $k \geq 2$, we have
  \[\pi_{8n+1-k} \Sigma^\infty K(\Z, 4n-1)^{\otimes k} \otimes \MO \langle 4n+1 \rangle = 0.\]
\end{lem}

\begin{proof}
  When $k > 2$, this is zero for connectivity reasons.
  When $k=2$, we can use connectivity arguments to conclude that
  \begin{align*}
    \pi_{8n-1} \Sigma^\infty K(\Z, 4n-1) ^{\otimes 2} \otimes \MO \langle 4n+1 \rangle &\cong \pi_{8n-1} \Sigma^\infty K(\Z, 4n-1)^{\otimes 2}\\
    &\cong \pi_{8n-1} (\Sigma^{4n-1} \Z) ^{\otimes 2}\\
    &\cong 0. \qedhere
  \end{align*}
\end{proof}

As a consequence, we find that there is an exact sequence:
\[\pi_{8n} \Sigma^\infty K(\Z, 4n-1) \otimes \MO\langle 4n+1 \rangle \xrightarrow{f} \pi_{8n} \MO\langle 4n+1 \rangle \to \pi_{8n} \MO \langle 4n \rangle,\]
where $f$
is the map of nonunital $\Einf$-$\MO \langle 4n+1 \rangle$-algebras adjoint to the map $\Sigma^{4n} \Z \to \bgl_1(\MO \langle 4n+1 \rangle)$ of \Cref{lem:thom}.
More generally, we have an exact sequence:
\[\pi_{8n} \Sigma^\infty K(\Z, 4n-1) \otimes \MO\langle 4n+1 \rangle \xrightarrow{f \circ (\Sigma^\infty [d] \otimes 1)} \pi_{8n} \MO\langle 4n+1 \rangle \to \pi_{8n} \MO \langle 4n \rangle_d . \]
To prove Theorems \ref{thm:spin-det} and \ref{thm:str-det}, it therefore suffices to prove the following two lemmas:

\begin{lem} \label{lem:4-zero}
  The map $(\Sigma^\infty [4] \otimes 1) : \Sigma^\infty K(\Z, 4n-1) \otimes \MO \langle 4n+1 \rangle \to \Sigma^\infty K(\Z, 4n-1) \otimes \MO \langle 4n+1 \rangle$
  induces the zero map on $\pi_{8n}$.
\end{lem}

\begin{lem} \label{lem:2-kills}
  When $n=1$ or $2$, the image of the map $f \circ (\Sigma^\infty [2] \otimes 1) : \Sigma^\infty K(\Z, 4n-1) \otimes \MO \langle 4n+1 \rangle \to \MO \langle 4n+1 \rangle$ contains the image of the natural map $\Theta_{8n} \to \pi_{8n} \MO \langle 4n+1 \rangle$.
\end{lem}

First, we prove \Cref{lem:4-zero}. To begin with, we consider the following sequence of maps:

\[D_2 (\Sigma^\infty K(\Z,4n-1)) \xrightarrow{\hat{m}} \Sigma^\infty K(\Z, 4n-1) \xrightarrow{\pi} \Sigma^{4n-1} \Z,\]
where $\hat{m}$ comes from the $\Einf$-structure map and $\pi$ is the adjoint to the identity map of $K(\Z, 4n-1) = \Omega^{\infty} (\Sigma^{4n-1} \Z)$ under the $\Sigma^\infty$-$\Omega^\infty$ adjunction. Note that the composite $\pi \circ \hat{m}$ is nullhomotopic.

It follows from Goodwillie calculus arguments (cf. \cite[Section 4.2]{inertia}) that through degree $12n-4$, this sequence is a fiber sequence; moreover, the map $D_2 (\pi) : D_2 (\Sigma^\infty K(\Z, 4n-1)) \to D_2 (\Sigma^{4n-1} \Z)$ is an equivalence in this range.
In conclusion, after truncating to $\Sp_{\leq 12n-4}$ we have a fiber sequence
\[D_2 (\Sigma^{4n-1} \Z) \to \Sigma^\infty K(\Z, 4n-1) \to \Sigma^{4n-1} \Z.\]

\begin{ntn}
  In the rest of this section, we will work implicitly in the category $\Sp_{\leq 12n-4}$ of $(12n-4)$-coconnective spectra.
\end{ntn}

Moreover, for each nonnegative integer $d$ we have a commutative diagram
\begin{center}
  \begin{tikzcd}
    D_2 (\Sigma^{4n-1} \Z) \ar[r] \ar[d, "D_2 (d)"] & \Sigma^\infty K(\Z, 4n-1) \ar[r] \ar[d, "\Sigma^\infty {[d]}"] & \Sigma^{4n-1} \Z \ar[d,"d"] \\
    D_2 (\Sigma^{4n-1} \Z) \ar[r] & \Sigma^\infty K(\Z, 4n-1) \ar[r] & \Sigma^{4n-1} \Z.
  \end{tikzcd}
\end{center}

Our proof of \Cref{lem:4-zero} will be based on an $\F_2$-Adams filtration argument. The first step in the following lemma:

\begin{lem} \label{lem:2-AF}
  The map $\Sigma^\infty [2] : \Sigma^\infty K(\Z, m) \to \Sigma^\infty K(\Z, m)$ is of $\F_2$-Adams filtration $1$.
\end{lem}

\begin{proof}
  This follows from the following diagram:
  \begin{center}
    \begin{tikzcd}
      \H^* (K(\F_2,m); \F_2) \ar[r, "0"] \ar[d, two heads] & \H^* (K(\F_2, m);\F_2) \ar[d, two heads] \\
      \H^* (K(\Z,m); \F_2) \ar[r, "{[2]}"] & \H^* (K(\Z,m); \F_2).
    \end{tikzcd}
  \end{center}
  \vspace{-0.3cm}
\end{proof}

The second step is the following lemma:

\begin{lem} \label{lem:two-lines}
  The $\F_2$-Adams spectral sequence for $\Sigma^\infty K(\Z, 4n-1) \otimes \MO \langle 4n+1 \rangle$ is concentrated on two lines in total degree $8n$.
\end{lem}

Given this lemma, we can prove \Cref{lem:4-zero}.

\begin{proof}[Proof of \Cref{lem:4-zero}]
  Writing $\Sigma^\infty [4] \otimes 1 = (\Sigma^\infty [2] \otimes 1) \circ (\Sigma^\infty [2] \otimes 1)$, we see from \Cref{lem:2-AF} that, for any $x \in \pi_{8n} \Sigma^\infty K(\Z, 4n-1) \otimes \MO \langle 4n+1 \rangle$, the $\F_2$-Adams filtration of $(\Sigma^\infty [4] \otimes 1) (x)$ is strictly greater than that of $(\Sigma^\infty [2] \otimes 1) (x)$, and the latter is strictly greater than that of $x$. By \Cref{lem:two-lines}, this implies that $(\Sigma^\infty [4] \otimes 1) (x) = 0$, as desired.
\end{proof}

Before we can prove \Cref{lem:two-lines}, we must first prove two lemmas.

\begin{lem} \label{lem:D2-E2}
  The $\mathrm{E}_2$-page of the $\F_2$-Adams spectral sequence for $D_2 (\Sigma^{4n-1} \Z)$ in degrees $\leq 8n$ is as follows, drawn in Adams grading:
\begin{center}
    \begin{tikzcd}[every arrow/.append style={dash}, column sep=0.2cm, row sep=0.2cm]
      1 &        &    & h_1 \cdot Q_1 \iota_{4n-1} \\
      0 & \iota_{4n-1}^2 &    & Q_2 \iota_{4n-1} + \iota_{4n-1} (\zeta_1 ^2 \iota_{4n-1}) \\[-5pt]
     & 8n-2  & 8n-1   & 8n    
    \end{tikzcd}.
\end{center}
  In particular, $\pi_{8n} D_2 (\Sigma^{4n-1} \Z) \cong \F_2 \oplus \F_2$.
\end{lem}

\begin{proof}
  Recalling that $\H_* (\Z; \F_2) \cong \F_2 [\zeta_1 ^2, \zeta_2, \dots]$, we see from the Nishida relations that the $\F_2$-homology of $D_2 (\Sigma^{4n-1} \Z)$ in degrees $\leq 8n+1$ is as follows:
  \begin{center}
    \begin{tikzcd}[every arrow/.append style={dash}]
      8n+1 & Q_3 (\iota_{4n-1}) & & \iota_{4n-1} (\zeta_2 \iota_{4n-1}) \\
      8n   & Q_2 (\iota_{4n-1}) \ar[u] & \iota_{4n-1} (\zeta_1 ^2 \iota_{4n-1}) \ar[ur] & \\
      8n-1 & Q_1 (\iota_{4n-1}) & & \\
      8n-2 & Q_0 (\iota_{4n-1}) \ar[u] \ar[uu, bend left=60] \ar[uur, bend right = 30] & & 
    \end{tikzcd}
  \end{center}
  The computation of the $\mathrm{E}_2$-page of the $\F_2$-Adams spectral sequence in the desired range follows from standard methods.
  In particular, note that $Q_2 \iota_{4n-1} + \iota_{4n-1} (\zeta_1 ^2 \iota_{4n-1})$ appears because it is in the kernel of $\Sq_2$.
  This class is $h_0$-torsion because otherwise we would have $h_0 \cdot (Q_2 \iota_{4n-1} + \iota_{4n-1} (\zeta_1 ^2 \iota_{4n-1})) = h_1 \cdot Q_1 \iota_{4n-1}$, which would contradict the existence of a map of Steenrod modules $\H_{\leq 8n+1} (D_2(\Sigma^{4n-1} \Z); \F_2) \to \F_2 \{Q_1 (\iota_{4n-1})\}$ (such a map exists because there are no $\Sq_i$ hitting $Q_1 (\iota_{4n-1})$ in this range); this map takes $Q_2 \iota_{4n-1} + \iota_{4n-1} (\zeta_1 ^2 \iota_{4n-1})$ but not $h_1 \cdot Q_1 \iota_{4n-1}$ to zero.
  There are no Adams differentials for bidegree reasons.
\end{proof}

\begin{lem} \label{lem:homology-map}
  The map
  \[\F_2 \{Q_2 (\iota_{4n-1}), \iota_{4n-1} (\zeta_1 ^2 \iota_{4n-1})\} \cong \H_{8n} (D_2 (\Sigma^{4n-1} \Z);\F_2) \to \H_{8n} (\Sigma^\infty K(\Z,4n-1); \F_2)\]
  sends $Q_2 (\iota_{4n-1})$ to zero and is nonzero on $\iota_{4n-1} (\zeta_1 ^2 \iota_{4n-1})$.
\end{lem}

\begin{proof}
  The Dyer-Lashof operations on the homology of Eilenberg-MacLane spaces vanish because they can be represented by strictly commutative topological monoids.
  In particular, this implies that $Q_2 (\iota_{4n-1})$ is sent to zero.
  
  It therefore suffices to show that the map $\H_{8n} (D_2 (\Sigma^{4n-1} \Z);\F_2) \to \H_{8n} (\Sigma^\infty K(\Z,4n-1); \F_2)$ is nontrivial, which is equivalent to the map $\H_{8n} (\Sigma^\infty K(\Z,4n-1);\F_2) \to \H_{8n} (\Sigma^{4n-1} \Z; \F_2)$ having nontrivial kernel.
  This is equivalent to its dual $\H^{8n} (\Sigma^{4n-1} \Z; \F_2) \to \H^{8n} (\Sigma^\infty K(\Z,4n-1); \F_2)$ having nontrivial cokernel.
  A nontrivial class in the cokernel is given by $(\iota_{4n-1}) (Sq^2 \iota_{4n-1})$.
  This follows from Serre's computation of $H^* (K(\Z,n);\F_2)$ \cite[Th{\'e}or{\`e}me 3]{SerreCohEM} and the fact that the image of $\H^{*} (\Sigma^{n} \Z; \F_2) \to \H^{*} (\Sigma^\infty K(\Z,n); \F_2)$ consists precisely of those classes which may be written as compositions of squares applied to $\iota_n$.
  Equivalently, these are the $2^k$th powers of compositions of squares of excess strictly less than $n$ applied to $\iota_n$.
  The class given above is not of the form, hence is nontrivial in the cokernel.
\end{proof}

\begin{proof} [Proof of \Cref{lem:two-lines}]
  The fiber sequence
  \[D_2 (\Sigma^{4n-1} \Z) \to \Sigma^\infty K(\Z,4n-1) \to \Sigma^{4n-1} \Z\]
  gives rise to an exact sequence
  \[\pi_{8n} D_2 (\Sigma^{4n-1} \Z) \otimes \MO \langle 4n+1 \rangle \to \pi_{8n} \Sigma^\infty K(\Z, 4n-1) \otimes \MO \langle 4n+1 \rangle \to \pi_{8n} \Sigma^{4n-1} \Z \otimes \MO \langle 4n+1 \rangle.\]
  We have isomorphisms
  \[\pi_{8n} \Sigma^{4n-1} \Z \otimes \MO \langle 4n+1 \rangle \cong \H_{4n+1} (\BO \langle 4n+1 \rangle; \Z) \cong \begin{cases} \F_2 &\text{ if }n\text{ is even}\\ 0 &\text{ if }n\text{ is odd.} \end{cases}\]
    Moreover, if $n$ is even, then the nonzero class is detected in $\F_2$-Adams filtration $0$. Indeed, the $\F_2$-Adams spectral sequence for $\Z$-modules agrees with the $2$-Bockstein spectral sequence. It is a general fact for an $a$-Bockstein spectral sequence that any element which is not divisible by $a$ must be detected on the $0$-line of the spectral sequence. 
  It therefore suffices to show that the image of
  \[\pi_{8n} D_2 (\Sigma^{4n-1} \Z) \otimes \MO \langle 4n+1 \rangle \to \pi_{8n} \Sigma^\infty K(\Z, 4n-1) \otimes \MO \langle 4n+1 \rangle\]
  consists of elements detected in Adams filtrations $0$ and $k$ for some $k$.

  By a connectivity argument, the Adams spectral sequences for $D_2 (\Sigma^{4n-1} \Z) \otimes \MO \langle 4n+1 \rangle$ and $D_2 (\Sigma^{4n-1} \Z)$ agree through total degree $8n$.
  By \Cref{lem:D2-E2}, we have $\pi_{8n} D_2 (\Sigma^{4n-1} \Z) = \F_2 \oplus \F_2$ with elements detected in Adams filtrations $0$ and $1$.
  It therefore suffices to show that a class $x \in \pi_{8n} D_2 (\Sigma^{4n-1} \Z)$ detected in Adams filtration $0$ continues to be detected in Adams filtration $0$ in $\pi_{8n} \Sigma^\infty K(\Z, 4n-1) \otimes \MO \langle 4n+1 \rangle$.

  The detecting element of such a class $x$ is $Q_2 \iota_{4n-1} + (\iota_{4n-1})(\zeta_1 ^2 \iota_{4n-1}) \in \H_{8n} (D_2 (\Sigma^{4n-1} \Z);\F_2)$ by \Cref{lem:D2-E2}.
  By \Cref{lem:homology-map}, the image of this class in $\H_{8n} (\Sigma^\infty K(\Z,4n-1); \F_2)$ is nontrivial, which implies that the image of $x$ in $\pi_{8n} \Sigma^\infty K(\Z, 4n-1) \otimes \MO \langle 4n+1 \rangle$ continues to be detected in Adams filtration $0$, as desired.
\end{proof}


We now move on to the proof of \Cref{lem:2-kills}. We begin by recalling the first few homotopy groups of $D_2 (\Ss^{4n-1})$.

\begin{lem} \label{lem:E2-S}
  The $\mathrm{E}_2$-page of the $\F_2$-Adams spectral sequence for $D_2 (\Ss^{4n-1})$ in degrees $\leq 8n$ is as follows:
  \begin{center}
    \begin{tikzcd}[every arrow/.append style={dash}, column sep=0.2cm, row sep=0.2cm]
      1 &        &    & h_1 \cdot Q_1 \iota_{4n-1} \\
      0 & \iota_{4n-1}^2 &    &  \\[-5pt]
     & 8n-2  & 8n-1   & 8n    
    \end{tikzcd}.
  \end{center}
  Moreover, there is a Massey product $\langle h_1, h_0, \iota_{4n-1} ^2 \rangle = h_1 \cdot Q_1 \iota_{4n-1}$. \todo{More proof?}
\end{lem}

\begin{proof}
  By the Nishida relations, the $\F_2$-homology of $D_2 (\Ss^{4n-1})$ in degrees $\leq 8n+1$ is given as follows:
  \begin{center}
    \begin{tikzcd}[every arrow/.append style={dash}]
      8n+1 & Q_3 (\iota_{4n-1}) \\
      8n   & Q_2 (\iota_{4n-1}) \ar[u] \\
      8n-1 & Q_1 (\iota_{4n-1}) \\
      8n-2 & Q_0 (\iota_{4n-1}) \ar[u] \ar[uu, bend left = 60]
    \end{tikzcd}
  \end{center}
  The computation of the $\mathrm{E}_2$-page is straightforward from this.

  To compute the Massey product, we note that in the bar complex, $h_1$, $h_0$, and $\iota_{4n-1} ^2$ are represented by $[\zeta_1 ^2]$, $[\zeta_1]$, and $[\iota_{4n-1}^2]$.
  Then we have differentials
  $d([\zeta_2]) = [\zeta_1 ^2 \vert \zeta_1]$
  and
  $d([Q_1 \iota_{4n-1}]) = [\zeta_1 \vert \iota_{4n-1}^2]$
  in the cobar complex witnessing the relations $h_1 h_0 = 0$ and $h_0 \iota_{4n-1} ^2 = 0$, so that the Massey product
  $\langle h_1, h_0, \iota_{4n-1} ^2 \rangle$
  is represented in the cobar complex by
  $[\zeta_2 \vert \iota_{4n-1}^2] + [\zeta_1 ^2 \vert Q_1 (\iota_{4n-1})]$.
  This is a cocycle which maps to $[\zeta_1 ^2 \vert Q_1 (\iota_{4n-1})]$ under the map $\H_{\leq 8n+1} (D_2 (\Ss^{4n-1}) ;\F_2) \to \F_2\{Q_1 (\iota_{4n-1})\}$, hence is a representative for the class we denote $h_1 \cdot Q_1 \iota_{4n-1}$.
  Finally, it is simple to verify that this Massey product has no indeterminacy.
\end{proof}

We now establish precisely how, when $n = 1$ or $2$, the image of $\coker(J)_{8n}$ in $\pi_{8n} \MO \langle 4n+1 \rangle$ is killed in the bar spectral sequence converging to $\pi_{8n} \MO \langle 4n \rangle$.

\begin{lem} \label{lem:how-dies}
  Suppose that $n=1$ or $2$ and let $x \in \pi_{8n} D_2 (\Sigma^{4n-1} \Z)$ be the unique element detected by $h_1 \cdot Q_1 \iota_{4n-1}$ in the $\F_2$-Adams spectral sequence. Under the map
  \[\pi_{8n} D_2 (\Sigma^{4n-1} \Z) \to \pi_{8n} \Sigma^\infty K(\Z, 4n-1) \to \pi_{8n} \Sigma^\infty K(\Z, 4n-1) \otimes \MO \langle 4n+1 \rangle \xrightarrow{f} \pi_{8n} \MO \langle 4n+1 \rangle,\]
  $x$ is sent to the unique nonzero class in the image of the unit map
  \[\pi_{8n} \Ss \to \coker(J)_{8n} \cong \Z/2\Z \hookrightarrow \pi_{8n} \MO \langle 4n+1 \rangle.\]
\end{lem}

\begin{proof}
  By definition, the class $\iota_{4n-1} \in \pi_{4n-1} \Sigma^\infty K(\Z, 4n-1)$ is sent to the image of $j_{4n-1} \in \pi_{4n-1} \Ss$, a generator of the image of $J$, under the unit map $\pi_{4n-1} \Ss \to \pi_{4n-1} \MO \langle 4n+1 \rangle$.
  As a consequence, we obtain a commutative diagram:
  \begin{center}
    \begin{tikzcd}
      D_2 (\Ss^{4n-1}) \ar[rr, "D_2 (j_{4n-1})"] \ar[d] & & \Ss \ar[d] \\
      D_2 (\Sigma^{4n-1} \Z) \ar[r] & \Sigma^\infty K(\Z, 4n-1) \ar[r] & \MO \langle 4n+1 \rangle
    \end{tikzcd}
  \end{center}

  It follows from Lemmas \ref{lem:D2-E2} and \ref{lem:E2-S} that $x \in \pi_{8n} D_2 (\Sigma^{4n-1} \Z)$ lifts uniquely to a class $\wt{x} \in \pi_{8n} D_2 (\Ss^{4n-1})$, again represented in the $\F_2$-Adams spectral sequence by $h_1 \cdot Q_1 \iota_{4n-1}$.
When $n=1$, we have $j_3 = \nu$, so the class $D_2 (j_{3})(\wt{x}) = D_2 (\nu)(\wt{x}) \in \pi_{8} \Ss$ lies in the Toda bracket $\langle \eta, 2, \nu^2 \rangle$ by \Cref{lem:E2-S} and Moss's theorem \cite{Moss}.
  Note that the indeterminacy of the bracket $\langle \eta, 2, \nu^2 \rangle$ is generated by $\eta \sigma$, which lies in the image of $J$.
  In particular, the image of the bracket in $\langle \eta, 2, \nu^2 \rangle$ is well-defined in $\coker(J)_8$. It follows from \cite[pp. 189-190]{toda} that $\langle \eta, 2, \nu^2 \rangle$ generates $\coker(J)_8$, so the image of $D_2 (\nu (\wt{x}))$ in $\coker(J)_8$ is a generator as desired.

  When $n=2$, we have $j_7 = \sigma$.
  It follows that the class $D_2 (j_{7}) (\wt{x}) = D_2 (\sigma) (\wt{x}) \in \pi_{16} \Ss$ is detected on the $\mathrm{E}_2$-page of the $\F_2$-Adams spectral sequence by $h_1 Q_1 (h_3) = h_1 h_4$.
  Such a class generates $\coker(J)_{16}$ (see e.g. Isaksen--Wang--Xu's charts \cite{Isaksen}), so the result follows.
\end{proof}

\begin{lem} \label{lem:D2-2-fil-1}
  The image of the map $\pi_{8n} D_2 (2) : \pi_{8n} D_2 (\Sigma^{4n-1} \Z) \to \pi_{8n} D_2 (\Sigma^{4n-1} \Z)$ contains the unique nonzero element $x$ of $\F_2$-Adams filtration $1$, i.e the unique element detected by $h_1 \cdot Q_1\iota_{4n-1}$.
\end{lem}

\begin{proof}
  Since $2 : \Sigma^{4n-1} \Z \to \Sigma^{4n-1} \Z$ is of $\F_2$-Adams filtration $\geq 1$, so is $D_2 (2) : D_2 (\Sigma^{4n-1} \Z) \to D_2 (\Sigma^{4n-1} \Z)$.
  It therefore suffices to show that the map $\pi_{8n} D_2 (2)$ is nonzero.

  To see this, for a spectrum $X$ we write $D_2 (X)$ as $\mathrm{Nm}^{C_2} (X)_{h C_2}$, where $\mathrm{Nm}^{C_2}$ is the $C_2$-spectrum $X^{\otimes 2}$ with the swap action.
  Then we have $\mathrm{Nm}^{C_2} (2) = \mathrm{Nm}^{C_2} (1) + \mathrm{Nm}^{C_2} (1) + \mathrm{tr}^{C_2} (1 \cdot 1) = 2 + \mathrm{tr}^{C_2} (1)$ by \cite[Lemma A.36]{HHR}.
  Here, given a $C_2$-spectrum $Y$ and a nonequivariant map $f : Y \to Y$, $\mathrm{tr}^C_2 (f)$ is the $C_2$-equivariant map $Y \to Y$ obtained as the composite $Y \xrightarrow{tr(f)} Y \otimes (C_2)_+ \xrightarrow{c} Y$, where $tr(f)$ is adjoint to the map $f$ and $c$ is adjoint to the nonequivariant identity map on $Y$.
  (Here we are using $\mathrm{tr}^{C_2} (f)$ to refer to an internal transfer, while we use $tr(f)$ to refer to an external transfer.)

  Letting $X = \Sigma^{4n-1} \Z$ and applying $C_2$-homotopy orbits,
  we see that $D_2 (2) = 2 + \mathrm{tr}^{C_2} (1)_{h C_2}$, where $\mathrm{tr}^{C_2} (1)_{h C_2}$ is the composite $D_2 (\Sigma^{4n-1} \Z) \xrightarrow{tr} \pi_{8n} (\Sigma^{4n-1} \Z)^{\otimes 2} \xrightarrow{q} \pi_{8n} D_2 (\Sigma^{4n-1} \Z)$ of the transfer and quotient maps.
  Since $\pi_{8n} D_2 (\Sigma^{4n-1} \Z)$ consists of simple $2$-torsion by \Cref{lem:D2-E2},
  it suffices to show that the composite
  $\pi_{8n} D_2 (\Sigma^{4n-1} \Z) \xrightarrow{tr} \pi_{8n} (\Sigma^{4n-1} \Z)^{\otimes 2} \xrightarrow{q} \pi_{8n} D_2 (\Sigma^{4n-1} \Z)$
  is nonzero.
  Let $y \in \pi_{8n} D_2 (X)$ be of Adams filtration $0$, so that its image under the Hurewicz map is $Q_2 \iota_{4n+1} + \iota_{4n-1} (\zeta_1 ^2 \iota_{4n-1}) \in \H_{8n}(D_2 (\Sigma^{4n-1} \Z); \F_2)$.
  It follows that the Hurewicz image of $tr(y) \in \pi_{8n} \Sigma^{8n-2} \Z\otimes\Z$ is equal to $\iota_{4n-1} \otimes \zeta_1 ^2 \iota_{4n-1} + \zeta_1 ^2 \iota_{4n-1} \otimes \iota_{4n-1} \in \H_{8n}(\Sigma^{8n-2} \Z \otimes \Z; \F_2)$.
  In particular, $tr(y) \neq 0$.

  Now, the first three nonzero homotopy groups of $\Z \otimes \Z$ are given by $\pi_0 (\Z \otimes \Z) \cong \Z$, $\pi_1 (\Z \otimes \Z) = 0$, and $\pi_2 (\Z \otimes \Z) \cong \F_2$.
  In particular, there is a unique nonzero class in $\pi_{8n} (\Sigma^{8n-2} \Z \otimes \Z)$, which must therefore be equal to $tr(y)$. As a consequence, it suffices to show that the quotient map $\pi_{8n} (\Sigma^{8n-2} \Z \otimes \Z) \to \pi_{8n} D_2 (\Sigma^{4n-1} \Z)$ is nonzero, which is equivalent to showing that the corresponding homotopy orbit spectral sequence has no differential that enters degree $8n$. This spectral sequence takes the form
  \[\H_s (C_2; \pi_t ((\Sigma^{4n-1} \Z)^{\otimes 2}) \Rightarrow \pi_{s+t} D_2 (\Sigma^{4n-1} \Z)\]
  where $\pi_* ((\Sigma^{4n-1} \Z)^{\otimes 2})$ is endowed with the swap action.

  By the computation of the first three homotopy groups of $\Z \otimes \Z$, as well as the fact that the swap action on $\pi_{8n-2} ((\Sigma^{4n-1} \Z)^{\otimes 2}) \cong \Z$ is given by the sign action, we see that the $\mathrm{E}_2$-page in total degree $8n$ is given by
  \[\H_2 (C_2; \pi_{8n-2} ((\Sigma^{4n-1} \Z)^{\otimes 2})) \oplus \H_0 (C_2; \pi_{8n}((\Sigma^{4n-1} \Z)^{\otimes 2})) \cong \F_2 \oplus \F_2.\]
  Since this is isomorphic to the final answer by Lemma \ref{lem:D2-E2}, there can be no differentials entering total degree $8n$.
\end{proof}

Finally, we are able to deduce \Cref{lem:2-kills}.

\begin{proof}[Proof of \Cref{lem:2-kills}]
  This follows from combining \Cref{lem:how-dies} and \Cref{lem:D2-2-fil-1} with the commutative diagram
\begin{center}
  \begin{tikzcd}
    D_2 (\Sigma^{4n-1} \Z) \ar[r] \ar[d, "D_2 (2)"] & \Sigma^\infty K(\Z, 4n-1) \ar[r] \ar[d, "\Sigma^\infty {[2]}"] & \Sigma^{4n-1} \Z \ar[d,"2"] \\
    D_2 (\Sigma^{4n-1} \Z) \ar[r] & \Sigma^\infty K(\Z, 4n-1) \ar[r] & \Sigma^{4n-1} \Z.
  \end{tikzcd}
\end{center}
\end{proof}

\section{Homotopy and concordance inertia groups} \label{sec:h-c}
As an application, we determine the homotopy and concordance inertia groups of $(n-1)$-connected $2n$-manifolds when $n\geq 3$.
\begin{definition} 
The \textit{homotopy inertia group} $I_h(M)$ of a smooth $M$
is a subgroup of $I(M)$ consisting of $\Sigma$ such that the diffeomorphism $M\rightarrow M\# \Sigma$ is homotopic to the standard homeomorphism. 
\end{definition}
Recall that given a topological manifold $M$, two pairs $(N_1, f_1), (N_2,f_2)$ of a smooth manifold $N_i$ with a homeomorphism $f_i:N_i\rightarrow M$ are \textit{concordant} if there exists a diffeomorphism $g:N_1\rightarrow N_2$ such that there exists a homeomorphism $F:N_1\times[0,1]\rightarrow M\times[0,1]$ with $F|_{N_1\times\{0\}}=f_1$ and $F|_{N_1\times\{1\}}=f_2 \circ g$.
 \begin{definition}
The \textit{concordance
inertia group} $I_c(M)$ of $M$  is the subgroup of $I_h(M)$ consisting of homotopy spheres $\Sigma$ such that
$M$ and $M\#\Sigma$ are concordant (with respect to $M$) \cite{munkres}.
\end{definition}
By definition, $I_c(M)\subseteq I_h(M)\subseteq I(M)$.
Hence if $M$ is an $(n-1)$-connected $2n$-manifold with $n\geq 3$, it follows from \Cref{determination} that $I_c (M) = I_h (M) = 0$ when $n \neq 4,8,9$.

Some of these remaining cases are the subject of the following result of Kasilingam \cite[Theorem 2.7]{ramesh}.
\begin{thm}[{{\cite[Theorem 2.7]{ramesh}}}]
  Suppose that $M$ denotes an $(n-1)$-connected $2n$-manifold.
  If $n=4$, then we have $I_c (M) = I_h (M) = 0$.
  If $n=8$ and $\H_n (M) \cong \Z$, then we have $I_c (M) = 0$
\end{thm}

Our main result resolves the remaining cases to prove that, in contrast to the inertia group, the homotopy and concordance inertia groups of a highly connected manifolds always vanish, which appears as \Cref{thm:h-c} in the introduction.

\begin{theorem}
Let $M$ be an $(n-1)$-connected, smooth, closed, oriented $2n$-manifold with $n\geq 3$. Then $I_c(M)=I_h(M)=0$.
\end{theorem}

\begin{proof}
  It remains to check the cases $n=8,9$, and it follows from \Cref{determination} that $I(M)$ is $2$-torsion, so we are free to implicitly $2$-complete.
  It suffices to prove that $I_h (M) = 0$.
  By \cite{Wall62}, the manifold $M$ has homotopy type  $\left( \bigvee_{j\in I} S^{n}_j \right)\cup_g D^{2n}$, where $I$ is a finite set indexing the $n$-spheres in the $n$-skeleton of $M$ and $g:S^{2n-1}\rightarrow\bigvee_{j\in I} S^{n}_j$ is the attaching map of the top dimensional cell.
  We also let $q : M \to S^{2n}$ denote the map to the top cell.

  Let $G = \mathrm{GL}_1 (\Ss)$. By surgery theory (see \cite[Lemma 3.4]{Crowley}), we have the following diagram with exact rows:
\begin{center}
    \begin{tikzcd}
    0=L_{2n+1}(1)\ar[r]\ar[d]&\Theta_{2n}\ar[r,"\eta_{S^{2n}}"]\ar[d,"f_M"]&\left[S^{2n},G/\mathrm{O}\right]\ar[d,"q^*"] \\
    0=L_{2n+1}(1)\ar[r]& S^{\mathrm{Diff}}(M)\ar[r,"\eta_{M}"]&\left[M,G/\mathrm{O}\right],
    \end{tikzcd}
\end{center}
  where $S^{\mathrm{Diff}} (M)$ is the smooth structure set of $M$ and $f_M$ takes $\Sigma$ to $M \# \Sigma$ equipped with the homotopy equivalence $M \simeq M \# \Sigma$ coming from the standard homeomorphism $M \cong M \# \Sigma$.
  As a consequence of the definition of $f_M$, $I_h (M) = \mathrm{ker} (\Theta_{2n} \xrightarrow{f_M} S^{\mathrm{Diff}} (M))$.
  It therefore suffices to show that
  \[q^* : \left[S^{2n}, G/O\right] \to \left[M, G/O\right]\]
  is injective, or equivalently that
  \[(\Sigma g)^* : [\bigvee_{j \in I} S^{n+1}, G/O] \to \left[S^{2n}, G/O\right]\]
  is zero.
  
  Examining the diagram with exact rows
  \begin{center}
    \begin{tikzcd}
      \left[\bigvee_{j \in I} S^{n+1}, O\right] \ar[r] \ar[d, "(\Sigma g)^*"]& \left[\bigvee_{j \in I} S^{n+1}, G\right] \ar[r] \ar[d, "(\Sigma g)^*"]& \left[\bigvee_{j \in I} S^{n+1}, G/O\right] \ar[r] \ar[d, "(\Sigma g)^*"] & \left[\bigvee_{j \in I} S^{n}, O\right] \ar[d, "g^*"]\\
      \left[S^{2n}, O\right] \ar[r]& \left[S^{2n}, G\right] \ar[r]& \left[S^{2n}, G/O\right] \ar[r] & \left[S^{2n-1}, O\right],
    \end{tikzcd}
  \end{center}
  we see that it suffices to show that the top right horizontal arrow is zero and the second vertical arrow lands in the image of $J$.

  Since $O, G$ and $G/O$ are infinite loop spaces, this only depends on the stable homotopy class of $g$, which may be viewed as an element of $\bigoplus_{j \in I} \pi_{n-1} (\Ss)$.
  To show that the top right horizontal arrow is zero, it suffices to note that $J : \pi_n \mathrm{o} \to \pi_n \Ss$ is injective for $n=8,9$ by \cite[Theorems 1.1 \& 1.3]{AdamsJIV}.


  To show that the second vertical map lands in the image of $J$, it suffices to show that the $\pi_* \Ss$-module structure map $\pi_{n-1} \Ss \times \pi_{n+1} \mathrm{gl}_1 \Ss \to \pi_{2n} \mathrm{gl}_1 \Ss$ lands in the image of $J$.
  Since there is an equivalence of spectra $\tau_{[n+1,2n+1]} \mathrm{gl}_1 \Ss \simeq \tau_{[n+1, 2n+1]} \Ss$ (see \cite[Corollary 5.2.3]{AV}), this map may be identified with the product map $\pi_{n-1} \Ss \times \pi_{n+1} \Ss \to \pi_{2n} \Ss$.
  For this, we make use of Toda's tables \cite[pp. 189-190]{toda}.

  Reading off of Toda, we have:
  \[\pi_7 \Ss = \Z/16\Z\{\sigma\}, \pi_9 \Ss = \Z/2\Z\{\nu^3\} \oplus \Z/2\Z\{\mu\} \oplus \Z/2\Z\{\eta \varepsilon\}\]
  and
  \[\pi_8 \Ss = \Z/2\Z\{\overline{\nu}\}\oplus \Z/2\Z\{\epsilon\}, \pi_{10} \Ss = \Z/2\Z\{\eta \mu\}.\]
  By \cite[Theorem 14.1]{toda}, we have:
  \[\sigma \cdot \nu^3 = 0,\hspace{0.3cm} \sigma \cdot \mu = \eta \rho,\hspace{0.3cm} \sigma \cdot \eta\varepsilon = 0\]
  and
  \[\overline{\nu} \cdot \eta \mu = 0,\hspace{0.3cm} \varepsilon \cdot \eta \mu = \eta^3 \rho = 4\nu \rho = 0.\]
  It therefore suffices to note that $\eta \rho$ lies in the image of $J$.
\end{proof}

\section{Kernel of the boundary map} \label{sec:kernel}
In \Cref{thm:which-bdy}, we computed the image of the map
\[\partial : A_{2n+1} ^{\langle n \rangle} \to \Theta_{2n}.\]
The of goal of this section is to prove \Cref{thm:kernel-main-intro}, restated below as \Cref{thm:kernel-main}, in which we determine the map $\partial$ itself, up to multiplication by a $2$-adic unit.
\begin{thm} \label{thm:kernel-main}
  \begin{enumerate}
    \item[(i)] (Frank \cite[Theorem 3]{Frank}) The boundary of a $3$-connected $9$-manifold $M$ is the standard sphere if and only if $\Psi_{-L_{\mathbb{H}}} (M) \in \{1, [\nu_4 \circ \eta_{7}]\}$. Otherwise, the boundary is $[\epsilon] \in \coker(J)_8 \cong \Theta_8$.
    \item[(ii)] The boundary of a $7$-connected $17$-manifold $M$ is the standard sphere if and only if $\Psi_{L_{\mathbb{O}}} (M) \in \{1, [\sigma_8 \circ \eta_{15}]\}$. Otherwise, the boundary is $[\eta_4] \in \coker(J)_{16} \cong \Theta_{16}$.
    \item[(iii)] The boundary of an $8$-connected $19$-manifold is diffemorphic to $\omega(f) [h_2 h_4]$. Here, $\omega(f) \in \Z/8\Z$ is an invariant defined by Wall in \cite{Wall67} and $[h_2 h_4] \in \pi_{18} (\Ss) \cong \coker(J)_{18} \cong \Theta_{18}$ is a generator for the kernel of the unit map $\pi_{18} \Ss \to \pi_{18} \ko$, which is isomorphic to $\Z/8\Z$.\footnote{This statement should be interpreted as holding up to multiplication by a $2$-adic unit. Indeed, we have only specified $[h_2 h_4]$ up to multiplication by a $2$-adic unit, and Wall only defined $\omega (f)$ up to a $2$-adic unit.}
  \end{enumerate}
\end{thm}

The $n=9$ case is a consequence of what we have already proven. Indeed, in this case $A_{19} ^{\langle 9 \rangle} \cong \Z/8\Z$ via Wall's $\omega(f)$, so we know the map $\partial : A_{19} ^{\langle 9 \rangle} \to \Theta_{18}$ up to $2$-adic unit as soon as we know its image.
What remains is the cases $n=4$ and $n=8$: in these cases, we have $A_{2n+1} ^{\langle n \rangle} \cong \Z/2\Z \oplus \Z/2\Z$ and the image of $\partial$ is isomorphic to $\Z/2\Z$, so to determine $\partial$ we need to figure out which of the three nonzero elements of $A_{2n+1} ^{\langle n \rangle}$ lies in its kernel.

When $n = 4, 8$, we have chosen to phrase our results using Frank's description of $A_{2n+1} ^{\langle n \rangle}$ for $n \equiv 0 \mod 4$ in terms of his invariant $\Psi_g$ \cite{frank74wall}.
We have chosen to use this instead of \cite[Theorem 9]{Wall67} for two reasons.
First, Wall's invariant $\omega$ is not well-defined (it depends on a certain choice of subgroup $\pi_{n+1} \BSO(n)$.)
Second, Wall's theorem is incorrect for $n=4$, as pointed out in \cite[Section 8, (1)]{frank74wall}.
We begin by recalling Frank's work in \Cref{sec:frank}.

In \Cref{sec:n=4}, we consider the case $n=4$. In this case, \Cref{thm:kernel-main} was announced by Frank in \cite[Theorem 3]{Frank}.
However, to the best of the authors' knowledge, a full proof never appeared, so we include a complete proof of \Cref{thm:kernel-main}(i), following Frank's outline.
We conclude by proving the case $n=8$ in \Cref{sec:n=8}.

\subsection{Frank's classification}\label{sec:frank}
\begin{ntn}
  In the remainder of this section, we fix a positive integer $n \equiv 0 \mod 4$ and a generator $g \in \pi_{n} \BSO(n) \cong \pi_{n} \BSO$, which determines a preferred isomorphism $\pi_{n} \BSO(n) \cong \Z$.
  In a departure from \cite{frank74wall}, we emphasize the role that the choice of $g$ plays in the below by including it in our notation for $\Phi_g$ and $\Psi_g$.
\end{ntn}

\begin{exm}
  Let $L_{\mathbb{H}}$ and $L_{\mathbb{O}}$ denote the canonical quaternionic and octonionic line bundles on $\HP^1 \cong S^4$ and $\OP^1 \cong S^8$, respectively.
  The associated sphere bundles are the Hopf maps $\nu_4: S^7 \to S^4$ and $\sigma_8 : S^{15} \to S^8$.
  The bundles $L_{\mathbb{H}}$ and $L_{\mathbb{O}}$ determine elements of $\pi_4 \BSO(4)$ and $\pi_8 \BSO (8)$ that stabilize to generators of $\pi_4 \BSO$ and $\pi_8 \BSO$. (See e.g. \cite[Section 3.1]{BaezOct}.)
\end{exm}
\begin{cnstr}
  Given an $(n-1)$-connected almost closed $(2n+1)$-manifold $M$, let $x \in \H^n (M; \pi_n \BSO) \cong \H^n (M; \Z)$ denote the first obstruction to trivializing the stable normal bundle of $M$.
  Since the cohomology of $M$ is concentrated in degrees $n$ and $n+1$, obstruction theory implies that $x$, viewed as a map to an Eilenberg--MacLane space, factors as
  \[M \xrightarrow{f} S^n \xrightarrow{\iota_n} K(\Z,n).\]
  We let $\Phi_g (M) \subset \pi_{2n} (S^n)$ denote the collection of composites
  \[S^{2n} \hookrightarrow M \xrightarrow{f} S^n,\]
  where the first map is the inclusion of the boundary and $f$ varies over maps factoring $x$ as above.
\end{cnstr}
%
%
%

Frank determined the indeterminacy of the assignment $M \mapsto \Phi_g (M)$ and used this invariant to compute $A_{2n+1} ^{\langle n \rangle}$.

\begin{lem}[{\cite[Theorem 1.3]{frank74wall}}]
  Let $D_n \in \pi_{2n} (S^n)$ denote the element $\eta_n \circ J(g) + [e_n, \eta_n ]$, where $e_n \in \pi_n (S^n) \cong \Z$ denotes the identity map, $\eta_n \in \pi_{n+1} S^n$ is the Hopf invariant one class, and $[-,-]$ is the Whitehead product.
  Then the image of $\Phi_g (M)$ in $\pi_{2n} (S^n) / D_n$ consists of a single element, and $\Phi_g (M)$ is a $\langle D_n \rangle$-coset, where $\langle D_n \rangle$ is the subgroup generated by $D_n$.
\end{lem}

\begin{lem}[{\cite[Proposition 2.1 and Theorem 3.2]{frank74wall}}]
  The assignment $M \mapsto \Phi_g (M)$ descends to a homomorphism
  \[\Psi_g : A_{2n+1} ^{\langle n \rangle} \to \pi_{2n} (S^n) / D_n\]
  which factors through an isomorphism
  \[\Psi_g : A_{2n+1} ^{\langle n \rangle} \cong \im(J_n)_{2n} / D_n,\]
  where $J_n : \pi_{*-n+1} \BO(n) \to \pi_{*} (S^n)$ is the unstable $J$-homomorphism.
\end{lem}

\begin{exm}[{\cite[Propositions 4.1 and 4.2]{frank74wall}}]
  We have $\im (J_4)_8 / D_4 \cong \Z/2\Z \oplus \Z/2\Z$, $\im (J_{8k})_{16k} / D_{8k} \cong \Z/2\Z \oplus \Z/2\Z$ and $\im (J_{8+4k})_{16+8k} / D_{8+4k} \cong \Z/2\Z \oplus \Z/2\Z$ for $k \geq 1$.
  As a consequence, we recover the computation of $A_{2n+1} ^{\langle n \rangle}$ for $n \equiv 0 \mod 4$.
\end{exm}

We will need to examine the cases $n=4$ and $n=8$ in more detail.

\begin{ntn}
  We use Toda's notation for elements of the unstable homotopy groups of spheres, except that we write $\Sigma$ instead of $E$ for suspension.
  The elements $\eta_n$, $\nu_n$ and $\sigma_n$ are the usual complex, quaternionic, and octonionic Hopf invariant one classes, suspended to lie in the homotopy groups of $S^n$.
  The element $\nu' \in \pi_6 (S^3)$ is defined on \cite[p. 40]{toda} to be an element of the bracket $\langle \eta_3, 2 e_4, \eta_4 \rangle$,
  and the element $\sigma' \in \pi_{14} (S^7)$ is defined in \cite[Lemma 5.14]{toda}.
  The elements $\epsilon_n \in \pi_{n+8} (S^n)$ and $\overline{\nu} \in \pi_{n+8} (S^n)$ are treated in \cite[Section VI(i)]{toda} and \cite[Section VI(ii)]{toda}, respectively.
\end{ntn}

\begin{prop}
  For $n=4$, we have
  \[\im(J_4)_{8} / D_4 \cong \im(J_4)_8 \cong \Z/2\Z \{\nu_4 \circ \eta_7, \Sigma \nu' \circ \eta_7\}.\]
  For $n=8$, we have
  \[\im(J_8)_{16} / D_8 \cong \Z/2\Z \{\sigma_8 \circ \eta_{15}, \Sigma \sigma' \circ \eta_{15}\}.\]
\end{prop}

\begin{proof}
  It follows from \cite[Proposition 4.1]{frank74wall} that $\im(J_4)_8 \cong \Z/2\Z \oplus \Z/2\Z$ and $\im(J_8)_{16} \cong \Z/2\Z \oplus \Z/2\Z \oplus \Z/2\Z$.
  Furthermore, \cite[Proposition 4.2]{frank74wall} states that $D_4 = 0$ and $D_8 \neq 0$.

  For $n=4$, the claim follows directly from Toda's calculation (\cite[Proposition 5.8]{toda})
  \[\pi_8 (S^4) \cong \Z/2\Z \{\Sigma \nu' \circ \eta_7, \nu_4 \circ \eta_7\}.\]

  For $n=8$, Toda has computed in \cite[Theorem 7.1]{toda} that
  \[\pi_{16} (S^8) \cong \Z/2\Z \{\sigma_8 \circ \eta_{15}, \Sigma \sigma' \circ \eta_{15}, \overline{\nu}_8, \epsilon_8\}.\]
  Using the fact that the image of the stable $J$-homomorphism is generated by $\eta \sigma$ in this degree, we conclude that
  \[\im(J_8)_{16} \cong \Z/2\Z \{\sigma_8 \circ \eta_{15}, \Sigma \sigma' \circ \eta_{15}, \overline{\nu}_8 + \epsilon_8\}.\] 
  Finally, since $[e_8, \eta_8]$ is nonzero by \cite{Hilton}, it must generate the kernel of the suspension map $\Sigma : \pi_{16} (S^8) \to \pi_{17} (S^9)$ by \cite[Theorem 7.1]{toda}.
  It therefore follows from \cite[Proof of Theorem 7.1]{toda} that $[e_8, \eta_8] = \Sigma \sigma' \circ \eta_{15}$.

  On the other hand, it follows from \cite[(7.3)]{toda} that $\eta_8 \circ \sigma_9 = \Sigma \sigma' \circ \eta_{15} + \overline{\nu}_8 + \varepsilon_8$.
  We therefore have $D_8 = \overline{\nu}_8 + \varepsilon_8$, so that $\im(J_8)_{16}/D_8$ is as described.
\end{proof}

Finally, we review the computation of $\Psi_g$ for certain plumbings.

\begin{exm}[{\cite[Lemma 1.1]{frank74wall}}] \label{exm:frank-plumb}
  Let $b \in \pi_{n+1} \BSO(n)$ and $k \in \Z$, and let $M(kg,b)$ denote the plumbing of the disk bundles associated to $kg$ and $b$.
  Then
  \[\Psi_g (M(kg,b)) = [(k e_n) \circ J_n (b)],\]
  where $e_n \in \pi_n (S^n)$ is the identity map.
\end{exm}

\begin{rmk}
  Note that $(k e_n) \circ J_n (b) \neq k J_n (b)$ in general.
  Indeed, recall that there is a formula 
  \[(\beta_1 + \beta_2) \circ \alpha = \beta_1 \circ \alpha + \beta_2 \circ \alpha + [\beta_1, \beta_2] \circ H(\alpha),\]
  where $\beta_i \in \pi_k (S^n)$ and $\alpha \in \pi_p (S^k)$.
\end{rmk}

\subsection{The case $n=4$}\label{sec:n=4}

In this section, we prove \Cref{thm:kernel-main}(i) by filling in the gaps in \cite{Frank}.
We first prove the general \Cref{thm:Frank}, then apply it to our situation.

\begin{cnv}
  Given a space $X$ and a virtual vector bundle $V$ over $X$, we will denote by $\Th(X;V)$ the Thom spectrum of $V$ over $X$.
  We will implicitly consider all virtual vector bundles to be of dimension $0$ when taking the Thom spectrum.
\end{cnv}

\begin{ntn}
  Given a smooth manifold $M$, we let $N_M$ denote its stable normal bundle.
\end{ntn}

We begin with a simple lemma.

\begin{lem} \label{lem:bdy-PT-null}
  Let $(M,\partial M)$ denote an $n$-manifold with boundary.
  Then the composite $\Ss^{n-1} \to \Th(\partial M; N_{\partial M}) \to \Th(M; N_{M})$, where the first map is the Pontryagin-Thom collapse map, is nullhomotopic.
\end{lem}

\begin{proof}
  By \cite[Corollary 18.7.6]{MaySigurdsson}, the map $\Th (\partial M; N_{\partial M}) \to \Th(M; N_M)$ is Spanier-Whitehead dual to the map $\Sigma^\infty M / \partial M \to \Sigma^\infty _+ \Sigma \partial M$ coming from the cofiber sequence
  \[\Sigma^\infty_+ \partial M \to \Sigma^\infty_+ M \to \Sigma^\infty M / \partial M \to \Sigma^\infty _+ \Sigma \partial M\]
  Since the Pontryagin-Thom collapse map $\Ss^{n-1} \to \Th(\partial M; N_{\partial M})$ is Spanier-Whitehead dual to the projection $\Sigma^\infty_+ \Sigma \partial M \to \Ss^0$, it suffices to note that this projection factors through the map $\Sigma^\infty_+ \Sigma \partial M \to \Sigma^\infty_+ \Sigma M$.
\end{proof}

\begin{cnstr}
Given a sequence
\[S^p \xrightarrow{x} S^i \xrightarrow{\beta} \BSO,\]
we obtain a diagram
\begin{center}
  \begin{tikzcd}
    \Ss^0 \ar[r,"\mathrm{id}"] \ar[d] & \Ss^0 \ar[d] \\
    \Th(S^p; x^* \beta) \ar[r,"\Th(x)"] & \Th(S^i; \beta)
  \end{tikzcd}
\end{center}
by taking Thom spectra.
Taking cofibers vertically, we obtain a map
\[\Sigma^\infty_\beta (x) : \Ss^p \to \Ss^i.\]
\end{cnstr}

\begin{rmk}
  Given a map $S^i \xrightarrow{\beta} \BSO$, there is a natural identification of $\Th(S^i; \beta)$ with $C(J(\beta))$.
\end{rmk}

\begin{rec}
  Suppose that we are given a homotopy $(n-1)$-sphere $\Sigma$.
  Then $\Sigma$ is stably framed, and the choice of a stable framing determines a map $\Th(\Sigma; N_{\Sigma}) \to \Ss^0$, where the source is the Thom spectrum of the normal bundle of $\Sigma$.
  Let $\Ss^{n-1} \to \Th(\Sigma; N_\Sigma)$ denote the Pontryagin--Thom collapse map.
  Then the composition
  \[\Ss^{n-1} \to \Th(\Sigma; N_{\Sigma}) \to \Ss^0\]
  is a representative of $[\Sigma] \in \coker(J)_{n-1}$.
\end{rec}

\begin{thm} [{\cite[Theorem 2]{Frank}}] \label{thm:Frank}
  Let $M$ denote an oriented, almost closed $n$-manifold whose stable normal bundle $M \to \BSO$ fits into a diagram
  \begin{center}
    \begin{tikzcd}
      & S^p \ar[dr,"x"] & & \\
      S^{n-1} \ar[ur,"y"] \ar[r] & M \ar[r] & S^i \ar[r,"\beta"] & \BSO
    \end{tikzcd}
  \end{center}
  such that $\beta \circ x \simeq 0$. Here, $S^{n-1} \to M$ is the inclusion of the boundary.
  Then:
  \begin{enumerate}
    \item The Toda bracket $\langle J(\beta), \Sigma^\infty _\beta (x), \Sigma^\infty (y) \rangle$ is defined.
    \item $[\partial M] \in \pm \langle J(\beta), \Sigma^\infty _\beta (x), \Sigma^\infty (y) \rangle$ in $\coker (J)_{n-1}.$
  \end{enumerate}
\end{thm}

\begin{proof}
  Taking Thom spectra, we obtain the following diagram of spectra under $\Ss^0$:
  \begin{center}
    \begin{tikzcd}
      & \Th(S^p; x^* \beta) \ar[dr] & \\
      \Th(S^{n-1}; N_{S^{n-1}}) \ar[ur] \ar[r] & \Th(M; N_M) \ar[r] & \Th(S^i; \beta).
    \end{tikzcd}
  \end{center}
  Choosing a nullhomotopy $\beta \circ x \simeq 0$ and recalling the definition of $\Sigma^\infty _\beta$, we obtain the diagram
  \begin{center}
    \begin{tikzcd}
      & \Ss^p \oplus \Ss^0 \ar[dr] \ar[r, "\mathrm{pr}_0"] & \Ss^p \ar[dr,"\Sigma^\infty _\beta x"] & \\
      \Ss^{n-1} \oplus \Ss^0 \ar[ur,"\Sigma^\infty y \oplus id"] \ar[r] & \Th(M; N_M) \ar[r] & C(J(\beta)) \ar[r] & \Ss^i
    \end{tikzcd}.
  \end{center}

  Now, the Pontryagin--Thom collapse map $\Ss^{n-1} \to \Th(S^{n-1}; N_{S^{n-1}}) \simeq \Ss^{n-1} \oplus \Ss^0$ is of the form $\pm id \oplus [\partial M]$ by definition of $[\partial M]$. Here, we abusively use $[\partial M]$ to denote a lift of $[\partial M] \in \coker(J)_{n-1}$ to $\pi_{n-1} (\Ss^0)$. By \Cref{lem:bdy-PT-null}, it follows that we have a diagram:
  \begin{center}
    \begin{tikzcd}[column sep=large]
       & & \Ss^0 \ar[d] \\
      \Ss^{n-1} \ar[r, "\Sigma^\infty y"] \ar[urr, bend left = 30, "{\pm [\partial M]}"] & \Ss^p \ar[r] \ar[dr, bend right = 20, "\Sigma^\infty _\beta x"] & C(J(\beta)) \ar[d] \\
       & & \Ss^i.
    \end{tikzcd}
  \end{center}
  This diagram expresses the inclusion $[\partial M] \in \pm \langle J(\beta), \Sigma^\infty _\beta (x), \Sigma^\infty (y) \rangle$, and in particular implies that the Toda bracket is defined.
\end{proof}

We examine plumbings as examples of almost-closed manifolds that we might apply \Cref{thm:Frank} to.

\begin{exm}[{\cite[Lemma 1]{Frank}}]
  Suppose that $g_1 \in \pi_i (\BSO(n-i))$ and $g_2 \in \pi_{n-i} (\BSO(i))$, and let $M(g_1, g_2)$ denote the plumbing of the corresponding linear disk bundles over spheres.
  Then the composition
  \[S^i \vee S^{n-i} \hookrightarrow M(g_1, g_2) \to \BSO,\]
  where the second arrow classifies the stable normal bundle of $M(g_1, g_2)$,
  is given by the stabilization of $-g_1$ on $S^i$ and the stabilization of $-g_2$ on $S^{n-i}$.
  
  If we assume further that $g_2$ is stably trivial, then this map factors up to homotopy as
  \[M(g_1,g_2) \simeq S^i \vee S^{n-i} \to S^i \xrightarrow{-1} S^i \xrightarrow{g_1} \BSO,\]
  where the map $S^i \vee S^{n-i} \to S^i$ crushes $S^{n-i}$ to a point.

  If the induced composite $S^{n-1} \simeq \partial M(g_1, g_2) \to S^i \vee S^{n-i} \to S^i \xrightarrow{-1} S^i$ factors as in \Cref{thm:Frank}, then we are in a situation where we can apply the theorem.
  Moreover, this composite is a representative for $\Psi_{g_1} (M)$.

\end{exm}

\begin{exm}
  We will be particularly interested in the case where $g_1 \in \pi_4 \BSO(5) \cong \pi_4 \BSO$ is $-L_{\mathbb{H}} \oplus \RR$ and $g_2 \in \pi_5 \BSO (4)$ is $L_{\mathbb{H}} \circ \eta_4$.
  The bundle $g_2$ is stably trivial since $\pi_5 \BSO \cong 0$,
  so the classifying map of the stable normal bundle of $M(-L_{\mathbb{H}} \oplus \RR, L_{\mathbb{H}} \circ \eta_4)$ factors as
  \[M(-L_{\mathbb{H}} \oplus \RR, L_{\mathbb{H}} \circ \eta_4) \to S^4 \xrightarrow{-L_{\mathbb{H}}} \BSO.\]

  Now, the composite
  \[S^8 \cong \partial M(-L_{\mathbb{H}} \oplus \RR, L_{\mathbb{H}} \circ \eta_4) \hookrightarrow M(-L_{\mathbb{H}} \oplus \RR, L_{\mathbb{H}} \circ \eta_4) \to S^4\]
  must lie in $\Phi_{-L_{\mathbb{H}}} (M(-L_{\mathbb{H}} \oplus \RR, L_{\mathbb{H}} \circ \eta_4))$.
  It therefore follows from \Cref{exm:frank-plumb} and the fact that $D_4 = 0$ \cite[Proposition 4.2]{frank74wall} that this composite is homotopic to $J_4 (L_{\mathbb{H}} \circ \eta_4) = \nu_4 \circ \eta_7 \in \pi_8 (S^4)$.
\end{exm}

Applying \Cref{thm:Frank}, we obtain the following result.
\begin{prop} \label{prop:plumbing-toda}
  There is an inclusion
  \[[\partial M(-L_{\mathbb{H}} \oplus \RR, L_{\mathbb{H}} \circ \eta_4)] \in \pm \langle \nu, \Sigma^\infty_{-L_{\mathbb{H}}} (\nu_4), \eta \rangle. \]
\end{prop}
%
%
%

What remains is to compute $\Sigma^\infty_{-L_{\mathbb{H}}} (\nu_4)$.

\begin{lem} \label{lem:betaHopf}
  We have $\Sigma^\infty _{-L_{\mathbb{H}}} (\nu_4)=0$.
\end{lem}

\begin{proof}
  It follows from the identification of $\HP^2$ with the Thom space of $L_{\mathbb{H}}$ over $\HP^1$ that the normal bundle of
  $S^4 \cong \HP^1 \hookrightarrow \HP^2$
  is $L_{\mathbb{H}}$.
  From this, we obtain a homotopy commutative diagram
  \begin{center}
    \begin{tikzcd}
      S^7 \ar[r,"\nu_4"] & S^4 \ar[d] \ar[r,"-L_{\mathbb{H}}"] & \BSO \\
      & C(\nu_4) \simeq \mathbb{HP}^2, \ar[ur] &
    \end{tikzcd}
  \end{center}
  where the map $\mathbb{HP}^2 \to \BSO$ classifies the stable normal bundle.
  Taking Thom spectra and cofibering out by the map $\Ss^0 \simeq \Th(\ast; 0) \to \Th(X;V)$ which we have for any pointed $X$, we obtain a cofiber sequence
  \[\Ss^7 \xrightarrow{\Sigma^\infty_{-L_{\mathbb{H}}} (\nu_4)} \Ss^4 \to \Th(\HP^2; N_{\HP^2}) / \Ss^0.\]
  It therefore suffices to show that $\Th(\HP^2; N_{\HP^2}) / \Ss^0 \simeq \Ss^4 \oplus \Ss^8.$
  By Atiyah duality, we have
  \[\Th(\mathbb{HP}^2; N_{\HP^2}) \simeq \Sigma^8 D (\Sigma^\infty _+ \mathbb{HP}^2) \simeq \Sigma^\infty C(\nu_4) \oplus \Ss^{8},\]
  so the desired result follows after cofibering out by the bottom cell.
%
%
%
\end{proof}

\begin{proof}[Proof of \Cref{thm:kernel-main}(i)]
  Since we already know that $\partial : A_{9} ^{\langle 4 \rangle} \cong \Z/2\Z \oplus \Z/2\Z \to \Z/2\Z \cong \Theta_{8} \cong \coker(J)_8$ is surjective, it suffices to exhibit a $3$-connected almost closed $9$-manifold $M$ with $\Psi_{L_{\mathbb{H}}} (M) = [\nu_4 \circ \eta_7]$ and boundary diffeomorphic to the standard $8$-sphere.

  Combining \Cref{prop:plumbing-toda} and \Cref{lem:betaHopf}, we find that
  \[[\partial M(-L_{\mathbb{H}} \oplus \RR, L_{\mathbb{H}} \circ \eta_4)] \in \pm \langle \nu, 0, \eta \rangle. \]
  Clearly this Toda bracket contains zero.
  Examining Toda's tables \cite{toda} or the chart \cite[p. 8]{Isaksen}, we see that any element of $\pi_8 (\Ss^0)$ which is divisible by $\eta$ or $\nu$ is contained in the image of $J$.
  Thus $[\partial M(-L_{\mathbb{H}} \oplus \RR, L_{\mathbb{H}} \circ \eta_4)] = 0 \in \coker(J)_8 \cong \Theta_8$.
  Finally, we note that \Cref{exm:frank-plumb} implies that $\Psi_{-L_{\mathbb{H}}} (M(-L_{\mathbb{H}} \oplus \RR, L_{\mathbb{H}} \circ \eta_4)) = [\nu_4 \circ \eta_7]$, which completes the proof.
%
\end{proof}

%
%

\subsection{The case $n=8$}\label{sec:n=8}

In this section, we exhibit an explicit $7$-connected almost closed $17$-manifold that lies in the kernel of the boundary map $\partial : A_{17} ^{\langle 8 \rangle} \to \Theta_{16}$.
Our method is to start by removing a disk from $\OP^2$, then ``multiplying the resulting almost closed manifold by $\eta$'' to obtain a manifold in the kernel of $\partial$.
Finally, using some results of Stolz \cite[Sections 9 \& 10]{stolz}, we identify the resulting manifold with a certain plumbing and evaluate Frank's invariant $\Psi_{L_{\mathbb{O}}}$ on it.

%
%
%

\begin{lem} \label{lem:plumbing-eta}
  The plumbing
  $M(L_{\mathbb{O}} \oplus \mathbb{R}, L_{\mathbb{O}} \circ \eta_8 )$
  represents $\eta [\OP^2 - D^{16}]$ in $\pi_{17} A[8] \cong A_{17} ^{\langle 8 \rangle}$.
\end{lem}

\begin{proof}
  Let $\alpha \in \pi_2 \BO(15) \cong \Z/2\Z$ denote the generator.
  Then Stolz constructs an almost closed $7$-connected $17$-manifold
  $S^1 \times (\OP^2 - D^{16}) \cup_\alpha D^2 \times D^{17}$
  which represents $\eta [\OP^2 - D^{16}]$ in $\pi_{17} A[8]$. (See \cite[Lemma 10.4]{stolz})

  Now, the normal bundle of $S^8 \cong \OP^1 \hookrightarrow \OP^2$ is $L_{\mathbb{O}}$.
  Combining \cite[Lemma 9.5]{stolz} with \cite[Lemma 9.8]{stolz}, we are able to conclude that \todo{Do you need the self-intersection to be zero to apply Lemma 9.5?? I think no}
  $S^1 \times (\OP^2 - D^{16}) \cup_\alpha D^2 \times D^{17}$
  is diffeomorphic to
  $P(L_{\mathbb{O}} \oplus \mathbb{R}, L_{\mathbb{O}} \circ \eta_8).$
\end{proof}

\begin{cor} \label{cor:bdy-plumb}
  The boundary of
  $M(L_{\mathbb{O}} \oplus \mathbb{R}, L_{\mathbb{O}} \circ \eta_{8})$
  is diffeomorphic to the standard $16$-sphere $S^{16}$.
\end{cor}

\begin{proof}
  The map
  $\Theta_{16} \to \coker(J)_{16}$
  is an isomorphism, so it suffices to show that
  $[\partial M(L_{\mathbb{O}} \oplus \mathbb{R}, L_{\mathbb{O}} \circ \eta)] \in \coker(J)_{16}$
  is equal to zero.
  Now, we have a commutative diagram
  \begin{center}
    \begin{tikzcd}
      \pi_{17} A[8] \ar[d, "\cong"] \ar[r] & \pi_{16} (\Ss) \ar[d] \\
      A_{17} ^{\langle 8 \rangle} \ar[r] & \coker(J)_{16},
    \end{tikzcd}
  \end{center}
  where the top horizontal map arises from a map of spectra.
  It therefore follows that the top horizontal map, and hence the bottom horizontal map, commutes with multiplication by $\eta$.
  Combining this with \Cref{lem:plumbing-eta}, we find that
  \[[\partial M(L_{\mathbb{O}} \oplus \mathbb{R}, L_{\mathbb{O}} \circ \eta)] = \eta [\partial(\mathbb{OP}^2 - D^{16})] = \eta [S^{15}] = 0,\]
  as desired.
\end{proof}

\begin{proof}[Proof of \Cref{thm:kernel-main}(ii)]
  By \Cref{cor:bdy-plumb}, it suffices to prove that
  \[\sigma_8 \circ \eta_{15} \in \Psi_{L_{\mathbb{O}}} (M(L_{\mathbb{O}} \oplus \mathbb{R}, L_{\mathbb{O}} \circ \eta)).\]
  This is a consequence of \Cref{exm:frank-plumb} and the fact that $J_8 (L_{\mathbb{O}} \circ \eta_8) = \sigma_8 \circ \eta_{15}$.
%
%
\end{proof}

\bibliographystyle{alpha}
\bibliography{bibliography}

\end{document}